\documentclass[12pt,a4paper,twoside,reqno]{amsart}
\title[Vertex algebras from Hull--Strominger]{Vertex algebras from the Hull--Strominger system}
\author[L. \'Alvarez-C\'onsul]{Luis \'Alvarez-C\'onsul}
\address{Instituto de Ciencias Matem\'aticas (CSIC-UAM-UC3M-UCM)\\ Nicol\'as Cabrera 13--15, Cantoblanco\\ 28049 Madrid, Spain}
  \email{l.alvarez-consul@icmat.es}
\author[A. De Arriba de La Hera]{Andoni De Arriba de La Hera}
\address{Instituto de Ciencias Matem\'aticas (CSIC-UAM-UC3M-UCM)\\ Nicol\'as Cabrera 13--15, Cantoblanco\\ 28049 Madrid, Spain}
  \email{andoni.dearriba@icmat.es}
  \author[M. Garcia-Fernandez]{Mario Garcia-Fernandez}
\address{Instituto de Ciencias Matem\'aticas (CSIC-UAM-UC3M-UCM)\\ Nicol\'as Cabrera 13--15, Cantoblanco\\ 28049 Madrid, Spain}
\email{mario.garcia@icmat.es}

\thanks{
Partially supported by the Spanish Ministry of Science and Innovation, through the `Severo Ochoa Programme for Centres of Excellence in R\&D' (CEX2019-000904-S). The first author is partially supported by MICINN under grants PID2019-109339GB-C31 and PID2022-141387NB-C21. The second author is partially supported by MICINN under grant BES-2017-080578. The second and third authors are partially supported by MICINN under grants PID2019-109339GA-C32, PID2022-141387NB-C22 and CNS2022-135784.
}

\usepackage{hyperref,url}%\usepackage[backref]{hyperref}%
\usepackage{amssymb,latexsym}
\usepackage{amsmath,amsthm}
\usepackage[latin1]{inputenc}
\usepackage{enumerate}
\usepackage[all]{xy} \CompileMatrices
\SelectTips{cm}{12} % computer modern fonts of size 12 for arrowheads
\usepackage{tikz-cd}

\usepackage{verbatim} %this package has the ``comment'' environment
\usepackage{wrapfig}
\usepackage{graphicx}
\usepackage{slashed}
\usepackage{multicol}
\theoremstyle{plain}
\newtheorem{theorem}{Theorem}[section]
\newtheorem{lemma}[theorem]{Lemma}

\newtheorem{proposition}[theorem]{Proposition}

\newtheorem*{theorem*}{Theorem}
\theoremstyle{definition}
\newtheorem{definition}[theorem]{Definition}
\newtheorem{definition-theorem}[theorem]{Definition-Theorem}
\newtheorem{example}[theorem]{Example}
\newtheorem*{acknowledgements}{Acknowledgements}
\theoremstyle{remark}
\newtheorem{remark}[theorem]{Remark}

\numberwithin{equation}{section} 

\newcommand{\tr}{\operatorname{tr}}
\newcommand{\End}{\operatorname{End}}

\newcommand{\ad}{\operatorname{ad}}

\newcommand{\du}{d}%{\operatorname{d}}
\newcommand{\dbar}{\bar{\partial}}

\newcommand{\CC}{{\mathbb C}}

\newcommand{\RR}{{\mathbb R}}
\newcommand{\ZZ}{{\mathbb Z}}

\newcommand{\ch}{\operatorname{ch}}
\renewcommand{\(}{\left(}
\renewcommand{\)}{\right)}

\newcommand{\Ker}{\operatorname{ker}}
\newcommand{\hra}{\hookrightarrow}
\newcommand{\lhra}{\lhook\joinrel\longrightarrow}
\newcommand{\lto}{\longrightarrow}

\newcommand{\cL}{\mathcal{L}}

\newcommand{\cR}{\mathcal{R}}

\newcommand{\cH}{\mathcal{H}} %\newcommand{\cH}{\operatorname{Ham}(X)}

\newcommand{\U}{\operatorname{U}}
\newcommand{\SU}{\operatorname{SU}}

\newcommand{\C}{{\mathbb{C}}}
\newcommand{\R}{{\mathbb{R}}}
\newcommand{\Q}{{\mathbb{Q}}}

\newcommand{\Z}{{\mathbb{Z}}}
\newcommand{\N}{{\mathbb{N}}}

\newcommand{\la}{\langle}
\newcommand{\ra}{\rangle}

\allowdisplaybreaks% Permite que el las fórmulas del entorno align ocupen varias pagans.

\begin{document}

\begin{abstract}
Motivated by the programme on mirror symmetry for non-K\"ahler manifolds, we construct representations of the $N=2$ superconformal vertex algebra associated to solutions of the Hull--Strominger system. The construction is via embeddings of the $N=2$ superconformal vertex algebra in the chiral de Rham complex of a string Courant algebroid. Our results require that the connection $\nabla$, one of the unknowns of the system, is Hermitian--Yang--Mills. Our main theorem proves that any solution of the Hull--Strominger system satisfying this condition has an associated $N=2$ embedding. %We also prove a general existence result for these embeddings on compact complex surfaces.
\end{abstract}

\maketitle

\setlength{\parskip}{5pt}
\setlength{\parindent}{0pt}

%\tableofcontents

%%%%%%%%%%%%%%%%%%%%%%%%%%
\section{Introduction}\label{sec:intro}
%%%%%%%%%%%%%%%%%%%%%%%%%%

Vertex algebras provide a fruitful bridge between mathematics and physics. On the physical side, they yield a rigorous definition of the chiral part of a $2$-dimensional conformal field theory, and have become important tools in string theory. 
On the mathematical side, they were formulated by Borcherds to prove the Monstrous Moonshine Conjecture~\cite{Borcherds2} and, since then, have played an important role in many areas, such as the representation theory of Kac--Moody algebras. 
In particular, vertex algebras have had a strong impact on geometry, by the constructions of the chiral de Rham complex~\cite{MSV} and the chiral differential operators %of Beilinson and Drinfeld \cite{BD} and Gorbounov, Malikov and Shechtman \cite{GSV}
\cite{BD,GSV}, their intimate connections with elliptic genera and mirror symmetry \cite{BorLib}, and also by their applications to the geometric Langlands correspondence~\cite{BD-Langlands} and gauge theory \cite{Joyce}.

The chiral de Rham complex $\Omega_M^{\ch}$ is a sheaf of vertex algebras on any smooth algebraic, complex or differentiable manifold $M$, introduced by Malikov--Schechtmann--Vaintrob~\cite{MSV}, that is closely related to a supersymmetric $\sigma$-model with target $M$~\cite{Kapustin,Witten02}. 
Its original construction was carried out locally, gluing the so-called $bc$-$\beta\gamma$ system over coordinate patches, by showing that the transformation properties of this system are compatible with appropriate coordinate changes.
% Its name was adopted because it carries a vertex-algebra derivation and a grading by conformal weight that is preserved by the derivation, such that the usual de Rham complex is isomorphic to the weight-zero component and quasi-isomorphic to the full chiral de Rham complex.
%
% Based on coordinate-free constructions by Beilinson--Drinfeld~\cite{BD} and Gorbounov--Malikov--Shechtman--\cite{GSV},
In subsequent work, Bressler~\cite{Bressler} and Heluani~\cite{Heluani09} obtained a more general coordinate-free construction of a sheaf of vertex algebras that can be attached functorially to any Courant algebroid. 
% The construction is similar to the construction of the universal enveloping algebra of a Lie algebroid.
The chiral de Rham complex $\Omega_M^{\ch}$ of \cite{MSV} corresponds to the standard Courant algebroid.

Since the pioneering work \cite{MSV} for Calabi--Yau manifolds, it has been a fundamental problem to construct embeddings of various superconformal vertex algebras in the vertex algebra of global sections of $\Omega_M^{\ch}$ for manifolds $M$ with special holonomy (see \cite{HeluaniRev} for a recent review). 
These embeddings resemble similar structures found in physics for non-linear $\sigma$-models.

The problem studied in this paper consists of finding appropriate geometric conditions under which the vertex algebra of global sections of the chiral de Rham complex $\Omega^{\ch}_E$ associated to a Courant algebroid $E$ over a differentiable manifold admits an embedding of the $N=2$ superconformal vertex algebra. 
The Courant algebroids that we consider are the so-called string Courant algebroids~\cite{grt2}, and the geometric conditions are inspired by the Killing spinor equations in heterotic supergravity. 

%\cite{BZHS,EHKZ,Heluani09,GCYHeluani,StructuresHeluani,Rodriguez} (see \cite{HeluaniRev} for a recent review).
%This includes related $N=2$ superconformal vertex algebra embeddings for K\"ahler manifolds \cite{SCDR2}, two commuting $N=2$ superconformal vertex algebra embeddings for generalized Calabi--Yau metric manifolds \cite{SCGCY}, two commuting embeddings of the Odake algebra (an extension of the $N=2$ superconformal vertex algebra) for Calabi--Yau threefolds \cite{EKHZ}, two commuting embeddings of the Shatashvili-Vafa superconformal algebra for holonomy $G_2$ manifolds \cite{G2CDR}, etc. See \cite{HeluaniRev} for more information. 
%Complete descriptions of the space of global sections of the chiral de Rham complex have been obtained in the special case of a $K3$ surface \cite{Song1,Song2}, and a compact Ricci-flat K\"ahler manifold \cite{Song3}.

Embeddings of the $N=2$ superconformal algebra in $\Omega^{\ch}_M$ have been constructed for (generalized) K\"ahler Calabi--Yau manifolds \cite{BZHS,EHKZ,Heluani09,GCYHeluani,StructuresHeluani,LinshawSong}, motivated by two-dimensional nonlinear sigma models with $(2,2)$ supersymmetry \cite{GHR}. However, until now nothing was known about potential embeddings in the chiral de Rham complex $\Omega^{\ch}_E$ for the more general $(0,2)$-models \cite{Witten02}. Results in this direction, for the $N=1$ superconformal algebra, can be found in the physics literature in recent work of de la Ossa, Fiset and Galdeano \cite{OssaF,Galdeano}. 
In the present work we give a precise answer to this problem, providing, in full generality, embedded $N=2$ superconformal algebras in the chiral de Rham complex of a string Courant algebroid over a complex manifold \cite{grt2,grt3}, carrying solutions to the Killing spinor equations \cite{GF3,grt}. In this setup, these equations are essentially equivalent to the Hull--Strominger system,  which describes supersymmetric compactifications in heterotic supergravity \cite{HullTurin,Strom}. % and are closely tied to complex non-K\"ahler geometry. 

One motivation for our embeddings is the construction of a new elliptic genus on complex non-K\"ahler manifolds, that we plan to address in future work. 
% Our embeddings in $\Omega^{\ch}_E$ are motivated by the construction of a new elliptic genus on complex non-K\"ahler manifolds, that we plan to address in future work. 
The Killing spinor equations that we consider avoid well-known no-go theorems that apply to similar generalized geometries with torsion on compact manifolds that prevented previous constructions of new elliptic genera (see~\cite[Remark 13]{StructuresHeluani}). It is therefore reasonable to expect that our results yield a new definition of elliptic genus, upon BRST reduction, on compact complex non-K\"ahler manifolds carrying solutions to the Killing spinor equations (see~\cite{Israel} for recent progress in physics). 

Our constructions are also motivated by a conjectural extension of mirror symmetry to complex non-K\"ahler geometry.
So far, the mathematical study of mirror symmetry was essentially bound to K\"ahler manifolds, in relation to type IIA and type IIB string theories \cite{COGP,KontsevichHMS,SYZ}. Motivated by physical developments in \cite{AGM,GMS}, several authors have pursued various strategies to extend this picture, based on type II string theories, to the realm of non-K\"ahler Calabi--Yau manifolds \cite{AldiHel,Bedulli-Vannini,LTY,Popovici,Ward}. However, there is another type of string theory, called heterotic, that is getting more attention from complex geometers %from the mathematical community 
since 20 years ago \cite{CoPiYau,FIUV,Fino,FuYau,grst,PPZ}, and a version of mirror symmetry for this theory, called $(0,2)$ mirror symmetry, is expected to exist by the mathematical physics community \cite{ABS,AAGa,Borisov02,Donagi,McOrist,MelPle,MelSeShe}.

A general way to approach mirror symmetry, as formulated geometrically by Borisov \cite{Borisov}, is via vertex algebras. %following more closely the physics approach to mirror symmetry. 
%As proved by Borisov and Kauffman \cite{Borisov02}, and more recently in \cite{AAGa}, this is actually well suited for understanding certain aspects of $(0,2)$ mirror symmetry. 
As proved by \cite{AAGa,Borisov02}, this is actually well suited for understanding certain aspects of $(0,2)$ mirror symmetry \cite{Witten02}. The basic idea is to construct representations of the $N=2$ superconformal vertex algebra, associated to mirror spaces, and to relate them via an automorphism, called the mirror involution. An important step to find such structures is the construction of embeddings of the $N=2$  superconformal vertex algebra in the chiral de Rham complex, as mentioned earlier. 
Our main results and methods to obtain these embeddings greatly expand those developed in the homogeneous setup in \cite{AAGa}, where first examples of $(0,2)$ mirror symmetry on compact complex non-K\"ahler manifolds were found. %In fact, we believe that the present work takes an important step in Borisov's vertex algebra approach to $(0,2)$ mirror symmetry on non-K\"ahler manifolds.

\subsection*{Outline and main results}

Let $(M,J)$ be a complex manifold of dimension $n$. We assume that the canonical bundle $K_M$ of $J$ is topologically trivial. An $\SU(n)$-structure on $(M,J)$ is a pair $(\Psi,\omega)$ consisting of a Hermitian form $\omega$, with metric $g = \omega(J,)$, and a smooth section $\Psi$ of the canonical bundle satisfying  
\begin{equation*}\label{eq:SUn}
(-1)^{\frac{n(n-1)}{2}}i^n\Psi \wedge \overline{\Psi} =  \frac{\omega^n}{n!}.
\end{equation*}
Let $K$ be a compact Lie group with Lie algebra $\mathfrak{k}$. Fix a non-degenerate bi-invariant symmetric bilinear form $\langle,\rangle: \mathfrak{k} \otimes \mathfrak{k} \to \mathbb{R}$. Let $P$ be a principal $K$-bundle over $M$ such that
\begin{equation*}\label{eq:BC22}
p_1(P) = 0 \in H^4_{dR}(M,\RR),
\end{equation*}
where $p_1$ is the characteristic class associated to $\langle,\rangle$ via Chern--Weil theory. More explicitly, given a principal connection $A$ on $P$, $p_1(P)$ is represented by the four-form $\langle F_A \wedge F_A \rangle \in \Omega^4$. 

In this setup, we say that a triple $(\Psi,\omega,A)$ is a solution of the \emph{twisted Hull--Strominger system} if the following equations hold:
\begin{equation}\label{eq:twistedStromintro}
\begin{split}
	F_A^{0,2} = 0, \qquad F_A \wedge \omega^{n-1} & = 0,\\
	d \Psi - \theta_\omega \wedge \Psi & = 0,\\
	d \theta_\omega & = 0,\\
	dd^c \omega + \left\langle F_A \wedge F_A\right\rangle  & = 0.
	\end{split}
\end{equation}
The previous equations were introduced in \cite{grst} as a natural generalization of the Hull--Strominger system \cite{HullTurin,Strom}, motivated by generalized geometry \cite{Hit1,GuThesis}. In fact, when the Lee form $\theta_\omega = J d^*\omega$ is exact, the second and third equations imply the existence of a holomorphic volume form $\Omega$ on $(M,J)$, %such that $d(\|\Omega\|_\omega \omega^{n-1}) = 0$, 
and one recovers the Hull--Strominger system with the \emph{Hermitian-Yang--Mills ansatz} for the connection $\nabla$ when $n=3$ (see \cite{GF4} for a lengthy discussion about this condition).

The relationship with generalized geometry, which we explain in detail in Section \ref{sec:THSFterm}, comes from the observation that a solution of \eqref{eq:twistedStromintro} has an associated smooth Courant algebroid $E$ of string type \cite{GF3,grt}, endowed with a %divergence operator (depending only on the Riemannian volume form and $\theta_\omega$) and a 
generalized metric 
$$
V_+ \subset E
$$
solving the \emph{Killing spinor equations} \cite{GF3,grt} (see Proposition \ref{prop:KSEqevendimCA}). Furthermore, upon complexification, we obtain a pair of isotropic involutive subbundles $\ell$ and $\overline{\ell}$ such that
$$
\ell \oplus \overline{\ell} = V_+ \otimes \CC.
$$
Our first technical results, in Lemma \ref{lem:holconstdetatlas}, Lemma~\ref{lem:frameprop} and Proposition \ref{prop:DtermTHS}, show that one can construct an adapted holomorphic atlas and special frames $\{\epsilon_j\}$ of $\ell$ and $\{\overline{\epsilon}_j\}$ of $\overline{\ell}$ such that
$$
\frac{1}{2}\sum_{j=1}^n\left[\overline{\epsilon}_j,\epsilon_j\right] = \varepsilon_{\ell}  - \varepsilon
$$
for an infinitesimal symmetry $\varepsilon \in \Omega^0(E\otimes \CC)$ of $E \otimes \CC$, that is, satisfying
$$
[\varepsilon,] = 0.
$$
Here, $\varepsilon_{\ell}$ denotes the projection of $\varepsilon$ onto $\ell$ associated to the above direct-sum decomposition. Using this, and a complicated sequence of vertex algebra calculations contained in the appendices, we obtain our main result (see Theorem \ref{teo:mainTCA}). We denote $e^j = \Pi \epsilon_j$ and $e_j = \Pi \overline{\epsilon}_j$, where $\Pi \colon E \otimes \CC \to \Pi (E \otimes \CC)$ is the parity-reversing operator.

\begin{theorem}\label{teo:mainTCAintro}
Let $(\Psi,\omega,A)$ be a solution of the twisted Hull--Strominger system \eqref{eq:twistedStromintro}. Consider the associated string Courant algebroid $E$ and the chiral de Rham complex $\Omega^{\ch}_{E \otimes \C}$. Then the following expressions define global sections of $\Omega^{\ch}_{E \otimes \C}$:
\begin{equation}\label{eq:JHglocalintro}
\begin{split}
J & := \frac{i}{2}\sum_{j=1}^n:e^je_j: + Si\Pi \theta_\omega,\\
H & := \frac{1}{2}\sum_{j=1}^n\left(:e_j\left(Se^j\right):+:e^j\left(Se_j\right):\right) + \frac{1}{4}\sum_{j,k=1}^n \Big{(}:e_j:e^k[e^j,e_k]:: \\
& + :e^j:e_k[e_j,e^k]:: - :e_j:e_k[e^j,e^k]:: - :e^j:e^k [e_j,e_k]::\Big{)} - T\Pi(\theta_\omega)_\ell.
\end{split}
\end{equation}
Furthermore, the sections $J$ and $H$ induce an embedding of the $N=2$ superconformal vertex algebra with central charge $c = 3n$ into the space of global sections of $\Omega^{\ch}_{E\otimes\C}$.
\end{theorem}

In particular, Theorem \ref{teo:mainTCAintro} proves that, on an arbitrary Calabi-Yau threefold, any solution of the Hull--Strominger system \cite{HullTurin,Strom} with connection $\nabla$ being Hermitian-Yang--Mills has an associated embedding of the $N=2$ superconformal vertex algebra. There is actually only a handful of these examples, contained in \cite{AGF1,CoPiYau,FeiYau,FIUV,GF4,GFGM,GM,OU,OUVi}. We shall discuss concrete non-K\"ahler examples in Section \ref{sec:geoexam}. Of particular interest is the existence of these embeddings in the algebraic case, as considered in \cite{AGF1,CoPiYau}, which we believe is an important step in the programme on $(0,2)$-mirror symmetry via vertex algebras \cite{AAGa,Borisov02}.

The proof of Theorem \ref{teo:mainTCAintro} is outlined in Section \ref{Resultados1}, while the technically hard calculations which lead to the proof are contained in Appendix \ref{app:MRP}. Our formulae for the generators of supersymmetry \eqref{eq:JHglocalintro} is inspired by work on generalized Calabi--Yau manifolds by Heluani and Zabzine \cite{StructuresHeluani}. Appendix \ref{app:MRP} actually provides alternative and more direct proofs of some of their results. 
%It should be emphasized that Theorem \ref{teo:mainTCAintro} applies, in particular, to solutions of the Hull--Strominger system as originally defined in \cite{HullTurin,Strom}, provided that the connection $\nabla$ in the tangent bundle is Hermitian-Yang--Mills (see Section \ref{sec:THSFterm}). 

Using a similar method of proof, in Theorem \ref{teo:mainTCAcHE} we construct a different embedding which allows for more general topologies. In this case, we consider the \emph{coupled Hermitian-Einstein equation} (see Definition \ref{def:cHE})
$$
F_ \mathbf{G} \wedge \omega^{n-1} = 0 
$$
on a holomorphic string algebroid over a complex manifold $(M,J)$ endowed with a holomorphic volume form $\Omega$. These equations have been studied very recently by the third author jointly with Gonzalez Molina \cite{GFGM0}, inspired by \cite{OssaLarforsSvanes,GaJoSt}. They admit solutions on manifolds which are not balanced, and also bundles with $\Lambda_\omega F_A \neq 0$ \cite{GM}. For instance, the central charge of this class of embeddings is given by 
$$
c=3n -\frac{3}{2} \langle \Lambda_\omega F_A ,\Lambda_\omega F_A \rangle
\in \R.
$$
For solutions of the Hull--Strominger system, both Theorem \ref{teo:mainTCAintro} and Theorem \ref{teo:mainTCAcHE} apply and give the same embedding.

In Section \ref{sec:geoexam} we study several examples where our main results apply. In Section \ref{sec:exdomain} we study a local situation in complex dimension $2$, and relate our embeddings with the \emph{heterotic string solitons} constructed by Callan, Harvey, and Strominger \cite{CHS}. In Section \ref{sec:surfaces} we consider compact solutions of the twisted Hull--Strominger system \eqref{eq:twistedStromintro} and their associated embeddings on complex surfaces, building on the classification in \cite{grt}. Section \ref{subsubsec:Iwasawa} and Section \ref{sec:Picard} contain examples in non-K\"ahler Calabi--Yau threefolds where Theorem \ref{teo:mainTCAintro} applies. In particular, in Section \ref{sec:Picard} we discuss the class of non-K\"ahler solutions of the Hull--Strominger constructed by the third author in \cite{GF4}, inspired by the seminal work by Fu and Yau \cite{FuYau}. The underlying geometry is given by a $\mathbb{T}^2$-fibration over a K\"ahler $K3$ surface, considered originally in the physics literature \cite{DRS,GoPro}. Interestingly, these solutions are compatible with T-duality, and we speculate that they actually provide examples of $(0,2)$-mirror symmetry, following the methods in \cite{AAGa}. We also expect a relation between our $N=2$ embeddings in this case (see Proposition \ref{prop:GPchiral}) and the new supersymmetric index studied in \cite{Israel}.  

\begin{acknowledgements}
The authors wish to thank Xenia de la Ossa, Mateo Galdeano, Pedram Hekmati, Reimundo Heluani, Joshua Jordan, Jock McOrist, Sebastien Picard, Carl Tipler and Alessandro Tomasiello for useful discussions. The second author is grateful to the Instituto de Matem\'atica Pura e Aplicada (IMPA) for the hospitality. 
\end{acknowledgements} 

%%%%%%%%%%%%%%%%%%%%%%%%%%%%%%%%%%%%%%%%%%%%%%%%%%%%%%%%%%%%%%%%%%%%%%%%%
\section{Killing spinors, special coordinates and the Bismut Ricci form}
\label{sec:KSeq}
%%%%%%%%%%%%%%%%%%%%%%%%%%%%%%%%%%%%%%%%%%%%%%%%%%%%%%%%%%%%%%%%%%%%%%%%%%

\subsection{The Killing spinor equations}

In this section we review some background on the Killing spinor equations on a spin manifold endowed with a principal bundle. This will help us to introduce the various geometric structures considered in this paper, and will serve as motivation for the twisted Hull--Strominger system in Section \ref{sec:THSFterm}.

Let $M$ be an $m$-dimensional smooth spin manifold. Given a vector bundle $V$ over $M$, we denote by $\Omega^k(V)$ the space of $V$-valued differential $k$-forms on $M$. Given a Riemannian metric $g$ on $M$, we denote by $\nabla^g$ its Levi-Civita connection. Given a three-form $H \in \Omega^3$, we define metric connections on $T = TM$ with skew-symmetric torsion, compatible with $g$, by
\begin{equation}\label{eq:nabla++1/3}
\nabla^+ = \nabla^g+\frac12g^{-1}H, \qquad \nabla^{\frac13} = \nabla^g+\frac16g^{-1}H.
\end{equation}
We will abuse notation and denote the associated spin connections with the same symbol. Let $K$ be a compact Lie group, and consider a principal $K$-bundle $P \to M$. Given a principal connection $A$ on $P$, we denote its curvature by $F_A \in \Omega^2(\ad P)$. %The following equations have its origins in the mathematical physics literature.

\begin{definition}\label{defn:KSEqSM}
We say that a tuple $(g,H,A,\eta,\varphi)$, where $g$, $H$, $A$ are as above, $\eta$ is a spinor for $(T,g)$, and $\varphi \in \Omega^1$, is a solution to the \textit{Killing spinor equations} if
\begin{subequations}
\begin{align}
F_A \cdot \eta & = 0,\label{eq:GauCourT}\\
\nabla^+\eta & = 0,\label{eq:gravCourT}\\
\left(\slashed{\nabla}^{\frac13} - \tfrac{1}{2}\varphi \cdot \right) \eta & = 0\label{eq:dilCourT},
\end{align}
\end{subequations}
where $\cdot$ denotes Clifford multiplication and $\slashed{\nabla}^{\frac13}$ is the Dirac operator associated to the spin connection $\nabla^{\frac13}$. We will call $\varphi$ the \emph{dilaton one-form}. 
\end{definition}

\begin{remark}
When $\varphi$ is exact, the equations in Definition \ref{defn:KSEqSM} are equivalent in low dimensions to the Killing spinor equations in the internal manifold of a compactification of the ten-dimensional heterotic supergravity (see, e.g., \cite{Strom}).
In this case, the equations are known as the \textit{gaugino}, \emph{gravitino}, and \textit{dilatino equations}, respectively, and a potential for $\varphi$ is known as the \emph{dilaton field}.  %Note that there exists a well-stablished\break notion of Killing spinors for pseudo-Riemannian geometry, which does not agree with the notion considered here. Nonetheless, the name Killing spinor equations refering to \eqref{eq:KSEqCH5} is now well-stablished in the mathematical physics literature.
\end{remark}

In this work, we are interested in solutions of the Killing spinor equations on an even-dimensional oriented spin manifold, that is for $m = 2n$. In this case, a complex spinor bundle decomposes as
$$
S = S^+ \oplus S^-,
$$ 
where $S^\pm$ is the $\pm 1$-eigenspace with respect to the Clifford multiplication by the volume form. We will say that a section of $S^+$ has positive chirality. A pure spinor $\eta \in \Omega^0(S^+)$ determines, by definition, an almost complex structure $J$ compatible with the orientation, with holomorphic tangent bundle
\begin{equation}\label{eq:T01J}
T^{1,0}_J = \{V \in T \otimes \C \; | \; v \cdot \eta = 0 \}.
\end{equation}
The Killing spinor equations for pure spinors $\eta \in \Omega^0(S^+)$ admit a natural characterization in terms of $\SU(n)$-structures, as follows. Recall that an $\SU(n)$-structure on $M$ is given by a pair $(\Psi,\omega)$, where $\Psi$ is a locally decomposable complex $n$-form, with associated almost complex structure $J$, and $\omega$ is a real $(1,1)$-form satisfying $\|\Psi\|_\omega  = 1$. Here, the pointwise norm $\|\Psi\|_\omega$ is defined by
\begin{equation*}\label{eq:norm}
(-1)^{\frac{n(n-1)}{2}}i^n\Psi \wedge \overline{\Psi} = \|\Psi\|_\omega^2 \frac{\omega^n}{n!}.
\end{equation*}
By \cite[Ch. IV, Proposition 9.16]{MicLaw}, a pure spinor $\eta \in \Omega^0(S^+)$ determines a $\SU(n)$-structure $(\Psi,\omega)$ on $M$, with almost complex structure \eqref{eq:T01J}. We will denote by $\theta_\omega = Jd^*\omega \in \Omega^1$ the Lee form of the almost Hermitian structure $(J,\omega)$. Alternatively, $\theta_\omega$ is defined by the structure equation
$$
d \omega^{n-1} = \theta_\omega \wedge \omega^{n-1}.
$$ 
We denote by $N_J$ the Nijenhuis tensor of $J$. The following result was originally proved by Hull and Strominger \cite{HullTurin,Strom} in the case that $\varphi$ is exact. Here, we give a sketch of the proof following \cite{grt}.

\begin{proposition}\label{prop:KSEqevendim}
Let $M$ be a $2n$-dimensional spin manifold endowed with a principal $K$-bundle $P$. A solution of the Killing spinor equations $(g,H,A,\eta,\varphi)$ in Definition \ref{defn:KSEqSM} with $\eta \in \Omega^0(S^+)$ pure, induces an $\SU(n)$-structure $(\Psi,\omega)$ on $M$ satisfying
\begin{equation}\label{eq:KSEqon2ndimman}
\begin{array}{rlrl}
F_A \wedge \omega^{n-1} & = 0, & F_A^{0,2} & = 0,\\ 
\nabla^+ \Psi & = 0, &  H+d^c\omega &= 0,\\ 
\theta_\omega +\varphi & = 0, & N_J & = 0.
\end{array}
\end{equation}
Conversely, a triple $(\Psi,\omega,A)$, %where $(\Psi,\omega)$ is an $\SU(n)$-structure on $M$ and $A$ is connection on $P$, 
such that 
\begin{equation}\label{eq:KSEqon2ndimmanconv}
\begin{array}{rlrl}
F_A \wedge \omega^{n-1} & = 0, & F_A^{0,2} & = 0,\\ 
\nabla^B \Psi & = 0, & N_J & = 0,
\end{array}
\end{equation}
where $\nabla^B = \nabla^g - \frac12g^{-1}d^c\omega$ is the Bismut connection of $\omega$, induces a solution of the Killing spinor equations with $g = \omega(,J)$, $H = - d^c\omega$, and $\varphi = -  \theta_\omega$.
\end{proposition}

\begin{proof}
Let $(\Psi,\omega)$ be the $\SU(n)$-structure on $M$ determined by $\eta$. Then condition \eqref{eq:gravCourT} implies that
$$
\nabla^+ \omega = 0, \qquad \nabla^+ \Psi = 0, \qquad \nabla^+J = 0,
$$
and, by \cite[Eq. (2.5.2)]{Gau}, we have
$$
H^{3,0 + 0,3} = g(N_J,), \qquad H^{2,1 + 1,2} = - (d^c\omega)^{2,1 + 1,2}.
$$ 
We denote $T^{0,1} = T^{0,1}_J$ and $T^*_{0,1} = (T^{0,1})^*$. Consider the model for the complex spinor bundle
$$
S_+ \cong \Lambda^{0,\text{even}}T^*_{0,1},
$$
with Clifford multiplication given by
$$
\xi \cdot \sigma = \sqrt{2}(\iota_{g^{-1} \xi^{1,0}}\sigma + \xi^{0,1} \wedge \sigma).
$$
In this model, the spinor $\eta$ can be identified with $\lambda \in \C^*$. Using that $\nabla^+ \eta = 0$, equation \eqref{eq:dilCourT} is equivalent to
$$
\left(H - \varphi \right) \cdot 1 = 0.
$$ 
Arguing now as in \cite[Th. 5.1]{grt}, we obtain that $H^{0,3} = 0$, and therefore $N_J = 0$, and also that $\varphi = - \theta_\omega$. Similarly, \eqref{eq:GauCourT} is equivalent to the first two conditions in \eqref{eq:KSEqon2ndimman}. 

Conversely, given a solution of \eqref{eq:KSEqon2ndimmanconv}, by \cite[Eq. (3.2.3)]{Gau} the pure spinor 
$$
\eta = 1 \in \Omega^0(\Lambda^{0,\text{even}}T^*_{0,1})
$$
satisfies $\nabla^B \eta = 0$. Similarly as above, the first two conditions in \eqref{eq:KSEqon2ndimman} imply $F_A \cdot \eta = 0$. Finally, setting $H = - d^c\omega$, one has $\slashed{\nabla}^{+\frac13} \eta  = \tfrac{1}{2} H \cdot \eta$, and the proof follows.
\end{proof}

\subsection{Closed dilaton and special coordinates}\label{sec:FDterm}

In this work, we are mainly concerned with solutions of the Killing spinor equations with closed dilaton one-form $\varphi$, that is, satisfying $d\varphi = 0$. In the setup of Proposition \ref{prop:KSEqevendim}, this is equivalent to the Lee form $\theta_\omega$ of the $\SU(n)$-structure being closed or, equivalently, the Hermitian form $\omega$ being locally conformally balanced. The goal of this section is to construct a special holomorphic atlas associated to solutions of \eqref{eq:KSEqon2ndimmanconv} satisfying $d\theta_\omega = 0$. To start, we recall a more amenable characterization of the $\SU(n)$-holonomy condition for the Bismut connection, following \cite{grst,grst2}. 

\begin{lemma}\label{lem:holonomia}
Let 	$\left(\Psi,\omega\right)$ be an $\SU(n)$-structure on $M$ with $N_J = 0$ and $d \theta_\omega =0$. Then
$$
\nabla^B \Psi = 0 \qquad \textrm{ if and only if } \qquad d\Psi - \theta_\omega \wedge \Psi = 0,
$$
where $\nabla^B = \nabla^g-\frac12 g^{-1}d^c\omega$ denotes the Bismut connection of $\left(\Psi,\omega\right)$.
\end{lemma}

\begin{proof}
The `if part' was proved in \cite[Lemma 2.2]{grst}, and requires the condition $d \theta_\omega =0$. For the `only if part', we use Gauduchon's formula \cite[Equation (2.7.6)]{Gau}, relating the connections induced on the canonical bundle by the Chern connection $\nabla^C$ and the Bismut connection $\nabla^B$, namely,
\begin{equation*}
\nabla^C\Psi = \nabla^B\Psi + id^*\omega\otimes\Psi.
\end{equation*}
Taking the $(0,1)$-part of this expression combined with $\nabla^B\Psi = 0$, it follows that
$$
\dbar \Psi = \theta_\omega^{0,1} \wedge \Psi,
$$
which implies $d\Psi = \theta_\omega \wedge \Psi$, as claimed.
\end{proof}

We are ready to prove the main result of this section. It generalizes the existence of special holomorphic coordinates on a complex manifold equipped with a holomorphic volume form, i.e., a holomorphic atlas whose transition maps have holomorphic Jacobian with determinant equal to one.
%We are ready to prove the main result of this section. The following lemma provides a generalization of the existence of special holomorphic coordinates, that is, with constant determinant of the holomorphic Jacobian equal to one, on a complex manifold endowed with a holomorphic volume form.  

\begin{lemma}\label{lem:holconstdetatlas}
Let $M$ be a $2n$-dimensional smooth manifold endowed with an $\SU(n)$-structure $(\Psi,\omega)$ satisfying
\begin{equation}\label{eq:confbal}
d\Psi-\theta_\omega\wedge\Psi = 0, \qquad d\theta_\omega = 0.
\end{equation}
Then $M$ admits a unique maximal holomorphic atlas such that for any coordinate open patch $U \subset M$ in this atlas there exists a smooth function $\phi_U \in C^\infty(U)$ such that 
\begin{equation}\label{eq:Psinormalized}
\left.\Psi\right|_U = e^{\phi_U} dz_1 \wedge \cdots \wedge dz_n, \qquad \theta_\omega |_U = d\phi_U.
\end{equation}
Consequently, the holomorphic Jacobian of any change of coordinates in this atlas has constant determinant.
\end{lemma}
\begin{proof}
We first prove that the almost complex structure $J$ determined by $\Psi$ is integrable. Let $p \in M$ and $U \subset M$ be a contractible open set with $p \in U$. By hypothesis, there exists a function $\phi_U \in C^\infty(U)$ satisfying $\theta_{\omega}|_U = d\phi_U$. By the first equation in \eqref{eq:confbal}, the local $(n,0)$-form
\begin{equation}
\label{eq:omegapsi}
\Omega_U = e^{-\phi_U}\left.\Psi\right|_U
\end{equation}
is closed, and hence $M$ admits holomorphic coordinates around $p$. 

Next, given a contractible holomorphic coordinate patch $U \subset M$ around $p \in M$, we can repeat the above argument to define a local holomorphic $(n,0)$-form \eqref{eq:omegapsi}, with $\theta_{\omega}|_U = d\phi_U$. By Moser's trick, we can construct new holomorphic coordinates $(z_1,\ldots,z_n)$ in a contractible open subset $U'\subset U$ around $p$ such that
$$
\Omega_{U'} = dz_1 \wedge \cdots \wedge dz_n.
$$
Notice that \eqref{eq:omegapsi} still holds in the new coordinates, for a suitable local potential $\phi_{U'}$ for $\theta_{\omega}|_{U'} = d\phi_{U'}$. Consequently, \eqref{eq:Psinormalized} is satisfied in these coordinates.

Consider now the unique maximal holomorphic atlas formed by all the contractible holomorphic coordinate patches $U \subset M$ such that \eqref{eq:Psinormalized} is satisfied. Let $U$ and $U'$ be two coordinate patches in this atlas, with coordinates $(z_1,\ldots,z_n)$ and $(z_1',\ldots,z_n')$, respectively. Then on $U \cap U'$ we have
$$
d\left(\phi_U - \phi_{U'}\right) = d\phi_U- d\phi_{U'} = \left.\theta_\omega\right|_U-\left. \theta_\omega\right|_{U'} = 0,
$$
and therefore $c := \phi_U-\phi_{U'} \in \C$ in $U\cap U'$. Denoting by $(z_1',\ldots,z_n') = \psi(z_1,\ldots,z_n)$ the corresponding change of coordinates, we obtain
$$
\det\left(d\psi\right)dz_1 \wedge \cdots \wedge dz_n = \Omega_{U'}|_{U\cap U'} = e^{-\phi_{U'}}\left.\Psi\right|_{U\cap U'} = e^c\Omega_{U}|_{U\cap U'} = e^cdz_1 \wedge \cdots \wedge dz_n
$$
%$$
%\det\left(d\psi\right)dz_1 \wedge \cdots \wedge dz_n = \psi^*\phi'_*\Omega_{U'} = \phi_*e^{-\phi_{U'}}\Psi = \phi_*e^{c-%\phi_{U}}\Psi = e^c\phi_*\Omega_{U} =  e^cdz_1 \wedge \cdots \wedge dz_n
%$$
on $U\cap U'$, which implies $\det\left(d\psi\right)=e^c$, as claimed.
\end{proof}

\subsection{Bismut Ricci form and torsion bivector}\label{sec:sigmaomega}

The aim of this section is to study some new aspects of the Bismut Ricci form of a Hermitian structure on a complex manifold. In particular, we will derive a new formula for its $(2,0)$-component, and relate it to a bivector field canonically associated to the torsion of the Hermitian structure. This interesting quantity plays a distinguished role in our main results (Theorem \ref{teo:mainTCA} and Theorem \ref{teo:mainTCAcHE}).

Let $(M,J)$ be a complex $n$-dimensional manifold with Hermitian metric $g$ and corresponding Hermitian form $\omega = g(J\cdot,\cdot)$. We denote $T^{0,1} = T^{0,1}_J$ and $T^*_{0,1} = (T^{0,1})^*$. Recall that the Bismut Ricci form $\rho_B$ of $g$ is defined as $i := \sqrt{-1}$ times the curvature of the connection induced by the Bismut connection $\nabla^B = \nabla^g - \frac12g^{-1}d^c\omega$ on the anti-canonical bundle 
$$
K^{-1}_M = \Lambda^n T^{1,0}.
$$ 
More explicitly, for a choice of $g$-orthonormal basis $e_1, \ldots, e_{2n}$ of $T$ at a point, one has
\begin{equation}\label{eq:secondRicci}
\rho_B(V,W) = \frac{1}{2} \sum_{j=1}^{2n} g(R_{\nabla^B}(V,W)Je_j,e_j),
\end{equation}
where $R_{\nabla^B}$ is the curvature tensor of $\nabla^B$ on the tangent bundle. By the proof of \cite[Proposition 5.6]{GFGM0}, we have the following formulae, valid on any Hermitian manifold:
\begin{equation}\label{eq:rhoBdecomp}
\begin{split}
    \rho_B^{1,1}(\cdot,J\cdot) & = \operatorname{Rc} - \frac{1}{4}H^2 - \frac{1}{4} dH (\cdot,J\cdot,e_j,Je_j) + \tfrac{1}{2}L_{\theta_\omega^\sharp} g,\\
    \rho_B^{2,0+0,2}(\cdot,J\cdot) & = -\frac{1}{2}( d^*H - d \theta_\omega +  i_{\theta_\omega^\sharp}H),
\end{split}
\end{equation}
where $\operatorname{Rc}$ is the Riemannian Ricci tensor, $H = -d^c\omega$ and we denote
$$
H^2 =\sum_{i,j=1}^{2n} H(e_j,e_k,\cdot)H(e_j,e_k,\cdot).
$$
In this section we are specifically interested in $\rho_B^{2,0}$, the $(2,0)$ component of the Bismut Ricci form, for which we derive next a formula in holomorphic coordinates following \cite[Proposition 2.2]{Jordan}. We note that the original result in Jordan's PhD Thesis \cite{Jordan} is stated for pluriclosed metrics, but the proof, which we sketch here for the convenience of the reader, is valid for arbitrary Hermitian structures.

\begin{proposition}[\cite{Jordan}]\label{prop:rho20}
Given local holomorphic coordinates around a point, one has
\begin{equation*}\label{eq:rho20}
(\rho_B^{2,0})_{jk} = - i g^{p\overline{l}}\nabla^C_p(\partial \omega)_{jk \overline{l}},
\end{equation*}
where $\nabla^C$ denotes the Chern connection of the Hermitian metric $g$.
\end{proposition}

\begin{proof}
Denote $T = - i\partial \omega$. Then, as a consequence of the Bianchi identities, one has
$$
(R_{\nabla^B})_{qjk\overline{l}} = \nabla^C_qT_{kj \overline{l}} - \nabla^C_jT_{kq \overline{l}} + T_{qj}^p T_{kp \overline{l}} - T_{qk}^p T_{jp \overline{l}} + T_{jk}^p T_{qp \overline{l}}.
$$
The condition $\partial T = - \sqrt{-1}\partial^2 \omega = 0$ can be written in terms of the Chern connection as
$$
\nabla^C_qT_{jk \overline{l}} - \nabla^C_jT_{qk \overline{l}} + \nabla^C_kT_{qj \overline{l}} +  T_{qj}^p T_{pk \overline{l}} +  T_{jk}^p T_{pq \overline{l}} - T_{qk}^p T_{pj \overline{l}} = 0.
$$
Combining now the two previous identities, we obtain 
$$
(R_{\nabla^B})_{qjk\overline{l}} = \nabla^C_k T_{qj\overline{l}},
$$
and the statement follows from \eqref{eq:secondRicci} by taking trace in the previous expression with respect to the hermitian metric.
\end{proof}

We define next a bivector field canonically associated to the torsion of the Bismut connection via the Schouten bracket. Firstly, given $v \in \Omega^0(T^{0,1})$, we have a bivector field
\begin{equation*}
\left(g^{-1}\otimes g^{-1}\right)\left(\iota_{v}i\partial\omega\right) \in \Omega^0\left(\Lambda^2T^{0,1}\right)
\end{equation*}
of type $(0,2)$. 
% Then, we can take the Schouten bracket
%\begin{equation*}
%\left[w,\left(g^{-1}\otimes g^{-1}\right)\left(\iota_{v}i\partial\omega\right)\right]\in\Omega^0(\Lambda^2T\otimes \C),
%\end{equation*}
Then, for $w \in \Omega^0(T^{1,0})$, we can take the Schouten bracket and then project onto its $(0,2)$-component:
\begin{equation}\label{eq:02Sch}
\left[w,\left(g^{-1}\otimes g^{-1}\right)\left(\iota_{v}i\partial\omega\right)\right]^{0,2}\in\Omega^0(\Lambda^2T^{0,1}).
\end{equation}
Given local holomorphic coordinates around a point, we define
\begin{equation}
\label{eq:sigmaomega}
\sigma_\omega := \sum_{k=1}^n\left[g^{-1}d\overline{z}_k,\left(g^{-1}\otimes g^{-1}\right)\left(\iota_{\overline{\partial}_k}i\partial\omega\right)\right]^{0,2},
\end{equation}
where $\overline{\partial}_k = \frac{\partial}{\partial\overline{z}_k}$.

\begin{lemma}\label{lem:sigmaomega}
The expression \eqref{eq:sigmaomega} is independent of the choice of local holomorphic coordinates. In consequence, it defines a bivector field of type $(0,2)$, that is,
$$
\sigma_\omega \in \Omega^0\left(\Lambda^2T^{0,1}\right).
$$
Furthermore, the components of $\sigma_\omega$ are given by
\begin{equation}\label{eq:sigmaomegaexp}
\left(\sigma_\omega\right)_{\overline{pq}} = i\sum_{k=1}^ng^{-1}d\overline{z}_k\left(\partial\omega\left(g^{-1}d\overline{z}_p,g^{-1}d\overline{z}_q,\overline{\partial}_k\right)\right).
\end{equation}
\end{lemma}

\begin{proof}
Observe that the expression \eqref{eq:02Sch} is tensorial in $w$, by definition of the Schouten bracket. Furthermore, if $f$ is a local function such that $\partial f = 0$, then
$$
\left[w,\left(g^{-1}\otimes g^{-1}\right)\left(\iota_{f v}i\partial\omega\right)\right]^{0,2} = f \left[w,\left(g^{-1}\otimes g^{-1}\right)\left(\iota_{v}i\partial\omega\right)\right]^{0,2}.
$$
The first part follows now by change of holomorphic coordinates. The second part follows by definition, or alternatively by the local expression
\begin{equation}\label{eq:sigmaomegaexplocal}
\sigma_\omega = \sum_{1 \leq p<q \leq n} g^{s\overline{k}}\partial_s\left(g^{m\overline{p}}g^{l\overline{q}}\left(i\partial\omega\right)_{ml\overline{k}}\right)\overline{\partial}_p\wedge\overline{\partial}_q,
\end{equation}
where $\partial_k := \frac{\partial}{\partial z_k}$ and we use the Einstein summation convention for repeated indices.
\end{proof}

%We will prove in Chapter \ref{Resultados1} that we are able to construct embeddings of SUSY vertex algebras when thed F-term and D-term conditions are satisfied, provided that $\sigma_{\omega}=0$.

Applying now Proposition \ref{prop:rho20}, in our next result we establish the relation between the torsion bivector field and the Bismut Ricci form.

\begin{proposition}\label{prop:rho20sigma}
For an arbitrary hermitian metric, one has
\begin{equation*}\label{eq:rho20sigma}
\sigma_\omega=-(g^{-1}\otimes g^{-1})\rho_B^{2,0}.
\end{equation*}
\end{proposition}

\begin{proof}
Denote by $\Gamma_{ij}^s$ the Christoffel symbols of the Chern connection $\nabla^C$ in holomorphic coordinates, given by
$$
\Gamma_{ij}^s = g^{s \overline{l}}\partial_i g_{j\overline{l}}.
$$
Then, using that (see the proof of Proposition \ref{prop:rho20})
$$
T_{ij\overline{k}} = \partial_i g_{j \overline{k}} - \partial_j g_{i \overline{k}}
$$
we have
\begin{align*}
\nabla^C_kT_{ij \overline{l}} & = \partial_kT_{ij \overline{l}} - T_{sj \overline{l}}\Gamma^s_{ki} - T_{is \overline{l}}\Gamma^s_{kj} -  T_{ij \overline{s}}\Gamma^{\overline{s}}_{k\overline{l}}\\
& = \partial_k\partial_i g_{j \overline{l}} - \partial_k\partial_j g_{i \overline{l}} - (\partial_s g_{j\overline{l}} - \partial_j g_{s\overline{l}})g^{s \overline{t}}\partial_k g_{i\overline{t}} - (\partial_i g_{s\overline{l}} - \partial_s g_{i\overline{l}})g^{s \overline{t}}\partial_k g_{j\overline{t}},
\end{align*}
where in the second equality we have used that $\Gamma^{\overline{s}}_{k\overline{l}} = 0$, by the properties of the Chern connection. On the other hand, a straighforward calculation shows that
$$
\partial_k(g^{i\overline{p}}g^{j\overline{q}} T_{ij \overline{l}}) = g^{i\overline{p}}g^{j\overline{q}} \nabla^C_kT_{ij \overline{l}},
$$
and hence Proposition \ref{prop:rho20} combined with \eqref{eq:sigmaomegaexplocal} now gives
\[
(\sigma_\omega)_{\overline{pq}} = - g^{s \overline{l}}g^{j\overline{p}}g^{k\overline{q}} \nabla^C_sT_{jk \overline{l}} = - g^{j\overline{p}}g^{k\overline{q}}(\rho_B^{2,0})_{jk}.
\qedhere
\]
\end{proof}

To finish this section, we provide a simple example of Hermitian metric with non-vanishing torsion bivector field, and hence with $\rho_B^{2,0} \neq 0$, in complex dimension two.

\begin{example}\label{ex:sigmanonzero}
Let $X \subset \C^2$ be a complex domain and consider the Hermitian form
$$
\omega = \frac{i}{2}\left(u dz_1\wedge d\overline{z}_1 + v dz_2\wedge d\overline{z}_2\right),
$$
for a pair of smooth positive real functions $u,v \in C^\infty(X)$. Notice that
\begin{equation*}
g^{-1}d\overline{z}_1 = 2 u^{-1}\partial_1, \qquad g^{-1}d\overline{z}_2 = 2 v^{-1}\partial_2.
\end{equation*}
Then, by direct calculation using Lemma \ref{lem:sigmaomega}, we have
\begin{align*}
\left(\sigma_\omega\right)_{\overline{12}} & = ig^{-1}d\overline{z}_1\left(4 (uv)^{-1}\partial u \wedge \omega_0 \left(\partial_1,\partial_2,\overline{\partial}_1\right)\right) + ig^{-1}d\overline{z}_2\left(4 (uv)^{-1}\partial v \wedge \omega_0 \left(\partial_1,\partial_2,\overline{\partial}_2\right)\right)\\
%& = \frac{i}{2}2iu^{-1}\partial_1\left(- 4 (uv)^{-1}\partial_2 u \right) + \frac{i}{2} 2iv^{-1}\partial_2\left(4 (uv)^{-1}\partial_1 v \right)\\
& = 4u^{-1}\partial_1\left(v^{-1}\partial_2 \log u \right) - 4v^{-1}\partial_2\left(u^{-1}\partial_1 \log v \right)\\
%& = -4(vu)^{-1}(\partial_1\log v)(\partial_2 \log u) + 4(uv)^{-1}(\partial_2 \log u)(\partial_1 \log v)\\
%& + 4(vu)^{-1}(\partial_1\partial_2 \log u) - 4(uv)^{-1}(\partial_1 \partial_2 \log v)\\
& =4 \frac{\partial_1\partial_2 \log u - \partial_2\partial_1 \log v}{uv}
\end{align*}
and therefore
$$
\sigma_\omega = 4 \frac{\partial_1\partial_2\log (u/v)}{uv}\overline{\partial}_1\wedge\overline{\partial}_2.
$$
Since $\log(u/v)$ is real, for $\sigma_\omega$ to vanish we must have
%$$
%\overline{\partial}_2\log(u/v) = p(z_1,z_2,\overline{z_2}), \qquad \overline{\partial}_1\log(u/v) = p(z_1,\overline{z}_1,z_2)
%$$
$$
\log(u/v) = \overline{z}_1 p(z_2) + \overline{z}_2 q(z_1) + t(z_1) + s(z_2) + \text{c.c.}
$$
with $p,q,t,s$ holomorphic functions in one variable, and therefore the torsion bivector field is non-zero for generic choices of $u$ and $v$. %(e.g., one can take $u = 1$ and $v = 1 + |z_1|^2 + |z_2|^2$). 
%The case $u = v$, with vanishing torsion bivector field, will be relevant to Section \ref{sec:exdomain}.
\end{example}

%%%%%%%%%%%%%%%%%%%%%%%%%%%%%%%%%%%%%%%%%%%%%%%%%%%%%%%%%%%%%%%%%%%%%%%%%
\section{Canonical metrics on string algebroids}\label{sec:tHSString}
%%%%%%%%%%%%%%%%%%%%%%%%%%%%%%%%%%%%%%%%%%%%%%%%%%%%%%%%%%%%%%%%%%%%%%%%%%

\subsection{Background on string algebroids}

We start by recalling the necessary background material on Courant algebroids of string type, following \cite{grt2}.

\begin{definition}\label{defn:CA}
A \emph{Courant algebroid} $(E,\left\langle\cdot,\cdot\right\rangle,\left[\cdot,\cdot\right],\pi)$ over a smooth manifold $M$ consists of a vector bundle $E \to M$ endowed with a non-degenerate symmetric bilinear form $\left\langle\cdot,\cdot\right\rangle$, a (Dorfman) bracket $\left[\cdot,\cdot\right]$ on $\Omega^0(E)$, and a bundle map $\pi \colon E \to T$, called \emph{anchor map}, such that the following axioms are satisfied, for all $a,b,c \in \Omega^0(E)$ and $f \in C^\infty(M)$:
\begin{enumerate}
\item[(1)] $\left[a,\left[b,c\right]\right] = \left[\left[a,b\right],c\right]+\left[b,\left[a,c\right]\right]$,
\item[(2)] $\pi\left[a,b\right] = \left[\pi(a),\pi(b)\right]$,
\item[(3)] $\left[a,fb\right] = f\left[a,b\right] + \pi(a)(f)b$,
\item[(4)] $\pi(a)\left\langle b,c\right\rangle = \left\langle\left[a,b\right],c\right\rangle + \left\langle b,\left[a,c\right]\right\rangle$,
\item[(5)] $\left[a,b\right]+\left[b,a\right] = \mathcal{D}\left\langle a,b\right\rangle$,
\end{enumerate}
Here, the map $\mathcal{D}\colon C^\infty(M)\to\Omega^0(E)$ denotes the composition of the exterior differential $d \colon C^\infty(M) \to \Omega^1$, the dual map $\pi^* \colon T^* \to E^*$ and the isomorphism $E^* \cong E$ provided by $\left\langle\cdot,\cdot\right\rangle$. In particular, one has
\begin{equation*}
\left\langle\mathcal{D}f,a\right\rangle = \pi(a)(f),
\end{equation*}
\end{definition}
so the above identity (4) can be reformulated as
\[
\langle a,\mathcal{D}\langle b,c\rangle\rangle %=\pi(a)\langle b,c\rangle\rangle
=\langle[a,b],c\rangle + \langle b,[a,c]\rangle.
\]

We will denote a Courant algebroid $(E,\left\langle\cdot,\cdot\right\rangle,\left[\cdot,\cdot\right],\pi)$ simply by $E$. Using the isomorphism $\left\langle\cdot,\cdot\right\rangle \colon E \to E^*$, we obtain a complex of vector bundles
\begin{equation}\label{eq:Couseqaux}
T^* \overset{\pi^*}{\longrightarrow} E \overset{\pi}{\longrightarrow} T.
\end{equation}
We will say that $E$ is \emph{transitive} if the anchor map $\pi$ in \eqref{eq:Couseqaux} is surjective. Given a transitive Courant algebroid $E$ over $M$, there is an associated Lie algebroid 
$$
A_E := E/(\Ker \pi)^\perp.
$$ 
Furthermore, the subbundle
$$
\ad_E := \Ker \pi/(\Ker \pi)^\perp \subset A_E
$$
inherits the structure of a bundle of quadratic Lie algebras. Therefore, the bundle $E$ fits into a double extension of vector bundles
\begin{equation}\label{eq:Coustr}\begin{gathered}
0 \longrightarrow T^* \overset{\pi^*}{\longrightarrow} E \longrightarrow A_E \longrightarrow 0,\\
0 \longrightarrow \ad_E \overset{\pi^*}{\longrightarrow} A_E \overset{\pi}{\longrightarrow} T \longrightarrow 0.
\end{gathered}\end{equation}
A classification of transitive Courant algebroids has been obtained in \cite[Proposition A.6]{grt2}, in the special case that $A_E$ is isomorphic to the Atiyah algebroid of a smooth principal bundle. A Courant algebroid of this form is said  to be of \emph{string type}. To give a more precise definition, we briefly discuss the basic example which we will need.

\begin{example}\label{def:E0}
Let $K$ be a real Lie group with Lie algebra $\mathfrak{k}$. We assume that $\mathfrak{k}$ is endowed with a non-degenerate bi-invariant symmetric bilinear form
$$
\left\langle\cdot,\cdot\right\rangle: \mathfrak{k} \otimes \mathfrak{k} \longrightarrow \mathbb{R}.
$$
Let $p \colon P \to M$ be a smooth principal $K$-bundle over $M$. Consider the Atiyah Lie algebroid $A_{P} := TP/K$. The smooth bundle of Lie algebras 
$
\ad P:= \Ker dp \subset A_{P}
$
fits into the short exact sequence of Lie algebroids
$$
0 \longrightarrow \ad P \to A_{P} \longrightarrow T \to 0.
$$
We construct next a transitive Courant algebroid such that the second sequence in \eqref{eq:Coustr} is canonically isomorphic to the exact sequence of Lie algebroids above. Assume that
$$
p_1(P) = 0 \in H^4_{dR}(M,\mathbb{R}),
$$
where $p_1(P)$ denotes \emph{first Pontryagin class} of $P$ associated $\left\langle\cdot,\cdot\right\rangle$ via Chern--Weil theory. Then, given a choice of principal connection $A$ on $P$, there exists a smooth real three-form $H \in \Omega^{3}$ such that
\begin{equation}\label{eq:Bianchireal}
dH - \langle F_A \wedge F_A \rangle = 0.
\end{equation}
Given such a pair $(H,A)$, we define a Courant algebroid $E_{P,H,A}$ with underlying vector bundle
\begin{equation*}%\label{eq:Qexpb}
T \oplus\ad P\oplus T^*,
\end{equation*}
non-degenerate symmetric bilinear form
\begin{equation}\label{eq:pairingE0}
\langle X + r + \xi , X + r + \xi \rangle_0  = \xi(X) + \langle r,r \rangle,
\end{equation}
bracket given by
\begin{equation}\label{eq:bracketE0}
	\begin{split}
	[X+ r + \xi,Y + t + \eta]_0   = {} & [X,Y] - F_A(X,Y) + d^A_X t - d^A_Y r - [r,t]\\
	& {} + L_X \eta - \iota_Y d \xi + \iota_Y\iota_X H\\
	& {} + 2\langle d^A r, t \rangle + 2\langle \iota_X F_A, t \rangle - 2\langle \iota_Y F_A, r \rangle,
	\end{split}
\end{equation}
and anchor map $\pi_0(X + r + \xi) = X$. It is not difficult to see that $E_{P,H,A}$ defines a smooth transitive Courant algebroid over $M$, as defined above.
\end{example}

Transitive Courant algebroids as in Example \ref{def:E0} fit into the category of \email{string algebroids} \cite{grt2}, which motivates the following definition.

\begin{definition}\label{d:CAstring} 
A Courant algebroid $E$ over $M$ is of \emph{string type} if it is isomorphic to a Courant algebroid $E_{P,H,A}$ as in Example \ref{def:E0}, for some triple $(P,H,A)$ satisfying \eqref{eq:Bianchireal}. In this case, we will refer to $E$ simply as a \emph{string algebroid}.
\end{definition}

The notion of isomorphism which we use here is the standard one for Courant algebroids, given by (base-preserving) smooth orthogonal bundle morphisms which preserve the bracket and the anchor map (cf. \cite[Definition 2.3]{grt2}).

\subsection{Killing spinors on string algebroids}

In this section we recall the definition of the Killing spinor equations on string algebroids, following \cite{GF3,grt}. Roughly speaking, this is an enhancement of the conditions in Definition \ref{defn:KSEqSM}, where the equations are supplemented by the \emph{Bianchi identity} \eqref{eq:Bianchireal}, relating the three-form $H$ with the curvature of the connection $F_A$.

Let $M$ be a smooth oriented spin manifold endowed with a string algebroid $E$. The Killing spinor equations on $E$ are defined in terms of a pair $(V_+,\operatorname{div})$, as given in the following definition.

\begin{definition}\label{def:GmetricE}
\hfill

\begin{enumerate}
\item A (Riemannian) \emph{generalized metric} on $E$ is an orthogonal decomposition $E = V_+ \oplus V_-$, so that the restriction of $\langle \cdot,\cdot\rangle$ to $V_+$ is positive definite and that $\pi_{|V_+}:V_{+}\rightarrow T$ is an isomorphism.
	
\item A \emph{divergence operator} on $E$ is a map $\operatorname{div} \colon \Omega^0(E) \to C^\infty(M)$ satisfying
$$
\operatorname{div}(fa) = f \operatorname{div}(a) + \pi(a)(f)
$$
for any $f \in C^\infty(M)$.

\end{enumerate}
	
\end{definition}

A generalized metric $V_+ \subset E$ is equivalent to a Riemannian metric $g$ on $M$ and an isotropic splitting $s \colon T \to E$ (see \cite[Proposition 3.4]{GF1}). In particular, it induces an isomorphism $E \cong E_{P,H,A}$ for a uniquely determined principal bundle $P$, three-form $H \in \Omega^3$, and principal connection $A$ on $P$ satisfying \eqref{eq:Bianchireal} (see Example \ref{def:E0}). Furthermore, via this identification we have
\begin{equation}\label{eq:Vpm}
V_+ = \{X + g(X)\,\mid\, X \in T\}, \quad V_- = \{X - g(X) + r\,\mid\, X \in T, r \in \ad P\}.
\end{equation}
We will use the following notation for the induced orthogonal projections: 
$$
\pi_\pm\colon  E \longrightarrow V_\pm \colon\, a \longmapsto a_\pm.
$$
Note that the generalized metric has an associated \emph{Riemannian divergence} defined by
\begin{equation}\label{eq:div0}
\operatorname{div}_0(X + r + \xi) = \frac{L_X \mu_g}{\mu_g},
\end{equation}
where $\mu_g$ is the volume element of $g$. Following \cite{GFStreets}, we introduce next a compatibility condition for pairs $(V_+,\operatorname{div})$. Recall that the space of divergence operators on $E$ is affine, modelled on the space of sections of $E^* \cong E$.

\begin{definition}\label{def:compatibleE}
A pair $(V_+,\operatorname{div})$ is \emph{compatible} if $\la \varepsilon, \cdot \ra:= \operatorname{div}_0 - \operatorname{div}$ is an infinitesimal isometry of $V_+$, that is,
$$
[\varepsilon,V_+] \subset V_+.
$$ 
Furthermore, we say that $(V_+,\operatorname{div})$ is \emph{closed} if $\varepsilon \in \Omega^1 \subset \Omega^0(E)$.	
\end{definition}

Given a compatible pair $(V_+,\operatorname{div})$, we have that $\varepsilon = X  + r + \varphi$ in the splitting determined by $V_+$ and the infinitesimal isometry condition reads (cf. \cite[Lemma 2.50]{GFStreets})
\begin{equation}\label{eq:infisometry}
L_Xg = 0, \qquad d^A r = - \iota_X F_A, \qquad d\varphi = i_XH - 2 \langle F_A,r \rangle.
\end{equation}
Thus, the condition of being closed corresponds to $X = r = 0$ and $d \varphi = 0$.

To introduce the Killing spinor equations on $E$, let us fix a pair $(V_+,\operatorname{div})$ as above. The fixed spin structure on $M$, combined with the isometry
\begin{equation}\label{eq:isommetry}
\sigma_+ \colon (T,g) \longrightarrow V_+ \colon\, X \longmapsto X + g(X),
\end{equation} 
determines a complex spinor bundle $S$ for $V_+$ (upon a choice of irreducible representation of the complex Clifford algebra $Cl(n,\CC)$). Associated to the pair $(V_+,\operatorname{div})$ there are canonical first order differential operators \cite{GF3,grt} (see also \cite{CSW})
\begin{equation}\label{eq:LCspinpuregeom}
\begin{split}
	D^{S}_- & \colon \Omega^0(S) \to \Omega^0(V_-^* \otimes S), \qquad \slashed D^+ \colon \Omega^0(S) \to \Omega^0(S_+).
\end{split}
\end{equation}
The operator $D^{S}_-$ corresponds to the unique lift to $S$ of the metric-preserving operator 
$$
D_{a_-} b_ + = [a_-,b_+]_+,
$$
acting on sections $a_- \in \Omega^0(V_-)$ and $b_+ \in \Omega^0(V_+)$. The Dirac-type operator $\slashed D^+$ is more difficult to construct, as it involves torsion-free generalized connections. In the situation of our interest, both operators can be described explicitly in terms of the affine metric connections with totally skew-symmetric torsion in \eqref{eq:nabla++1/3}.

\begin{lemma}[\cite{grt}]\label{lem:DpmexpE}
Let $(V_+,\operatorname{div})$ be a pair given by a generalized metric and a divergence operator on a string algebroid $E$. Denote $\operatorname{div}_0 - \operatorname{div} = \la \varepsilon, \cdot \ra $, and set $\varphi_+ = g(\pi \varepsilon_+, \cdot) \in T^*$. We identify $S$ with a spinor bundle for $(T,g)$, via \eqref{eq:isommetry}. Then, for any spinor $\eta \in \Omega^0(S)$ one has
\begin{align*}
D_{a_-}^{S} \eta & = \nabla^+_X \eta + \langle F_A,r \rangle \cdot \eta,\\
\slashed D^+ \eta & = \slashed \nabla^{\tfrac{1}{3}} \eta - \tfrac{1}{2}\varphi_+ \cdot \eta,
\end{align*}
where $a_- = X - g(X) + r$.
\end{lemma} 

We are ready to introduce the equations of our interest. 

\begin{definition}\label{def:killingE}
Let $M$ be a smooth oriented spin manifold endowed with a string algebroid $E$. A triple $(V_+,\operatorname{div},\eta)$, given by a generalized metric $V_+$, a divergence operator $\operatorname{div}$, and a spinor $\eta \in \Omega^0(S)$, is a solution of the \emph{Killing spinor equations}, if
	\begin{equation}\label{eq:killingE}
		\begin{split}
			D^{S}_- \eta &= 0,\\
			\slashed D^+ \eta &= 0.
		\end{split}
	\end{equation}
\end{definition}

We finish this section establishing the relation with the equations in Definition \ref{defn:KSEqSM}.

\begin{lemma}\label{lem:KSEqSMCA}
Let $M$ be a smooth oriented spin manifold. A solution $(V_+,\operatorname{div},\eta)$ of the Killing spinor equations \eqref{eq:killingE} on a string algebroid $E$ determines a tuple $(g,H,A,\eta,\varphi)$ as in Definition \ref{defn:KSEqSM}, solving the equations
\begin{equation}\label{eq:KSEqBi}
\begin{split}
	F_A \cdot \eta & = 0,\\
	\nabla^+\eta & = 0,\\
	\left(\slashed{\nabla}^{+\frac13} - \tfrac{1}{2}\varphi \cdot \right) \eta & = 0,\\
	dH - \langle F_A \wedge F_A \rangle & = 0.
\end{split}
\end{equation}
Conversely, any solution of \eqref{eq:KSEqBi} determines a string algebroid as in Example \ref{def:E0}, endowed with a solution of the Killing spinor equations \eqref{eq:killingE}.
\end{lemma}

\begin{proof}
Given a solution $(V_+,\operatorname{div},\eta)$ of \eqref{eq:killingE}, the generalized metric $V_+$ determines a triple $(g,H,A)$ satisfying \eqref{eq:Bianchireal} and an isomorphism $E \cong E_{P,H,A}$ (see Example \ref{def:E0}). Set $\varphi_+ = g(\pi \varepsilon_+, \cdot) \in T^*$, where $\operatorname{div}_0 - \operatorname{div} = \la \varepsilon, \cdot \ra $. Then, by Lemma \ref{lem:DpmexpE}, $(g,H,A,\eta,\varphi_+)$ solves \eqref{eq:KSEqBi}. Conversely, given $(g,H,A,\eta,\varphi)$ solving \eqref{eq:KSEqBi}, we can construct a string algebroid $E_{P,H,A}$ as in Example \ref{def:E0}, endowed with the generalized metric \eqref{eq:Vpm}. Set
$$
\operatorname{div} = \operatorname{div}_0 - \la 2\varphi, \cdot \rangle,
$$
for $\operatorname{div}_0$ the Riemannian divergence \eqref{eq:div0}. Then $\varphi_+ := g(\pi (2\varphi)_+, \cdot) = \varphi$, and the result follows from Lemma \ref{lem:DpmexpE}.
\end{proof}

\begin{remark}
One can be more precise about which tuples $(g,H,A,\eta,\varphi)$ can occur as solutions of the Killing spinor equations on a fixed string algebroid. For instance, given $E_0 = E_{P,H_0,A_0}$ constructed as in Example \ref{def:E0}, a solution of \eqref{eq:killingE} on $E_0$ determines $(g,H,A,\eta,\varphi)$ and an isomorphism $E_0 \cong E_{P,H,A}$. Applying \cite[Proposition A.6]{grt2} (see also the proof of \cite[Proposition 3.4]{GF1}), one has
$$
H = H_0 + 2 \langle a \wedge F_{A_0} \rangle + \langle a \wedge d^{A_0}a \rangle + \tfrac{1}{3} \langle a \wedge [a \wedge a] \rangle + db,
$$
for some $b \in \Omega^2$, where $a = A_0 - A \in \Omega^1(\ad P)$.
\end{remark}

%\subsection{Twisted Hull--Strominger and coupled Hermitian-Yang--Mills}
\subsection{Canonical metrics and the F-term condition}\label{sec:THSFterm}

Let $M$ be a $2n$-dimensional spin manifold endowed with a principal $K$-bundle $P$. We assume that the Lie algebra $\mathfrak{k}$ is endowed with a non-degenerate bi-invariant symmetric bilinear form
$$
\left\langle\cdot,\cdot\right\rangle: \mathfrak{k} \otimes \mathfrak{k} \lto \mathbb{R}.
$$
Assume that (see Example \ref{def:E0})
$$
p_1(P) = 0 \in H^4_{dR}(M,\mathbb{R}).
$$

\begin{definition}\label{def:tHS}
We say that a triple $(\Psi,\omega,A)$, where $(\Psi,\omega)$ is an $\SU(n)$-structure on $M$ and $A$ is connection on $P$, satisfies the \emph{twisted Hull--Strominger system} on $(M,P)$ if
\begin{equation}\label{eq:tHS}
\begin{split}
	F_A^{0,2} = 0, \qquad F_A \wedge \omega^{n-1} & = 0,\\
	d \Psi - \theta_\omega \wedge \Psi & = 0,\\
	d \theta_\omega & = 0,\\
	dd^c \omega + \left\langle F_A \wedge F_A\right\rangle  & = 0.
	\end{split}
\end{equation}
\end{definition}

Notice that, by the proof of Lemma \ref{lem:holconstdetatlas}, the equations in the second and third line imply that the almost complex structure determined by $\Psi$ is integrable, that is, $N_J = 0$. An important motivation for the study of \eqref{eq:tHS} comes from the \emph{Hull--Strominger system} in heterotic string theory \cite{HullTurin,Strom}. Indeed, when the Lee form of a given solution of \eqref{eq:tHS} is exact, i.e., $\theta_\omega = d\phi$, there exists a global holomorphic volume on $(M,J)$,
$$
\Omega = e^{-\phi} \Psi,
$$
and $(\omega,A)$ solves the following system of equations
\begin{equation}\label{eq:HS}
\begin{split}
	F_A^{0,2} = 0, \qquad F_A \wedge \omega^{n-1} & = 0,\\
	d(\|\Omega\|_\omega \omega^{n-1}) & = 0,\\
	dd^c \omega + \left\langle F_A \wedge F_A\right\rangle  & = 0.
	\end{split}
\end{equation}
In complex dimension $3$ and for suitably chosen product bundle $P = P_g \times_M P'$, connection $A = \nabla \times A'$, and pairing $\langle ,\rangle$, one recovers the Hull--Strominger system in \cite{HullTurin,Strom}. Here, $P_g$ denotes the $\operatorname{Spin}(6)$-bundle of $(T,g)$ determined by the fixed spin structure. Notice that the equations \eqref{eq:HS} impose that $\nabla$ is Hermitian-Yang--Mills. This condition will be utterly important for our applications to vertex algebras in our main results (Theorem \ref{teo:mainTCA} and Theorem \ref{teo:mainTCAcHE}). In the sequel, we will refer to the abstract system \eqref{eq:HS} as the Hull--Strominger system.

Building on Proposition \ref{prop:KSEqevendim}, we establish next the relation between the twisted Hull--Strominger system \eqref{eq:tHS} and the Killing spinor equations on a string algebroid in Definition \ref{def:killingE}.

\begin{proposition}\label{prop:KSEqevendimCA}
Let $E = E_{P,H_0,A_0}$ be a string algebroid over $M$, as in Example \ref{def:E0}. Then the solutions $(V_+,\operatorname{div},\eta)$ of the Killing spinor equations \eqref{eq:killingE} on $E$, with $\operatorname{div}$ closed (see Definition \ref{def:compatibleE}) and $\eta \in \Omega^0(S^+)$ pure, are in one-to-one correspondence with tuples $(\Psi,\omega,A,b)$, where $(\Psi,\omega,A)$ is a solution of the twisted Hull--Strominger system \eqref{eq:tHS} on $(M,P)$ and $b \in \Omega^2$ satisfies
\begin{equation}\label{eq:HCS}
H = H_0 + 2 \langle a \wedge F_{A_0} \rangle + \langle a \wedge d^{A_0}a \rangle + \tfrac{1}{3} \langle a \wedge [a \wedge a] \rangle + db,
\end{equation}
where $a = A_0 - A \in \Omega^1(\ad P)$.
\end{proposition}

\begin{proof}
Given a solution of \eqref{eq:killingE} on $E$, the generalized metric $V_+ \subset E$ determines a Riemann metric $g$ on $M$ and an isotropic splitting $s \colon T \to E$ (see \cite[Proposition 3.4]{GF1}). The splitting $s$ is equivalent to a pair $(b,a) \in \Omega^2 \oplus \Omega^1(\ad P)$. Set $A = A_0 + a$ and define $H \in \Omega^3$ by \eqref{eq:HCS}. Then $(H,A)$ satisfies \eqref{eq:Bianchireal} and $s$ determines an isomorphism $E \cong E_{P,H,A}$ such that $V_+$ is as in \eqref{eq:Vpm}. The first part of the proof follows combining Lemma \ref{lem:KSEqSMCA} with Proposition \ref{prop:KSEqevendim}. For the converse, given $(\Psi,\omega,A,b)$ as in the statement, it follows from Proposition \ref{prop:KSEqevendim} and Lemma \ref{lem:KSEqSMCA} that $E_{P,-d^c\omega,A}$ carries a solution of the Killing spinor equations \eqref{eq:killingE} with closed divergence
$$
\operatorname{div} = \operatorname{div}_0 + \langle 2\theta_\omega,\cdot \rangle.
$$
Applying \cite[Proposition A.6]{grt2}, equation \eqref{eq:HCS} implies that $E_{P,-d^c\omega,A}$ is isomorphic to $E$, and the statement follows.
\end{proof}

Motivated by Proposition \ref{prop:KSEqevendimCA}, we want to understand in more detail the different conditions which arise from the Killing spinor equations in even dimensions, as in Proposition \ref{prop:KSEqevendim}, in terms of the geometry of the string algebroid. Our first goal is to interpret the integrability conditions in the second column of \eqref{eq:KSEqon2ndimman} as an involutivity condition. A similar analysis for the conditions in the first column will be undertaken in Section \ref{sec:Dterm}. This is indeed a very delicate question, related to the moment map construction in \cite{grt3}. We fix a string algebroid $E$ over our even-dimensional oriented spin manifold $M$. Given a generalized metric $V_+ \subset E$, a pure spinor line
$$
\langle \eta \rangle \subset \Omega^0(S^+)
$$
is equivalent to an orthogonal almost complex structure $J$ on $V_+$ (and hence on $T$) compatible with the orientation (see \cite[Chapter IV, Proposition 9.7]{MicLaw}). Equivalently, using that
$$
V_+^{1,0} = \{a_+ \in V_+ \otimes \C \; | \; a_+ \cdot \langle \eta \rangle = 0 \},
$$
the pure spinor line $\langle \eta \rangle$ is given by a decomposition
\begin{equation}\label{eq:lbarl}
\ell \oplus \overline{\ell} = V_+ \otimes \C
\end{equation}
such that $\langle \ell, \ell \rangle = \langle \overline \ell, \overline \ell \rangle = 0$ and $\overline \ell \cong \ell^*$ via $\langle,\rangle$. We will use the convention that $\ell = V_+^{1,0}$.

\begin{lemma}\label{lem:Ftermexp}
Let $E$ be a string algebroid over $M$. Let $(V_+,J)$ be a pair given by a generalized metric $V_+ \subset E$ and an almost complex structure $J$ on $M$. Then $J$ is orthogonal with respect to the metric $g$ determined by $V_+$ if and only if $\ell$ and $\overline{\ell}$ in \eqref{eq:lbarl} are isotropic, that is,
$$
\langle \ell, \ell \rangle = \langle \overline \ell, \overline \ell \rangle = 0.
$$
In this case, $\ell \oplus \overline{\ell}$ satisfies the F-term condition, that is,
\begin{equation}\label{eq:Fterm}
[\ell , \ell] \subset \ell, \qquad [\overline \ell , \overline \ell] \subset \overline \ell,
\end{equation}
if and only if the corresponding tuple $(g,H,A,J)$ satisfies
$$
N_J = 0, \qquad H + d^c\omega = 0, \qquad F_A^{0,2} = 0,
$$
where $\omega = g(J,)$.
\end{lemma}

\begin{proof}
The first part is straightforward and is left to the reader. Assuming orthogonality of $J$, via the isomorphism $E \cong E_{P,H,A}$ induced by $V_+$, we can write
$$
\overline \ell = e^{i \omega} T^{0,1}_J.
$$
Using the explicit formula for the bracket in Example \ref{def:E0}, for $X,Y \in \Omega^0(T^{0,1}_J)$, we have
\begin{equation*}
	[e^{i \omega} X, e^{i \omega} Y]_{H,A}  = e^{i\omega}[X,Y] - F_A(X,Y) + \iota_Y\iota_X (H + i d\omega).
\end{equation*}
Hence, $[\overline \ell , \overline \ell] \subset \overline \ell$ is equivalent to 
$$
[X,Y] \in  \Omega^0(T^{0,1}_J), \qquad H^{0,3 + 1,2} + i \dbar \omega = 0, \qquad F_A^{0,2} = 0.
$$
The statement follows from $H^{3,0 + 2,1} = \overline{H^{0,3 + 1,2}}$ and $[\overline \ell , \overline \ell] = \overline{[\ell,\ell]}$.
\end{proof}

%\begin{remark}
%Under the conditions of Lemma \ref{lem:Ftermexp}, the involutive subbundle $\overline \ell \subset E \otimes \C$ corresponds to a lifting of $T^{0,1}_J$ (see e.g. \cite{grt3}), and determines, via reduction, a holomorphic Courant algebroid over $(M,J)$. 
%\end{remark}

\begin{remark}
The name \emph{F-term} for the involutivity conditions \eqref{eq:Fterm} was already used in \cite[Definition 2.21]{AAGa} for analogue conditions in the study of Killing spinors on quadratic Lie algebras, and we shall adopt here the same nomenclature. This is motivated by similar equations appearing in supersymmetric field theories in theoretical physics.
\end{remark}

We pause for a moment our analysis of the Killing spinor equations in order to introduce a relaxed version of the twisted Hull--Strominger system which is relevant for Theorem \ref{teo:mainTCAcHE}. The \emph{coupled Hermitian-Einstein equation} arises from the study of canonical metrics on holomorphic Courant algebroids and has been recently studied by the third author jointly with Gonzalez Molina \cite{GFGM0}, inspired by \cite{OssaLarforsSvanes,GaJoSt}. For this purpose, we change gears and fix a complex manifold $(M,J)$ endowed with a string algebroid $E$, with underlying principal $K$-bundle $P$ and pairing $\left\langle\cdot,\cdot\right\rangle: \mathfrak{k} \otimes \mathfrak{k} \to \mathbb{R}$. We will consider generalized metrics $V_+ \subset E$ such that the associated decomposition \eqref{eq:lbarl} satisfies $\langle \ell, \ell \rangle = \langle \overline \ell, \overline \ell \rangle = 0$ and the F-term condition \eqref{eq:Fterm}. With these assumptions, one can obtain by reduction a string algebroid in the holomorphic category (see \cite{grt3}), defined by
$$
\mathcal{Q}_{\overline{\ell}} = \overline{\ell}^\perp/\overline{\ell}.
$$
In particular, $\mathcal{Q}_{\overline{\ell}}$ is a holomorphic orthogonal vector bundle over $(M,J)$ (endowed with a bracket on holomorphic sections and a holomorphic anchor map, satisfying holomorphic analogues of the axioms of Definition \ref{defn:CA}). As a smooth complex orthogonal vector bundle, it can be identified with
$$
\mathcal{Q}_{\overline{\ell}} \cong V_- \otimes \C,
$$
and therefore it canonically inherits a (possibly indefinite) Hermitian metric $\mathbf{G}$, given by
$$
\mathbf{G}(a_-,b_-) = -\langle a_-,\overline{b_-} \rangle,
$$
for $a_-,b_- \in V_- \otimes \C$. Since $\mathcal{Q}_{\overline{\ell}}$ has a holomorphic structure, $\mathbf{G}$ has an associated Chern connection, whose curvature we denote by $F_{\mathbf{G}}$.

\begin{definition}\label{def:cHE}
Let $E$ be string algebroid over an $n$-dimensional complex manifold $(M,J)$. Let $V_+ \subset E$ be a generalized metric such that the associated decomposition \eqref{eq:lbarl} satisfies $\langle \ell, \ell \rangle = \langle \overline \ell, \overline \ell \rangle = 0$ and the F-term condition \eqref{eq:Fterm}. Then, we say that $V_+$ satisfies the \emph{coupled Hermitian-Einstein equation} if
\begin{equation}\label{eq:cHE}
F_{\mathbf{G}} \wedge \omega^{n-1} = 0.
\end{equation}
\end{definition}

Observe that the generalized metric enters both in the definition of $\mathbf{G}$ and $\omega = g(J,)$ in \eqref{eq:cHE}, and therefore this equation has further coupling than the standard Hermitian-Einstein equation on holomorphic vector bundles. To see this more clearly, we next recall the explicit characterization of \eqref{eq:cHE} in terms of classical tensors obtained in \cite[Proposition 4.9]{GFGM0}. For the next result, given a connection $A$ on $P$, we define a smooth section $\Lambda_\omega F_A$ of the Lie-algebra bundle $(\ad P,[,])$ by the formula
\[
F_A \wedge \frac{\omega^{n-1}}{(n-1)!} = \Lambda_\omega F_A \frac{\omega^{n}}{n!}.
\]
We will denote by $d^A$ the corresponding covariant derivative on $\Omega^0(\ad P)$. 

\begin{proposition}[\cite{GFGM0}]\label{prop:cHECA}
Let $E = E_{P,H_0,A_0}$ be a string algebroid over an $n$-dimensional complex manifold $(M,J)$, as in Example \ref{def:E0}. Then there is a one-to-one correspondence between solutions $V_+ \subset E$ of \eqref{eq:cHE} and triples $(\omega,A,b)$, where $(\omega,A)$ is a solution of the coupled Hermitian-Yang--Mills system on $((M,J),P)$, defined by
\begin{equation}\label{eq:cHECA}
\begin{split}
	F_A^{0,2} = 0, \qquad [\Lambda_\omega F_A, ] & = 0,\\
	d^A(\Lambda_\omega F_A) & = 0,\\
	\rho_B + \left\langle \Lambda_\omega F_A,F_A\right\rangle & = 0,\\
	dd^c \omega + \left\langle F_A \wedge F_A\right\rangle  & = 0,	
	\end{split}
\end{equation}
for $\rho_B$ the Bismut Ricci form of the Hermitian form $\omega$, and $b \in \Omega^2$ satisfies \eqref{eq:HCS} with $a = A_0 - A \in \Omega^1(\ad P)$.
\end{proposition}

%Recall that the Bismut Ricci form is defined by $\rho_B = i F$, where $F$ is the curvature of the connection on the anti-canonical bundle $K^{-1}_M = \Lambda^n T^{1,0}$ induced by the Bismut connection $\nabla^B = \nabla^g - \frac12g^{-1}d^c\omega$. 
By Lemma \ref{lem:holonomia}, any solution of the twisted Hull--Strominger \eqref{eq:tHS} satisfies $\rho_B = 0$ and $\Lambda_\omega F_A = 0$, and therefore it automatically provides a solution of \eqref{eq:cHECA}. Solutions of the coupled Hermitian-Yang--Mills system \eqref{eq:cHECA} on a complex manifold $(M,J)$ which does not admit balanced metrics, and hence does not admit solutions of the Hull--Strominger system \eqref{eq:HS}, have been constructed recently by Gonzalez Molina in \cite{GM}.

\subsection{The D-term equation and special frames}\label{sec:Dterm}

In the light of Lemma \ref{lem:Ftermexp}, it is natural to ask whether the conditions in the first column of \eqref{eq:KSEqon2ndimman} can be characterized in terms of the string algebroid $E$. In this section we give a characterization of these equations using local frames, which shall be compared with the global moment map construction in \cite{grt3}. We will also provide a global characterization of the coupled Hermitian-Einstein equation \eqref{eq:cHE} in terms of $E$, provided that the manifold admits a global holomorphic volume form.

We start with a purely local setup: let $E$ be a string algebroid over an open set $U \subset \C^n$, where we fix holomorphic coordinates $(z_1, \ldots, z_n) \in \CC^n$. We consider the canonical holomorphic volume form
$$
\Omega_0 = dz_1 \wedge \ldots \wedge dz_n.
$$
Let $V_+ \subset E$ be a generalized metric such that the associated decomposition \eqref{eq:lbarl} satisfies $\langle \ell, \ell \rangle = \langle \overline \ell, \overline \ell \rangle = 0$ and the F-term condition \eqref{eq:Fterm}. Given a section $e \in \Omega^0(E \otimes \C)$, we will denote by
$$
e = e_\ell + e_{\overline \ell} + e_- \in \ell \oplus \overline{\ell} \oplus (V_- \otimes \C)
$$
the different pieces in this decomposition. Using the isomorphism $E \cong E_{-d^c\omega,A}$ induced by $V_+$, we construct local isotropic frames
\begin{equation}\label{eq:localisoframe}
\begin{split}
\epsilon_j & = e^{-i\omega}g^{-1}d\overline{z}_j  = g^{-1}d\overline{z}_j + d\overline{z}_j \in \ell,\\
\overline \epsilon_j & = e^{i\omega}\dbar_j = \dbar_j + g \dbar_j  \in \overline \ell,
\end{split}
\end{equation}
where we use the compact notation $\dbar_j : = \frac{\partial}{\partial\overline{z}_j}$. The key to our developments will be the quantity
\begin{equation}\label{eq:mu}
\mu = \frac{1}{2}\sum_{j=1}^n\left[\overline{\epsilon}_j,\epsilon_j\right] \in \Omega^0(E \otimes \C).
\end{equation}
To give an explicit formula for $\mu$, we need the following structural property of our frames.

\begin{lemma}\label{lem:frameprop}
The frames in \eqref{eq:localisoframe} satisfy, for all $j,k = 1, \ldots, n$, that
\begin{enumerate}

\item[\textup{(1)}] $\langle \epsilon_j, \overline \epsilon_k \rangle  = \delta_{jk}$, $\langle \epsilon_j, \epsilon_k \rangle  = \langle \overline \epsilon_j, \overline \epsilon_k \rangle = 0$,

\item[\textup{(2)}] $[\overline \epsilon_j, \overline \epsilon_k ] = 0$.

\end{enumerate}
	
\end{lemma}

\begin{proof}
The isotropic condition (1) is straightforward using the formula for the pairing \eqref{eq:pairingE0}, and is left to the reader. To check (2), we argue as in the proof of Lemma \ref{lem:Ftermexp}:
\begin{equation*}
[\overline \epsilon_j, \overline \epsilon_k] =	[e^{i \omega} \dbar_j, e^{i \omega} \dbar_k]  = e^{i\omega}[\dbar_j, \dbar_k]- F_A(\dbar_j, \dbar_k) + \iota_{\dbar_k}\iota_{\dbar_j} (-d^c\omega + i d\omega) = 0,
\end{equation*}
where we have used the F-term condition combined with Lemma \ref{lem:Ftermexp}.
\end{proof}

We prove next the main technical lemma of this section.

\begin{lemma}\label{lem:braketsum0}
With the hypotheses above, define a complex one-form by
$$
\varphi_\omega = - d \log \|\Omega_0\|_\omega - i \left(d^*\omega - d^c\log \|\Omega_0\|_\omega \right) \in \Omega^1(\C).
$$
Then $\mu$ in \eqref{eq:mu} satisfies
\begin{equation}
\label{eq:globalsumlocal}
2\mu = \sigma_+\left(g^{-1}\varphi_\omega^{1,0}\right) - \sigma_-\left(g^{-1}\varphi_\omega\right)+i\Lambda_\omega F_A,
\end{equation}
where $\sigma_\pm \colon T \to V_\pm: X \mapsto X \pm gX$.
\end{lemma}

\begin{proof}
Consider the decomposition
$$
\mu = \mu_\ell + \mu_{\overline \ell} + \mu_- \in \ell \oplus \overline{\ell} \oplus (V_- \otimes \C)|_U.
$$
We will compute these three terms separately, using Lemma \ref{lem:frameprop}. Given $w := \sigma_-(X) + r \in \Omega^0(V_- \otimes \C)$, by \cite[Equation (5.10)]{grt} we have 
\begin{equation}\label{eq:clavesuma}
[w,\sigma_+(Y)]_+ = \sigma_+ (\nabla^+_XY - g^{-1} \langle F_A ,r \rangle).
\end{equation}
Thus, using the Courant algebroid axiom $(5)$ in Definition \ref{defn:CA},
	\begin{equation*}
	\begin{split}
	\left\langle 2\mu,w\right\rangle & = \sum_{j=1}^n\left\langle \epsilon_j,\left[w,\sigma_+\left(\overline{\partial}_j\right)\right]\right\rangle\\
	& = \sum_{j=1}^n\left\langle g^{-1}d\overline{z}_j + d\overline{z}_j,\sigma_+\left(\nabla_X^+\overline{\partial}_j -g^{-1}\left\langle\iota_{\overline{\partial}_j}F_A,r\right\rangle\right)\right\rangle\\
%	&= \sum_{j=1}^nd\overline{z}_j\left(\nabla_X^+\overline{\partial}_j\right) - g\left(g^{-1 \left\langle\iota_{\overline{\partial}_j}F_A,r \right\rangle,g^{-1}d\overline{z}_j\right),
& = \Delta - \sum_{j=1}^n\left\langle F_A\left(\overline{\partial}_j,g^{-1}d\overline{z}_j\right),r\right\rangle,\\
%& = \Delta + \left\langle i \Lambda_\omega F_A,r\right\rangle
\end{split}
\end{equation*}
where
$$
\Delta := \sum_{j=1}^nd\overline{z}_j\left(\nabla_X^+\overline{\partial}_j\right).
$$
Note that, writing $\omega = \tfrac{i}{2}\sum_{l \leqslant j}g_{l\overline{j}}dz_j \wedge d \overline{z}_j$ and $F_A = \sum_{l \leqslant j}F_{l\overline{j}}dz_j \wedge d \overline{z}_j$, we have
$$
F_A\left(\overline{\partial}_j,g^{-1}d\overline{z}_j\right) = - 2 g^{l\overline{j}}F_{l\overline{j}} = - i \Lambda_\omega F_A.
$$
On the other hand, since $\ell \oplus \overline{\ell}$ satisfies the F-term condition \eqref{eq:Fterm}, by Lemma \ref{lem:Ftermexp} we have $\nabla^+ = \nabla^B$, the Bismut connection of the Hermitian metric $g$. Then, setting $\partial_j = \frac{\partial}{\partial z_j}$, 
\begin{equation*}
\overline{\Delta} = \sum_{j=1}^ndz_j\left(\nabla_{\overline X}^B\partial_j\right).
\end{equation*}
Let $\nabla^{B,K_M^{-1}}$ be the connection on $K_M^{-1} := \Lambda^nT^{1,0}$ induced by $\nabla^B$, given in holomorphic coordinates by
\begin{equation*}
\nabla^{B,K_M^{-1}} = d + \sum_{i,j=1}^n\left(\Gamma^i_{ji}dz_j+ \Gamma^i_{\overline{j}i}d\overline{z}_j\right), 
\end{equation*}
where the \textit{Christoffel symbols} $\Gamma^i_{jk}, \Gamma^i_{\overline{j}k}$ of $\nabla^B$ are defined by
\begin{equation*}
\sum_{i=1}^n\Gamma_{jk}^i\partial_i := \nabla^{B}_{\partial_j}\left(\partial_k\right), \qquad \sum_{i=1}^n\Gamma_{\overline{j}k}^i\partial_i := \nabla^{B}_{\overline{\partial}_j}\left(\partial_k \right).
\end{equation*}
Hence, we obtain that
\begin{equation*}
\nabla^{B,K_M^{-1}}_{\overline{X}} = \iota_{\overline X} d + \overline{\Delta}.
\end{equation*}
The connection on $K_M^{-1}$ induced by the Chern connection of $g$ is given by
$$
\nabla^{C,K_M^{-1}} = d - 2 \partial \log \|\Omega_0\|_\omega.
$$
Applying now Gauduchon's formula \cite[Equation (2.7.6)]{Gau}, we have that
\begin{equation*}
\nabla^{B,K_M^{-1}} = \nabla^{C,K_M^{-1}} + i d^*\omega  = d - d \log \|\Omega_0\|_\omega + i(d^*\omega - d^c\log \|\Omega_0\|_\omega) = d + \overline{\varphi_\omega},
\end{equation*}
and we conclude
$$
\Delta = \overline{\overline{\varphi_\omega}(\overline{X})} = \varphi_\omega(X).
$$
Collecting the different terms, we obtain the following
$$
2\mu_- = - \sigma_-(g^{-1} \varphi_\omega) + i \Lambda_\omega F_A.
$$
As for $\mu_\ell$, we note first that, by Lemma \ref{lem:frameprop}, for all $k$ we have
\begin{equation*}
\left\langle 2\mu , \overline{\epsilon}_k\right\rangle = \sum_{j=1}^n \left\langle \epsilon_j,\left[\overline \epsilon_k,\overline \epsilon_j\right]\right\rangle = 0,
\end{equation*}
and hence $\mu_\ell = 0$. Finally, using that
$$
\epsilon_k = 2d\overline{z}_k-d\overline{z}_k+g^{-1}d\overline{z}_k = 2d\overline{z}_k+\sigma_-\left(g^{-1}d\overline{z}_k\right),
$$
applying the Courant algebroid axiom $(5)$ in Definition \ref{defn:CA}, combined with the explicit formula for the Dorfman bracket \eqref{eq:bracketE0} and equation \eqref{eq:clavesuma}, since $d\overline{z}_{k}$ is exact,
\begin{equation*}
\begin{split}
\left\langle 2\mu,\epsilon_k\right\rangle & = \sum_{j=1}^n \left\langle \epsilon_j,\left[\epsilon_k,\overline{\epsilon}_j\right] \right\rangle\\
&  = \sum_{j=1}^n \left\langle \epsilon_j,\left[\sigma_-\left(g^{-1}d\overline{z}_k\right),\overline{\epsilon}_j\right]\right\rangle\\
& = \sum_{j=1}^nd\overline{z}_j\left(\nabla^B_{g^{-1}d\overline{z}_k}\overline{\partial}_j\right)\\
& = \varphi^{1,0}_\omega\left(g^{-1}d\overline{z}_k\right).
\end{split}
\end{equation*}
Similarly as before, we conclude that $2\mu_{\overline{\ell}} = \sigma_+(g^{-1}\varphi^{1,0}_\omega)$, which completes the proof.
\end{proof}

We describe next the global picture, regarding the twisted Hull--Strominger system \eqref{eq:tHS}. Let $E$ be a string algebroid over an even-dimensional oriented spin manifold $M$. We fix a pair $(V_+,\eta)$, given by a generalized metric on $E$ and a pure spinor $\eta \in \Omega^0(S^+)$, such that the corresponding decomposition \eqref{eq:lbarl} satisfies the F-term condition \eqref{eq:Fterm}. Via the identification $E = E_{-d^c\omega,A}$, we obtain a triple $(\Psi,\omega,A)$ as in Proposition \ref{prop:KSEqevendimCA}, where $(\Psi,\omega)$ is the $\SU(n)$-structure associated to $\eta$.
%By Lemma, \eqref{lem:holonomia}, this implies that $\nabla^B \Psi = 0$. Hence, in the setup of Lemma \ref{lem:Ftermexp}, we have a solution to the conditions in the second column of \eqref{eq:KSEqon2ndimman} (see Lemma \ref{lem:Ftermexp}), which also solves the middle equation in the first column. Under these assumptions, we want to understand when the full system \eqref{eq:KSEqon2ndimman} is satisfied in terms of the geometry of $E$ or, equivalently, when the associated tuple $(\Psi,\omega,A)$ is a solution of the twisted Hull--Strominger \eqref{eq:tHS} (see Proposition \ref{prop:KSEqevendimCA}).
Under these assumptions, we want to understand when $(\Psi,\omega,A)$ is a solution of the twisted Hull--Strominger \eqref{eq:tHS} in terms of the geometry of $E$.

\begin{lemma}\label{lem:braketsum}
Let $E$ be a string algebroid over an even-dimensional oriented spin manifold $M$. Consider a pair $(V_+,\eta)$, given by a generalized metric $V_+$ and a pure spinor $\eta \in \Omega^0(S^+)$. Assume that the corresponding decomposition \eqref{eq:lbarl} satisfies the F-term condition \eqref{eq:Fterm}, and furthermore that the associated $\SU(n)$-structure $(\Psi,\omega)$ satisfies
\begin{equation}\label{eq:confbalbis}
d\Psi-\theta_\omega\wedge\Psi = 0, \qquad d\theta_\omega = 0.
\end{equation}
Then, in the local holomorphic coordinates defined in Lemma \ref{lem:holconstdetatlas}, one has
\begin{equation}
\label{eq:globalsum}
2\mu = \sigma_+\left(g^{-1}\theta_\omega^{1,0}\right) - \sigma_-\left(g^{-1}\theta_\omega\right)+i\Lambda_\omega F_A,
\end{equation}
and therefore \eqref{eq:mu} defines a global section $\mu \in \Omega^0(E \otimes \C)$. 
\end{lemma}

\begin{proof}
The statement follows from Lemma \ref{lem:braketsum0}, by observing that
$$
\|\Omega_0\|_\omega = e^{-\phi_U}, \qquad \theta_\omega = d\phi_U,
$$
in the holomorphic coordinates in Lemma \ref{lem:holconstdetatlas}, and therefore
\[
\varphi_\omega = \theta_\omega - i (- J \theta_\omega - J d\log \|\Omega_0\|_\omega ) = \theta_\omega.
\qedhere
\]
\end{proof}

\begin{remark}
An abstract proof of the independence of $\mu$ on the choice of local coordinates in the atlas in Lemma \ref{lem:holconstdetatlas} has been given in \cite[Lemma 6.3.9]{Arriba}, using that the holomorphic Jacobian of the change of frame has constant determinant.
\end{remark}

We provide the desired characterization of the twisted Hull--Strominger \eqref{eq:tHS} in the following result. The name D-term equation below is again borrowed from \cite[Definition 2.21]{AAGa}, and inspired by similar conditions in theoretical physics.

\begin{proposition}\label{prop:DtermTHS}
Let $E$ be a string algebroid over an even-dimensional oriented spin manifold $M$. Consider a pair $(V_+,\eta)$, given by a generalized metric $V_+$ and a pure spinor $\eta \in \Omega^0(S^+)$. Assume that the corresponding decomposition \eqref{eq:lbarl} satisfies the F-term condition \eqref{eq:Fterm} and furthermore that the associated $\SU(n)$-structure $(\Psi,\omega)$ satisfies \eqref{eq:confbalbis}.
%Set $\varepsilon = - \theta_\omega \in \Omega^0(E)$. %and define a divergence operator on $E$ by
%$$
%\operatorname{div} = \operatorname{div}_0 - \langle \varepsilon, \cdot \rangle.
%$$
%Then, $(V_+,\operatorname{div},\eta)$ solves the Killing spinor equations \eqref{eq:killingE} 
Then $(\Psi,\omega,A)$ solves the twisted Hull--Strominger system \eqref{eq:tHS} if and only if there exists a complex one-form $\varepsilon \in \Omega^1(\C) \subset \Omega^0(E \otimes \C)$ such that $(V_+,\varepsilon)$ is a solution of the D-term equation, that is,
\begin{equation}\label{eq:Dterm}
\mu = \varepsilon_{\ell} - \varepsilon.
\end{equation}
Furthermore, if this is the case, then we have $\varepsilon = - \theta_\omega$ and
$$
[\varepsilon,] = 0, \qquad \langle \varepsilon , \varepsilon \rangle = 0. 
$$
%Consequently, with the previous assumptions, the corresponding tuple $(\Psi,\omega,A)$ solves the twisted Hull--Strominger system \eqref{eq:tHS} if and only if \eqref{eq:Dterm} holds.
\end{proposition}

\begin{proof}
With the given hypotheses, $(\Psi,\omega,A)$ solves the twisted Hull--Strominger system \eqref{eq:tHS} if and only $\Lambda_\omega F_A = 0$. Given $\varepsilon \in \Omega^1(\C)$, the decomposition $\varepsilon = \varepsilon_\ell + \varepsilon_{\overline{\ell}} + \varepsilon_-$ is given by
\begin{align*}
\varepsilon_\ell & = \frac{1}{2}(g^{-1}\varepsilon^{0,1} + \varepsilon^{0,1}) = \frac{1}{2} \sigma_+ (g^{-1}\varepsilon^{0,1}),\\
\varepsilon_{\overline{\ell}} & = \frac{1}{2}(g^{-1}\varepsilon^{1,0} + \varepsilon^{1,0}) = \frac{1}{2} \sigma_+ (g^{-1}\varepsilon^{1,0}),\\
\varepsilon_- & = \frac{1}{2}(- g^{-1}\varepsilon + \varepsilon) = - \frac{1}{2} \sigma_- (g^{-1}\varepsilon).
\end{align*}
Hence, we have
$$
\varepsilon_\ell - \varepsilon = - \frac{1}{2}\sigma_+ (g^{-1}\varepsilon^{1,0}) + \frac{1}{2}\sigma_- (g^{-1}\varepsilon),
$$
and therefore, by Lemma \ref{lem:braketsum}, equation \eqref{eq:Dterm} is equivalent to $\varepsilon = - \theta_\omega$ and $\Lambda_\omega F_A = 0$. The last part of the statement follows from $d \theta_\omega = 0$, by the formulae for the bracket and the pairing in Example \ref{def:E0}.
\end{proof}

\begin{remark}
Ideally, we would have liked to provide a characterization of the twisted Hull--Strominger system \eqref{eq:tHS} without imposing condition \eqref{eq:confbalbis}. However, without this assumption, the expression \eqref{eq:mu} does not necessarily define a global section of $E \otimes \C$.
\end{remark}

To finish this section, we provide a similar characterization of the coupled Hermitian-Einstein equation \eqref{eq:cHE} via the D-term equation \eqref{eq:Dterm}, assuming that the complex manifold $(M,J)$ admits a holomorphic volume form $\Omega$. We fix a string algebroid $E$ over $(M,J)$ and consider generalized metrics $V_+ \subset E$ such that the associated decomposition \eqref{eq:lbarl} satisfies $\langle \ell, \ell \rangle = \langle \overline \ell, \overline \ell \rangle = 0$ and the F-term condition \eqref{eq:Fterm}. In this setup, Lemma \ref{lem:braketsum0} gives a global section
$$
\mu \in \Omega^0(E \otimes \C).
$$
%and we can define $\varepsilon \in \Omega^0(E \otimes \C)$ through the D-term equation \eqref{eq:Dterm}, that is,
%\begin{equation}\label{eq:varepsilonDterm}
%\varepsilon = d \log \|\Omega\|_\omega + i \left(d^*\omega - d^c\log \|\Omega\|_\omega \right) - \frac{i}{2} %\Lambda_\omega F_A \in \Omega^0(E \otimes \C).
%\end{equation}

\begin{proposition}\label{prop:DtermCHE}
Let $E$ be a string algebroid over an $n$-dimensional complex manifold $(M,J)$ endowed with a holomorphic volume form $\Omega$.  Let $V_+ \subset E$ be a generalized metric such that the associated decomposition \eqref{eq:lbarl} satisfies $\langle \ell, \ell \rangle = \langle \overline \ell, \overline \ell \rangle = 0$ and the F-term condition \eqref{eq:Fterm}. Then the Hermitian metric $\mathbf{G}$ on the holomorphic bundle $\mathcal{Q}_{\overline{\ell}} : = \overline{\ell}^\perp/\overline{\ell}$ associated to $V_+$ satisfies (see Definition \ref{def:cHE})
$$
F_{\mathbf{G}} \wedge \omega^{n-1} = 0
$$
if and only if there exists $\varepsilon \in \Omega^0(E \otimes \C)$ that is an infinitesimal symmetry of $E \otimes \C$, i.e., 
$$
[\varepsilon,] = 0,
$$
such that $(V_+,\varepsilon)$ is a solution of the D-term equation \eqref{eq:Dterm}. Furthermore, if this is the case, then we have
\begin{equation}\label{eq:varepsilonDterm}
\begin{split}&
\varepsilon = d \log \|\Omega\|_\omega + i \left(d^*\omega - d^c\log \|\Omega\|_\omega \right) - \frac{i}{2} \Lambda_\omega F_A \in \Omega^0(E \otimes \C),\\&
\langle \varepsilon,\varepsilon \rangle = -\frac{1}{4} \langle \Lambda_\omega F_A ,\Lambda_\omega F_A \rangle \in \R.
\end{split}
\end{equation}
\end{proposition}

\begin{proof}
Let $\varepsilon = X + z + \varphi$ be an arbitrary section of $E \otimes \C$. Then, by the explicit formula for the Dorfman bracket in Example \ref{def:E0}, $[\varepsilon,]= 0$ is equivalent to
\begin{equation}\label{eq:infsymm}
X = 0, \qquad d^A z = 0, \qquad [z,\cdot] = 0, \qquad d\varphi + 2 \langle F_A,z \rangle = 0.
\end{equation}
Furthermore, for $\varepsilon = z + \varphi$, using Lemma \ref{lem:braketsum0} and arguing as in the proof of Proposition \ref{prop:DtermTHS}, the D-term equation \eqref{eq:Dterm} is equivalent to
$$
- \sigma_+ (g^{-1}\varphi^{1,0}) + \sigma_- (g^{-1}\varphi) - 2z = \sigma_+\left(g^{-1}\varphi_\omega^{1,0}\right) - \sigma_-\left(g^{-1}\varphi_\omega\right)+i\Lambda_\omega F_A,
$$
or, equivalently, to
$$
\varphi = - \varphi_\omega, \qquad z = - \frac{i}{2} \Lambda_\omega F_A.
$$
Assume now that $F_\mathbf{G} \wedge \omega^{n-1} = 0$ and set $\varepsilon = - \varphi_\omega - \tfrac{i}{2}\Lambda_\omega F_A$, as in \eqref{eq:varepsilonDterm}. Then, by Proposition \ref{prop:cHECA}, we have
$$
[\Lambda_\omega F_A,] = 0, \qquad d^A(\Lambda_\omega F_A) = 0.
$$
Furthermore, by the proof of Lemma \ref{lem:braketsum0}, we have
$$
d \varphi_\omega = -id\left(d^*\omega - d^c\log \|\Omega\|_\omega \right) = i \rho_B,
$$
which gives
$$
d \varphi_\omega + i\langle F_A, \Lambda_\omega F_A \rangle = 0,
$$
and we conclude that $[\varepsilon,]= 0$ by \eqref{eq:infsymm}. Conversely, if $\varepsilon \in \Omega^0(E \otimes \C)$ is an infinitesimal symmetry satisfying the D-term equation \eqref{eq:Dterm}, then $\varepsilon = - \varphi_\omega - \tfrac{i}{2}\Lambda_\omega F_A$ as in \eqref{eq:varepsilonDterm} and equation \eqref{eq:infsymm} implies that $(\omega,A)$ satisfies the coupled Hermitian-Yang--Mills equations \eqref{eq:cHECA}. The condition $F_\mathbf{G} \wedge \omega^{n-1} = 0$ follows now from Proposition \ref{prop:cHECA}. Finally, observe that $[\varepsilon,] = 0$ implies that
$$
\mathcal{D}\langle \varepsilon,\varepsilon\rangle = 2 [\varepsilon,\varepsilon] = 0,
$$
and hence, by definition of $\varepsilon$, we have $\langle \varepsilon,\varepsilon\rangle = -\frac{1}{4} \langle \Lambda_\omega F_A ,\Lambda_\omega F_A \rangle \in \R$, as claimed.
%%%%
% OLD PROOF, DIFFERENT STATEMENT
%The first part follows arguing as in the proof of Proposition \ref{prop:DtermTHS}. As the for the second part, we set $\varepsilon = - \varphi_\omega - \tfrac{i}{2}\Lambda_\omega F_A$. By the explicit formula for the Dorfman bracket in Example \ref{def:E0}, for $Y + t + \eta \in \Omega^0(E \otimes \C)$ we have
%\begin{align*}
%[\varepsilon,Y] & = \iota_Y\left(\tfrac{i}{2}d^A(\Lambda_\omega F_A) + d \varphi_\omega + i\langle F_A, \Lambda_\omega F_A \rangle\right),\\
%[\varepsilon,t] & = \tfrac{i}{2} [\Lambda_\omega F_A,t] - i\langle d^A (\Lambda_\omega F_A), t \rangle,\\
%[\varepsilon,\eta] & = 0,
%\end{align*}
%and therefore $[\varepsilon,] = 0$ is equivalent to
%$$
%[\Lambda_\omega F_A,] = 0, \qquad d^A(\Lambda_\omega F_A) = 0,\qquad d \varphi_\omega + i\langle F_A, \Lambda_\omega F_A %\rangle = 0.
%$$
%Now, by the proof of Lemma \ref{lem:braketsum0}, we have
%$$
%d \varphi_\omega = -id\left(d^*\omega - d^c\log \|\Omega\|_\omega \right) = i \rho_B
%$$
%and hence the second part of the statement follows from Proposition \ref{prop:cHECA}.
\end{proof}

\section{Embeddings of $N=2$ superconformal vertex algebras}\label{Resultados1}

\subsection{Background on SUSY vertex algebras}\label{ssec:background}

In this section, we review some basic notions and examples of the theory of supersymmetric vertex algebras that we will need. Our treatment will follow the superfield formalism introduced by Heluani--Kac~\cite{SUSYVA}, where details can be found. We are concerned only with $N_K=1$ SUSY vertex algebras within the formalism in~\cite[\S 4]{SUSYVA}, so we will refer to this special case simply as `SUSY vertex algebras'.

We work over an algebraically closed field $\CC$ of characteristic $0$. As usual, the adjective `super' applied to algebras, modules and vector spaces will mean $\ZZ/2\ZZ$-graded. Considering the pair $Z = (z,\theta)$ consisting of an even indeterminate $z$ and a Grassmannian indeterminate $\theta$ (odd, commuting with $z$, and such that $\theta^2 = 0$), we use the notation
$$
Z^{j|J} := z^j\theta^J, \qquad \textrm{ for $j \in \Z$ and $J \in \{0,1\}$}.
$$
In the superfield approach, a SUSY vertex algebra consists of a vector superspace $V$, equipped with a non-zero even vector $\left|0\right\rangle\in V$ (vacuum), an odd endomorphism $S\colon V\to V$ (supersymmetry generator) and an even linear map $Y\colon V\to(\End V)[[z^{\pm 1}]][\theta]$ (the state-superfield correspondence) mapping each vector $a\in V$ to a superfield, i.e., a formal sum
\[
Y(a,Z)=\sum_{j \in \Z \atop J=0,1} Z^{-1-j|1-J}a_{(j|J)},
\]
with `Fourier supermodes' $a_{(j|J)}\in\End V$, such that $Y(a,Z)b$ is a formal Laurent series for all $b\in V$. This data should satisfy several conditions called vacuum axioms, translation invariance
\begin{equation*}\label{eq:enhanced-translation-invariance}
[S,Y(a,Z)]=(\partial_{\theta}-\theta\partial_z)Y(a,Z),
\end{equation*}
and locality (see \cite{SUSYVA} for details).

Extending earlier work by Bakalov--Kac~\cite{BK03,Kac} for vertex algebras, Heluani--Kac~\cite{SUSYVA} obtained an equivalent definition of SUSY vertex algebra, % in terms of the so-called SUSY Lie conformal algebras,
which is the one we will use for our computations. 
To outline this definition, we introduce the even endomorphism $T := S^2$, so $V$ is a supermodule over the translation algebra $\cH$, defined as the (unital) associative superalgebra with an odd generator $S$, an even generator $T$, and the relation $S^2=T$. This superalgebra can be identified with the parameter algebra, defined as the associative superalgebra $\cL$ with an odd generator $-\chi$, an even generator $-\lambda$, and the relation $\chi^2=-\lambda$. The two pairs of generators are denoted $\nabla=(T,S)$ and $\Lambda=(\lambda,\chi)$. Using the state-field correspondence, the \emph{$\Lambda$-bracket} and the \emph{normally ordered product} of two states $a,b\in V$ are defined by
\begin{equation*}\label{eq:Lambda-bracket-nop}
\left[a_\Lambda b\right] = \displaystyle\sum_{j\in\N \atop J=0,1} \frac{\Lambda^{j|J}}{j!}a_{(j|J)}b,
\qquad
:ab:=a_{(-1|1)}b,
\end{equation*}
respectively, where $\Lambda^{j|J}=\lambda^j\chi^J$.

The properties of the $\Lambda$-bracket motivate the definition~\cite[Def. 4.10]{SUSYVA} of a \emph{SUSY Lie conformal algebra} as an $\cH$-supermodule $\cR$ equipped with a parity-reversing bilinear map
\begin{equation*}\label{eq:Lambda-bracket.1}
\cR\times \cR\lto \cL\otimes \cR,\quad (a,b)\longmapsto [a_{\Lambda}b],
\end{equation*}
that satisfies the following identities for all $a,b,c\in\cR$:
\begin{gather}\label{eq:sesquiLambda}%\label{eq:Lambda-bracket.2.a}
\left[Sa_\Lambda b\right] = \chi\left[a_\Lambda b\right],\qquad
\left[a_\Lambda Sb\right] = -(-1)^{\left|a\right|} \left(S+\chi\right)\left[a_\Lambda b\right],  
\\\label{eq:comLambda}%{eq:Lambda-bracket.2.b}
\left[a_\Lambda b\right] = (-1)^{\left|a\right|\left|b\right|}\left[b_{-\Lambda-\nabla}a\right],
\\\label{eq:JacobiLambda}%{eq:Lambda-bracket.2.c}
\left[a_\Lambda\left[b_\Gamma c\right]\right] = (-1)^{\left|a\right|+1}\left[\left[a_\Lambda b\right]_{\Lambda+\Gamma}c\right] + (-1)^{\left(\left|a\right|+1\right)\left(\left|b\right|+1\right)}\left[b_\Gamma\left[a_\Lambda c\right]\right].
\end{gather}
The identities~\eqref{eq:sesquiLambda} (sesquilinearity) and~\eqref{eq:comLambda} (skew-symmetry or commutativity) take place in $\cL\otimes \cR$, while the identity~\eqref{eq:JacobiLambda} (the Jacobi identity) takes place in $\cL\otimes\cL'\otimes\cR$, where $\cL'$ is a copy of $\cL$ with the formal variables $\Lambda=(\lambda,\chi)$ replaced by $\Gamma=(\gamma,\eta)$.
The degree of a homogeneous element $a$ is denoted $\left|a\right|\in\ZZ/2\ZZ$. 
The precise meaning of the expressions in~\eqref{eq:sesquiLambda},~\eqref{eq:comLambda},~\eqref{eq:JacobiLambda} is explained, with our conventions, in \cite[Appendix A.1]{AAGa}.

Adding the data given by the normally ordered product, one obtains the following~\cite[Def. 4.19]{SUSYVA}.
A \emph{SUSY vertex algebra} is equivalent to a SUSY Lie conformal algebra $V$, whose underlying vector superspace is equipped with a structure of differential superalgebra, where the (odd) differential $S\colon V\to V$ is induced by the $\cH$-module structure on $V$, that is unital with identity $\left|0\right\rangle\in V$, and such that the multiplication of this superalgebra, 
\begin{equation}\label{eq:VA-equiv.1}
V\times V\lto V,\quad (a,b)\longmapsto :ab:,
\end{equation}
called normally ordered product, satisfies the following identities, respectively called quasicommutativity, quasiassociativity and the non-commutative Wick formula:
%\begin{subequations}\label{eq:VA-equiv.2}
\begin{gather}\label{eq:cuasicon}%\label{eq:VA-equiv.2.a}
:ab:-(-1)^{\left|a\right|\left|b\right|}:ba: = \int_{-\nabla}^0 d\Lambda \left[a_\Lambda b\right],
\\\label{eq:cuasiaso}%\label{eq:VA-equiv.2.b}
::ab:c:-:a:bc:: = :\!\left(\int_0^\nabla d^r\!\Lambda\, a\right) \left[b_\Lambda c\right]\!:+(-1)^{\left|a\right|\left|b\right|} :\!\left(\int_0^\nabla d^r\!\Lambda\, b \right)\left[a_\Lambda c\right]\!:,
\\\label{eq:Wick}%\label{eq:VA-equiv.2.c}
\left[a_\Lambda:bc:\right] = :\left[a_\Lambda b\right]c:+(-1)^{\left(\left|a\right|+1\right)\left|b\right|} :b\left[a_\Lambda c\right]:+\int_0^\Lambda d\Gamma \left[\left[a_\Lambda b\right]_\Gamma c\right].
\end{gather}%\end{subequations}
The meaning of the integrals in these identities %(including the symbol $d^r\!\Lambda$)
is explained, with our conventions, in \cite[Appendix A.2]{AAGa}.

As in the theory of vertex algebras, a SUSY Lie conformal algebra $\cR$ determines a SUSY vertex algebra $V(\cR)$ and an embedding of SUSY Lie conformal algebras $\cR\hra V(\cR)$ that is universal for morphisms $\cR\to V'$ into other SUSY vertex algebras $V'$.
One says that $V(\cR)$ is the SUSY vertex algebra \emph{generated} by $\cR$~\cite[\S 1.8]{SUSYVA}, or more precisely, $V(\cR)$ is the \emph{universal enveloping SUSY vertex algebra} of $\cR$~\cite[Theorem 3.4.2 and \S 4.21]{SUSYVA}. 
%
%Given a scalar $c\in\C$ and an element $C\in\cR$ which is central and invariant (i.e., $\left[C_\Lambda\cR\right]=0$ and $SC=0$), one often considers the quotient $V^c(\cR)$ of $V(\cR)$ by the ideal generated by the relation $C=c\left|0\right\rangle$.

Next we describe two SUSY vertex algebras relevant to this paper, referring to \cite[Section 3.1]{AAGa} and~\cite{SUSYVA} for details on these examples in the classical approach to vertex algebras.

\begin{example}\label{exam:NS}
Let $\cR$ be the SUSY Lie conformal algebra  whose underlying $\cH$-module is freely generated by an odd vector $H$ and a scalar $c$, with $\Lambda$-bracket
\begin{equation}\label{eq:NS}
\left[{H}_\Lambda H\right] = \left(2T+\chi S+3\lambda\right)H+\frac{\chi\lambda^2}{3}c.
\end{equation}
The $N=1$ superconformal vertex algebra of central charge $c$, also called the \emph{Neveu--Schwarz vertex algebra}, is the corresponding universal enveloping SUSY vertex algebra.
\end{example}

\begin{example}\label{exam:N2}
Let $\cR$ be the SUSY Lie conformal algebra whose underlying $\cH$-module is freely generated by an odd vector $H$ (a `Neveu--Schwarz vector'), an even vector $J$ (a `current') and a scalar $c$, with $\Lambda$-brackets given by~\eqref{eq:NS} and 
\begin{equation}\label{eq:N2SCVA}
\left[{J}_\Lambda J\right] = -\left(H+\frac{\lambda\chi}{3}c\right),
\quad
\left[{H}_\Lambda J\right] = \left(2T+2\lambda+\chi S\right)J.
\end{equation}
The corresponding universal enveloping SUSY vertex algebra is called the \emph{$N=2$ superconformal vertex algebra of central charge $c$}.
\end{example}

\begin{remark}\label{rem:tech}
The Neveu--Schwarz vector $H$ can be recovered from the current $J$.
More precisely, given a SUSY vertex algebra $V$ and an even vector $J\in V$ satisfying the identities~\eqref{eq:N2SCVA} for some odd vector $H\in V$ and an invariant central element $c$, $H$ automatically satisfies~\eqref{eq:NS}, so $J$ and $H$ generate an $N=2$ superconformal vertex algebra of central charge $c$. This follows by a direct application of the Jacobi identity for SUSY Lie conformal algebras (see, e.g.,~\cite[Lemma A.1]{GCYHeluani} for details).
\end{remark}

\subsection{Local generators of supersymmetry}\label{sec:localgen}

The aim of this section is to construct local supersymmetry generators on the chiral de Rham complex, using abstract versions of the F-term and D-term conditions on complex Courant algebroids, generalizing those in Lemma \ref{lem:Ftermexp} and Proposition \ref{prop:DtermTHS}. This will serve as preparatory material for Theorem \ref{teo:mainTCAlocal}, where we will give a local version of our main results in Section \ref{sec:mainglobal}.

We start recalling the coordinate independent description of the chiral de Rham complex using Courant algebroids~\cite{Heluani09,HeluaniRev,GCYHeluani}. Let $E$ be a smooth complex Courant algebroid over a smooth manifold $M$. The definition of this object is analogous to Definition \ref{defn:CA}, replacing $\left\langle\cdot,\cdot\right\rangle$ by a $\C$-valued pairing, the anchor map by a morphism $\pi \colon E \to T \otimes \C$ of complex vector bundles, and smooth functions on $M$ by $\C$-valued smooth functions $C^\infty(M,\C)$. Let $\Pi E$ be the corresponding purely odd super vector bundle, so a local section $e$ of $E$ is identified with the corresponding local section $\Pi e$ of $\Pi E$. Abusing notation, $\langle, \rangle$, $[,]$ and
\begin{equation}\label{eq:parity-reversed-D.1}
\mathcal{D}\colon C^\infty(M,\C) \lto \Omega^0(\Pi E)  
\end{equation}
will respectively denote the corresponding super-skew-symmetric bilinear form on $\Pi E$, the Dorfman bracket on $\Pi E$, and the odd operator given by composing the exterior differential $d$, the map $\pi^* \colon T^* \otimes \C \to E^*$, the isomorphism $E^* \cong E$ provided by $\langle, \rangle$, and the parity change operator $\Pi$. 

Let $\cH$ be the translation algebra defined in Section~\ref{ssec:background} and $\underline{\CC}$ the sheaf of locally constant functions on $M$. 
The next result provides a coordinate-free description of the chiral de Rham complex on the smooth manifold $M$. The original chiral de Rham complex in \cite{MSV} is recovered from this construction when $E = (T \oplus T^*)\otimes \C$, the standard Courant algebroid on $M$.

\begin{proposition}[\cite{Heluani09,GCYHeluani}]\label{prop:CDRE}
Let $E$ be a complex Courant algebroid over $M$. Then there exists a unique sheaf of SUSY vertex algebras $\Omega^{\ch}_E$ over $M$, equipped with embeddings of sheaves of $\underline{\CC}$-modules
$$
\iota \colon C^\infty(M,\C) \lhra \Omega^{\ch}_E, \qquad \jmath \colon \Pi E \lhra \Omega^{\ch}_E,
$$
satisfying the following properties:
\begin{enumerate}
\item[\textup{(1)}] $\iota$ is an isomorphism of unital commutative algebras onto its image:
$$
\iota(1)=\left|0\right\rangle,\quad\iota(fg) = :\iota(f)\iota(g):,
$$

\item[\textup{(2)}] $\iota$ and $\jmath$ are compatible with the $C^\infty(M,\C)$-module structure of $\Pi E$ and the $\cH$-module structure of $\Omega^{\ch}_E$:
$$
\jmath(fa) = :\iota(f)\jmath(a):, \qquad 2 S \iota(f) = \jmath(\mathcal{D}f),
$$

\item[\textup{(3)}] $\iota$ and $\jmath$ are compatible with the Dorfman bracket and pairing:
$$
{[\jmath(a)}_\Lambda\jmath(b) ] = \jmath([a,b]) + 2 \chi \iota (\langle a,b\rangle),
$$

\item[\textup{(4)}] $\iota$ and $\jmath$ are compatible with the action of $\Omega^0(\Pi E)$ on $C^\infty(M,\C)$:
$$
{[\jmath(a)}_\Lambda \iota(f)] = \iota(\pi(a)(f))
$$

\item [\textup{(5)}] $\Omega^{\ch}_E$ is universal with the above properties,
\end{enumerate}
	
for all $f,g \in C^\infty(M,\C)$, $a,b \in \Omega^0(\Pi E)$. %Furthermore, when $E$ is the standard Courant algebroid $(T \oplus T^*)\otimes \C$, $\Omega^{\ch}_E$ is the chiral de Rham complex of $M$.
\end{proposition}

 % To simplify the notation, the space of sections of the chiral de Rham complex $\Omega^{\ch}_E$ over an open subset $U\subset M$ will be denoted $\Omega^{\ch}_E(U)$.

Next we introduce abstract versions of the F-term and D-term conditions, generalizing those on Lemma \ref{lem:Ftermexp} and Proposition \ref{prop:DtermTHS}. Let $E$ be a complex Courant algebroid over a smooth manifold $M$. We fix subbundles 
\begin{equation}\label{eq:abstractllbar}
\ell \oplus \overline{\ell} \subset E
\end{equation}
such that $\langle \ell, \ell \rangle = \langle \overline \ell, \overline \ell \rangle = 0$ and $\overline \ell \cong \ell^*$ via $\langle,\rangle$. By hypothesis, the restrictions of the pairing to $C_+ := \ell \oplus \overline{\ell} $ and $C_- := (\ell \oplus \overline{\ell})^\perp$ are non-degenerate. Considering the direct sum decomposition
\begin{equation}\label{eq:abstractdecomp}
E = \ell \oplus \overline{\ell} \oplus C_-,
\end{equation}
we will use the notation
\begin{equation}\label{eq:abstractdecomp.1}
a = a_+ + a_- = a_\ell + a_{\overline{\ell}} + a_-,  
\end{equation}
where $a_\pm \in C_\pm$, $a_\ell \in \ell$, $a_{\overline{\ell}} \in \overline{\ell}$, for the corresponding direct sum decompositions of any section $a\in\Omega^0(E)$.

\begin{definition}\label{def:Fterm}
We say that $\ell \oplus \overline{\ell} \subset E$ as in \eqref{eq:abstractllbar} satisfies the \emph{F-term condition} if
\begin{equation}\label{eq:Ftermabs}
[\ell , \ell] \subset \ell, \qquad [\overline \ell , \overline \ell] \subset \overline \ell.
\end{equation}
\end{definition}

We will denote $n = \dim \ell$. In this section, we will assume that $\ell$ and $\overline{\ell}$ are endowed with global isotropic frames 
$$
\left\lbrace\epsilon_j\right\rbrace_{j=1}^n \subset \ell, \qquad  \left\lbrace\overline{\epsilon}_j\right\rbrace_{j=1}^n \subset \overline{\ell},
$$
that is, satisfying (cf. Lemma \ref{lem:frameprop})
$$
\langle \epsilon_j, \overline \epsilon_k \rangle  = \delta_{jk}, \qquad  \langle \epsilon_j, \epsilon_k \rangle  = \langle \overline \epsilon_j, \overline \epsilon_k \rangle = 0.
$$
Motivated by Proposition \ref{prop:DtermTHS}, we introduce the following definition.

\begin{definition}\label{def:Dterm}
Let $(\ell \oplus \overline{\ell},\left\lbrace\epsilon_j,\overline{\epsilon}_j\right\rbrace,\varepsilon)$ be a triple given by $\ell \oplus \overline{\ell} \subset E$ as in \eqref{eq:abstractllbar}, isotropic global frames $\left\lbrace\epsilon_j,\overline{\epsilon}_j\right\rbrace_{j=1}^n$, and a section $\varepsilon \in \Omega^0(E)$. We say that this tuple satisfies
\begin{enumerate}

\item[\textup{(1)}] the \emph{weak D-term equation} if
\begin{equation}\label{eq:DtermgravCA}
\frac{1}{2}\sum_{j=1}^n\left[\overline{\epsilon}_j,\epsilon_j\right] + \varepsilon \in \ell \oplus \overline{\ell},
\end{equation}
\item[\textup{(2)}] the \emph{D-term equation} if
\begin{equation}\label{eq:DtermECA}
\frac{1}{2}\sum_{j=1}^n\left[\overline{\epsilon}_j,\epsilon_j\right] = \varepsilon_{\ell} - \varepsilon.
\end{equation}
\end{enumerate}
\end{definition}

In our main applications in Section \ref{sec:mainglobal}, Lemma \ref{lem:frameprop} will apply and the quantity $\sum_{j=1}^n\left[\overline{\epsilon}_j,\epsilon_j\right]$ will be independent of the given frame. The former implies $[\overline{\epsilon}_j,\overline{\epsilon}_k] = 0$, which explains the asymmetry in equation \eqref{eq:DtermECA}. The latter justifies that, in the sequel, we will abuse notation and refer to a solution of the (weak) D-term equation simply as $(\ell \oplus \overline{\ell},\varepsilon)$.

We are ready to introduce the generators of supersymmetry of our interest. Using $\left\lbrace\epsilon_j,\overline{\epsilon}_j\right\rbrace_{j=1}^n$, we define frames of $\Pi \ell$ and $\Pi \overline{\ell}$, respectively by
\begin{equation}\label{eq:framesodd}
e^j = \Pi\epsilon_j, \qquad e_j = \Pi\overline{\epsilon}_j.
\end{equation}
Recall that we work with parity-reversed sections. Note also that for all elements $a,b,c \in \left\lbrace e^j,e_j\right\rbrace_{j=1}^n$ of the isotropic frame,
\begin{subequations}\label{eq:local-frame-identitities.1}
\begin{gather}\label{eq:local-frame-identitities.1.a}
[a,b] = -[b,a],
\\\label{eq:local-frame-identitities.1.b}
\langle[a,b],c\rangle = -\langle b,[a,c]\rangle = \langle a,[b,c]\rangle,
\end{gather}\end{subequations}
by the Courant algebroid axioms. Define the endomorphism
\begin{equation*}
%\begin{array}{rrcl}
I_+\colon C_+\longrightarrow C_+\colon\quad a_+\longmapsto a_{\overline{\ell}} - a_\ell.
% I_+ \colon & C_+ & \longrightarrow & C_+\\
% & a_+ & \longmapsto &  a_{\overline{\ell}} - a_\ell
%\end{array},
\end{equation*}
Extending this to $I_+\colon\Pi C_+\to\Pi C_+$ by $I_+\Pi a_+ := \Pi I_+a_+$, we define
\begin{equation}\label{eq:def.w}
w := I_+ \left(\sum_{j=1}^n\left[e^j,e_j\right]_+\right) %= \Pi I_+ \left(\sum_{j=1}^n\left[\epsilon_j,\overline{\epsilon}_j\right]_+\right) 
= \sum_{j=1}^n\left[e^j,e_j\right]_{\overline{\ell}} - \sum_{j=1}^n\left[e^j,e_j\right]_{\ell} \in \Omega^0(\Pi C_+).
\end{equation}
Fix a global section $\varepsilon \in \Omega^0(E)$, that will play the role of a `complex divergence', and consider the corresponding odd section
\begin{equation*}
u = \Pi \varepsilon. %\qquad e_+ = \Pi I_+\varepsilon_+ \in \Omega^0(\Pi E).
\end{equation*}
We define the following sections of the chiral de Rham complex $\Omega^{\ch}_E$ of $E$:
\begin{equation}
\label{eq:JHlocal2}
\begin{split}
J :&= \frac{i}{2}\sum_{j=1}^n:e^je_j: - Siu,\\
H :&= \frac{1}{2}\sum_{j=1}^n\left(:e_j\left(Se^j\right):+:e^j\left(Se_j\right):\right)\\
& + \frac{1}{4}\sum_{j,k=1}^n\left(:e_j:e^k\left[e^j,e_k\right]::+:e^j:e_k\left[e_j,e^k\right]::\right.\\
&-\left.:e_j:e_k\left[e^j,e^k\right]::-:e^j:e^k\left[e_j,e_k\right]::\right)\\
&+\frac{T}{2}w-T I_+u_+ + \frac{S}{2}\sum_{j=1}^n\left(:\left[u,e^j\right]e_j:+:e^j\left[u,e_j\right]:-2T\left\langle\left[u,e^j\right],e_j\right\rangle\right),
\end{split}
\end{equation}
and
\begin{equation}
\label{eq:cc2}
c := 3\left(n - 2\left(\sum_{j=1}^n\left\langle\left[u,e^j\right],e_j\right\rangle-\left\langle u,u\right\rangle\right)\right).
\end{equation}
Note that we are considering $\ell\oplus\overline{\ell}$ as an ordered pair in the above definitions.

The proofs of the following two important technical results are given in Appendix \ref{app:MRP}.

\begin{proposition}
\label{teo:mainresult2}
Assume that $\ell\oplus\overline{\ell}\subset E$ satisfies the F-term condition \eqref{eq:Ftermabs}. Then the sections \eqref{eq:JHlocal2} and \eqref{eq:cc2} satisfy
\begin{equation}
\label{eq:JJcorregido}
\begin{split}
\left[J_\Lambda J\right] &= -\left(H+\frac{\lambda\chi}{3}c\right)-\frac{1}{2}\left(\chi S+\lambda \right)S\left(c-3n\right).
\end{split}
\end{equation}
\end{proposition}

\begin{remark}\label{rem:constant-central-charge.1}
By Remark~\ref{rem:tech} and~\eqref{eq:cc2}, a necessary condition to construct an embedding of the $N=2$ superconformal vertex algebra with the above generators $J$ and $H$ is to have
\begin{equation*}\label{eq:condnecN2}
\sum_{j=1}^n\left\langle\left.\left[u,e^j\right]\right|e_j\right\rangle-\left\langle u,u\right\rangle \in \C.
\end{equation*}
This also simplifies formula \eqref{eq:JJcorregido}, since the last summand will then be zero. In particular, this condition holds when $\left[\varepsilon,\cdot\right]=0$, because $\mathcal{D}\langle \varepsilon,\varepsilon \rangle = 2[\varepsilon,\varepsilon] = 0$ implies that $\langle \varepsilon,\varepsilon \rangle \in \C$.
	
%\item[(b)] If $\varphi \in \Gamma(E)$ is holomorphic, then we can apply \eqref{eq:identech1} for $a=u$ to obtain that
%\begin{equation}
%\label{eq:HH'rel}
%H = H'-2\frac{T}{k}e_+.
%\end{equation}
\end{remark}

Now, we define functions, for $i,j = 1, \ldots, n$,
\begin{subequations}
\label{eq:condigenCour}
\begin{align}\label{eq:R.bis}
R :&= \sum_{j,k=1}^n\left(3\left\langle\left[e^j,e_j\right]_-,\left[e^k,e_k\right]_-\right\rangle-\left\langle\left[e^j,e_k\right]_-,\left[e^k,e_j\right]_-\right\rangle\right),
\\\label{eq:Fijsuper.bis}
F^{ij} :&= \tr\(\pi_{\overline{\ell}}\circ\ad_{\left[e^i,e^j\right]}\big|_{\overline{\ell}}\)+\sum_{k=1}^n\left(\left\langle\mathcal{D}\left\langle e^i,\left[e_k,e^k\right]\right\rangle,e^j\right\rangle-\left\langle\mathcal{D}\left\langle e^j,\left[e_k,e^k\right]\right\rangle,e^i\right\rangle\right),
\\\label{eq:Fijsub.bis}
F_{ij} :&= \tr\(\pi_{\ell}\circ\ad_{\left[e_i,e_j\right]}\big|_{\ell}\)+\sum_{k=1}^n\left(\left\langle\mathcal{D}\left\langle e_i,\left[e^k,e_k\right]\right\rangle,e_j\right\rangle-\left\langle\mathcal{D}\left\langle e_j,\left[e^k,e_k\right]\right\rangle,e_i\right\rangle\right),
\end{align}
\end{subequations}
where $\pi_\ell\colon E\to\ell$ and $\pi_{\overline{\ell}}\colon E\to\overline{\ell}$ are the projections given by the decomposition~\eqref{eq:abstractdecomp}.

\begin{proposition}
\label{teo:mainresult}
Assume that $\ell\oplus\overline{\ell}\subset E$ satisfies the F-term condition \eqref{eq:Ftermabs} and that $\varepsilon \in \Omega^0(E)$ satisfies $\left[\varepsilon,\cdot\right]=0$. Then, setting $u=\Pi\varepsilon$, the sections \eqref{eq:JHlocal2} and \eqref{eq:cc2} satisfy
\begin{equation}
\label{eq:newiden}
\begin{split}
H=& \frac{1}{2}\sum_{j=1}^n\left(:e_j\left(Se^j\right):+:e^j\left(Se_j\right):\right)\\
& + \frac{1}{4}\sum_{j,k=1}^n\left(:e_j:e^k\left[e^j,e_k\right]::+:e^j:e_k\left[e_j,e^k\right]::\right.\\
&-\left.:e_j:e_k\left[e^j,e^k\right]::-:e^j:e^k\left[e_j,e_k\right]::\right) +\frac{T}{2}w-T I_+u_+,\\
c=&3\left(n+2\left\langle \varepsilon,\varepsilon\right\rangle\right) \in \C.
\end{split}
\end{equation}
Moreover,
\begin{align}
\nonumber
[J_\Lambda& J]= -\left(H+\frac{\lambda\chi}{3}c\right),
\\\label{eq:teo:mainresult}
[H_\Lambda& J]= \left(2\lambda+2T+\chi S\right)\left(\frac{i}{2}\sum_{j=1}^n:e^je_j:-Si\left(u_++\frac12\sum_{j=1}^n\left[e^j,e_j\right]_-\right)\right)
\\\nonumber
&+\frac{i}{4}TS\mathcal{D}R+\frac{i}{4}\sum_{i,j=1}^n\lambda\left(:F^{ij}:e_je_i::-:F_{ij}:e^je^i::\right)
\\\nonumber
&-i\left(T+\frac32\lambda\right)\sum_{j,k=1}^n\left(:e_k\left[u_++\frac12\left[e^j,e_j\right]_-,e^k\right]:+:e^k\left[u_++\frac12\left[e^j,e_j\right]_-,e_k\right]_-:\right).
\end{align}
\end{proposition}

\subsection{Local embeddings}\label{sec:maintheolocal}

Our next goal is to prove a local abstract version of the embeddings of our main Theorems~\ref{teo:mainTCA} and~\ref{teo:mainTCAcHE}, assuming the abstract versions of the F-term and D-term conditions given in Definitions~\ref{def:Fterm} and~\ref{def:Dterm}. The setup and the notation will be those of Section~\ref{sec:localgen}. We will use the Einstein summation convention for repeated indices. 

%The following result follows from Proposition \ref{teo:mainresult}.

\begin{lemma}\label{lem:local-embedding.1}
Assume that $(\ell\oplus\overline{\ell},\varepsilon)$ satisfies the F-term condition \eqref{eq:Ftermabs} and the weak D-term equation \eqref{eq:DtermgravCA}, and that $\varepsilon \in \Omega^0(E)$ satisfies $\left[\varepsilon,\cdot\right]=0$. Then, setting $u=\Pi\varepsilon$, the sections \eqref{eq:JHlocal2} satisfy 
\begin{equation}
\label{eq:teoDtermgrav}
\begin{split}
\left[{H}_\Lambda{J}\right] & = \left(2\lambda+2T+\chi S\right)J\\
& + \frac{i}{4}\left(TS\mathcal{D}R+\lambda\left(:F^{ij}:e_je_i::-:F_{ij}:e^je^i::\right)\right).
\end{split}
\end{equation}
\end{lemma}

\begin{proof}
Applying the weak D-term equation \eqref{eq:DtermgravCA}, this result follows directly from Proposition \ref{teo:mainresult} using that $\left[\varepsilon,\cdot\right]=0$.
\end{proof}

Using the Jacobi identity for the $\Lambda$-bracket, we provide now a key relation between $R$ and the quantities $:F^{ij}:e_je_i::$ and $:F_{ij}:e^je^i::$ appearing in the right-hand side of~\eqref{eq:teoDtermgrav}. % using the Jacobi identity for the $\Lambda$-bracket.

\begin{lemma}\label{lem:locura}
Assume $(\ell\oplus\overline{\ell},\varepsilon)$ satisfies the F-term condition \eqref{eq:Ftermabs} and the weak $D$-term equation \eqref{eq:DtermgravCA}, and that $\varepsilon \in \Omega^0(E)$ satisfies $\left[\varepsilon,\cdot\right]=0$. Then, setting $u = \Pi \varepsilon$, we have
\begin{equation}\label{eq:igualdadloca}
3T^2R = T\left(:F^{ij}:e_je_i::-:F_{ij}:e^je^i::\right).
\end{equation}
In particular, if $:F^{ij}:e_je_i::$ and $:F_{ij}:e^je^i::$ are identically zero, then $T^2R=0$.
\end{lemma}

\begin{proof}
By \eqref{eq:teoDtermgrav}, applying skew-symmetry~\eqref{eq:comLambda} for the $\Lambda$-bracket, we have
\begin{equation*}
\left[J_\Lambda H\right] = -\left(2\lambda+T+\chi S\right)J+\frac{i}{4}\left(TS\mathcal{D}R-\left(\lambda+T\right)\left(:F^{ij}:e_je_i::-:F_{ij}:e^je^i::\right)\right).
\end{equation*}
Now, by the Jacobi identity~\eqref{eq:JacobiLambda} for the $\Lambda$-bracket, since
\begin{equation*}
\left[J_\Lambda J\right] = -\left(H+\frac{\lambda\chi}{3}c\right),
\end{equation*}
using that $\left[\varepsilon,\cdot\right]=0$ and $\langle \varepsilon , \varepsilon \rangle \in \C$, we obtain
\begin{equation*}
\begin{split}
\left[J_\Lambda H\right] & = \left[J_\Lambda\left(H+\frac{\gamma\eta}{3}c\right)\right]=-\left[J_\Lambda\left[{J}_\Gamma{J}\right]\right]\\
& = -\left[\left[J_\Lambda{J}\right]_{\Lambda+\Gamma}{J}\right]-\left[{J}_\Gamma\left[J_\Lambda{J}\right]\right]\\
& = -\left(\left[\left[J_\Lambda{J}\right]_{\Lambda+\Gamma}{J}\right]+\left[\left[J_\Lambda{J}\right]_{-\nabla-\Gamma}{J}\right]\right)\\
& = -\left(2\lambda+T+\chi S\right)J-\frac{i}{2}TS\mathcal{D}R-\frac{i}{4}\left(\lambda-T\right)\left(:F^{ij}:e_je_i::-:F_{ij}:e^je^i::\right).
\end{split}
\end{equation*}
Subtracting these two identities and using the identity $TS\mathcal{D}R=2T^2R$ (by part (2) of Proposition~\ref{prop:CDRE}), we obtain \eqref{eq:igualdadloca}.
\end{proof}

For the next result, we will assume that the given frames satisfy the following natural condition, which we already met in Lemma \ref{lem:frameprop},
\begin{equation}
\label{eq:specframe}
\left[e_j,e_k\right] = 0, \quad \text{for } j,k \in \{1,\ldots,n\}.
\end{equation}

\begin{proposition}\label{teo:teoglobalespe}
Assume that $(\ell\oplus\overline{\ell},\varepsilon)$ satisfies the F-term condition \eqref{eq:Ftermabs} and the $D$-term equation \eqref{eq:DtermECA}, that $\varepsilon \in \Omega^0(E)$ satisfies $\left[\varepsilon,\cdot\right]=0$, and that the given frames satisfy \eqref{eq:specframe}. Then $F_{ij}=0$ and setting $u = \Pi\varepsilon$, the sections \eqref{eq:JHlocal2} satisfy
\begin{equation}\label{eq:teo:teoglobalespe.1}
\left[{H}_\Lambda{J}\right]=\left(2\lambda+2T+\chi S\right)J + \frac{i}{4}\left(TS\mathcal{D}R + \lambda :F^{ij}:e_je_i::\right),
\end{equation}
where
\begin{subequations}\label{eq:teo:teoglobalespe.2}
\begin{align}\label{teo:teoglobalespe.2.a}
R&%=3\langle[e^j,e_j],[e^k,e_k]\rangle-4\langle [e_j,u],e^j\rangle+4\langle u,u\rangle
= - 4\langle [e^j,u],e_j\rangle + 8\left\langle u_-,u_-\right\rangle,
\\\label{eq:teo:teoglobalespe.2.b}
F^{ij} & = \left\langle\mathcal{D}\left\langle\left[e^i,e^j\right],e_k\right\rangle,e^k\right\rangle,\quad \text{for } i,j \in \{1,\ldots,n\}.
\end{align}
\end{subequations}
\end{proposition}

\begin{proof}
By~\eqref{eq:local-frame-identitities.1.a}, the D-term equation~\eqref{eq:DtermECA} can be written as
\begin{equation}\label{eq:Dterm.bis}
\frac{1}{2}[e^j,e_j]=u-u_\ell,
\end{equation}
where $u-u_\ell\in\overline{\ell}\oplus C_-$, %=u_{\overline{\ell}}+u_-$
by~\eqref{eq:abstractdecomp.1}. 
Applying this %, equation, the identities $\langle e_i,u_{\overline{\ell}}+u_-\rangle=0=\langle e_j,u_{\overline{\ell}}+u_-\rangle$ % (as $e_i,e_j\in\Pi\overline{\ell}$ and $\overline{\ell}=(\overline{\ell}\oplus C_-)^\perp$),
and~\eqref{eq:specframe}, we see from~\eqref{eq:Fijsub.bis} that
\begin{align*}
F_{ij} &
=2\(\langle\mathcal{D}\langle e_i,u-u_\ell\rangle,e_j\rangle-\langle\mathcal{D}\langle e_j,u-u_\ell\rangle,e_i\rangle\)=0,
% = \left\langle\mathcal{D}\left\langle e_i,\left[e^k,e_k\right]\right\rangle,e_j\right\rangle-\left\langle\mathcal{D}\left\langle e_j,\left[e^k,e_k\right]\right\rangle,e_i\right\rangle
% \\&
% =2\(\langle\mathcal{D}\langle e_i,u_{\overline{\ell}}+u_-\rangle,e_j\rangle-\langle\mathcal{D}\langle e_j,u_{\overline{\ell}}+u_-\rangle,e_i\rangle\)=0,
\end{align*}
so~\eqref{eq:teo:teoglobalespe.1} follows from Lemma~\ref{lem:local-embedding.1}. Furthermore,~\eqref{eq:teo:teoglobalespe.2.b} follows from Lemma~\ref{lem:rewriting-Fij}, as
\begin{equation*}
\frac{1}{2}\left\langle\left[[e_k,e^k],e^i\right],e^j\right\rangle = \left\langle\left[u_{\ell}-u,e^i\right],e^j\right\rangle = \left\langle\left[u_{\ell},e^i\right],e^j\right\rangle = 0,
\end{equation*}
by the D-term equation \eqref{eq:DtermECA}, $\left[\varepsilon,\cdot\right] = 0$ and $[u_{\ell},e^i]\in\ell$ (by the F-term condition~\eqref{eq:Ftermabs}). 
As for the function $R$, note that $\langle [e^r,e_s],e_i\rangle=\langle e^r,[e_s,e_i]\rangle=0$ by~\eqref{eq:specframe}, so 
% As for the function $R$, note that for all indices $1\leq r,s\leq n$,
\[
[e^r,e_s]_\ell=0,
\]
% since $\langle [e^r,e_s],e_i\rangle=\langle e^r,[e_s,e_i]\rangle=0$ for all $1\leq i\leq n$, by~\eqref{eq:local-frame-identitities.1.b} and~\eqref{eq:specframe}.
for all $1\leq r,s\leq n$. 
Thus $[e^j,e_k]=[e^j,e_k]_{\overline{\ell}}+[e^j,e_k]_-$ by~\eqref{eq:abstractdecomp.1}, and hence
\begin{align*}
\left\langle\left[e^j,e_k\right],\left[e^r,e_s\right]\right\rangle 
% & 
% = \left\langle\left[e^j,e_k\right]_{\overline{\ell}},\left[e^r,e_s\right]\right\rangle+\left\langle\left[e^j,e_k\right]_-,\left[e^r,e_s\right]\right\rangle
% \\
& = \left\langle\left[e^j,e_k\right],\left[e^r,e_s\right]_\ell\right\rangle+\left\langle\left[e^j,e_k\right]_-,\left[e^r,e_s\right]_-\right\rangle
% \\&
= \left\langle\left[e^j,e_k\right]_-,\left[e^r,e_s\right]_-\right\rangle.
\end{align*}
Applying this formula twice in~\eqref{eq:R.bis}, we obtain %and using~\eqref{eq:Dterm.bis}, we obtain
\begin{equation}\label{eq:R-intermediate.1}
\begin{split}
R&=3\left\langle\left[e^j,e_j\right],\left[e^k,e_k\right]\right\rangle-\left\langle\left[e^j,e_k\right],\left[e^k,e_j\right]\right\rangle
%\\&
%=12\left\langle u_-,u_-\right\rangle +24\left\langle u_\ell,u_{\overline{\ell}}\right\rangle-\left\langle\left[e^j,e_k\right],\left[e^k,e_j\right]\right\rangle.
\end{split}\end{equation}
Note now that~\eqref{eq:local-frame-identitities.1.b} and~\eqref{eq:specframe} imply 
%\[
$[u_\ell,e^k]=\mathcal{D}\langle u_\ell,e^k\rangle-[e^k,u_\ell]=-[e^k,u_\ell]$.
%\]
Using this,~\eqref{eq:local-frame-identitities.1.b}, \eqref{eq:specframe}, $[u,\cdot]=0$,~\eqref{eq:Dterm.bis} and the Courant algebroid axioms, we obtain
\begin{align*}
\langle e_i,[e_j,e^k]\rangle&=\langle[e_i,e_j],e^k\rangle=0,
\\
\langle e_k,[[e^j,e_j],e^k]\rangle&
%=\langle e_k,[2(u-u_\ell),e^k]\rangle=-2\langle e_k,[u_\ell,e^k]\rangle
=2\langle e_k,[e^k,u_\ell]\rangle
%\\&
% =2\langle e^k,\mathcal{D}\langle e_k,u_\ell\rangle\rangle-2\langle[e^k,e_k],u_\ell\rangle
=2\langle e^k,\mathcal{D}\langle e_k,u\rangle\rangle-2\langle[e^k,e_k],u_\ell\rangle
%\\&
%=2\langle e^k,\mathcal{D}\langle e_k,u\rangle\rangle-4\langle u-u_\ell,u_\ell\rangle
%=2\langle e^k,[e_k,u]+[u,e_k]\rangle-4\langle u,u_\ell\rangle
%\\&
\\&
=2\langle e_k,[e^k,u]\rangle+4\langle u-u_\ell,u-u_\ell\rangle,
\\
\langle e_k,[e_j,[e^j,e^k]]\rangle&
%=\langle\mathcal{D}\langle e_k,[e^j,e^k]\rangle,e_j\rangle-\langle[e_j,e_k],[e^j,e^k]\rangle
=\langle\mathcal{D}\langle e_k,[e^j,e^k]\rangle,e_j\rangle
%\\&
%=\langle\mathcal{D}\langle e^j,[e^k,e_k]\rangle,e_j\rangle
=2\langle\mathcal{D}\langle e^j,u-u_\ell\rangle,e_j\rangle
%=2\langle\mathcal{D}\langle e^j,u\rangle,e_j\rangle
%\\&
%=2\langle[e^j,u]+[u,e^j],e_j\rangle
%=2\langle[e^j,u],e_j\rangle
%=2\langle e^j,\mathcal{D}\langle u,e_j\rangle\rangle-2\langle u,[e^j,e_j]\rangle
%\\&
%=2\langle e^j,[u,e_j]+[e_j,u]\rangle-2\langle u,[e^j,e_j]\rangle
%=2\langle e^j,[e_j,u]\rangle-2\langle u,[e^j,e_j]\rangle
%\\&
=2\langle [e^j,u],e_j\rangle%\\
%& = 2\langle\mathcal{D}\langle e^j,u\rangle,e_j\rangle = 2 \langle [e_j,e^j] , u \rangle + 2 \langle e^j , [e_j,u] \rangle
\end{align*}

The last three identities,~\eqref{eq:local-frame-identitities.1.a} and the Courant algebroid axioms imply %and~\eqref{eq:Dterm.bis} imply 
\begin{align*}
\langle[e^j,e_k],[e^k,e_j]\rangle&
% =-\langle[e^j,e_k],[e_j,e^k]\rangle
% =\langle e_k,[e^j,[e_j,e^k]]\rangle-\langle e^j,\mathcal{D}\langle e_k,[e_j,e^k]\rangle\rangle
%\\&
=\langle e_k,[e^j,[e_j,e^k]]\rangle
% =\langle e_k,[[e^j,e_j],e^k]\rangle+\langle e_k,[e_j,[e^j,e^k]]\rangle
% \\&
% =\(2\langle e^k,[e_k,u]\rangle-4\langle u,u_\ell\rangle\)
% +\(2\langle [e_j,u],e^j\rangle-4\langle u,u-u_\ell\rangle\)
%\\&
=4\(\langle [e^j,u],e_j\rangle+\langle u-u_\ell,u-u_\ell\rangle\).
\end{align*}
%%%%%%%%%%%%%%%%%%%%
%Formula alternativa
%\begin{align*}
%\langle[e^j,e_k],[e^k,e_j]\rangle& =\langle e_k,[e^j,[e_j,e^k]]\rangle = 4\langle [e_j,u],e^j\rangle + 4 \langle u-%u_\ell,u\rangle + 4 \langle u_-,u_-\rangle\\
%& = 4\langle [e_j,u],e^j\rangle + 8 \langle u_-,u_-\rangle + 4 \langle u_{\overline{\ell}} , u_\ell \rangle
%\end{align*}
Substituting this into~\eqref{eq:R-intermediate.1} and using~\eqref{eq:Dterm.bis}, we obtain~\eqref{teo:teoglobalespe.2.a}.
\end{proof}

We finish this section with a construction of a local embedding of the $N=2$ superconformal vertex algebra on the space of sections of the chiral de Rham complex.

\begin{theorem}\label{teo:mainTCAlocal}
Let $E$ be a smooth complex Courant algebroid over a smooth manifold $M$, $\ell\oplus\overline{\ell}\subset E$ as in \eqref{eq:abstractllbar} and $\varepsilon \in \Omega^0(E)$.
Suppose that $\Pi(\ell\oplus\overline{\ell})$ admits a global isotropic frame $\left\lbrace e^j,e_j\right\rbrace$ satisfying \eqref{eq:specframe},
$(\ell\oplus\overline{\ell},\varepsilon)$ satisfies the F-term condition \eqref{eq:Ftermabs} and the $D$-term equation \eqref{eq:DtermECA},
and that $\left[\varepsilon,\cdot\right]=0$.
Assume further that
\begin{equation}\label{eq:algecond}
:F^{ij}:e_je_i:: = 0.
\end{equation}
Then, setting $u=\Pi\varepsilon$, the sections
\begin{equation}
\label{eq:JHnewexact}
\begin{split}
J = & \frac{i}{2}:e^je_j: - Siu,\\
H = & \frac{1}{2}\left(:e_j\left(Se^j\right):+:e^j\left(Se_j\right):\right)+\frac{1}{4}\left(:e_j:e^k\left[e^j,e_k\right]::\right.\\
+ & \left.:e^j:e_k\left[e_j,e^k\right]::-:e_j:e_k\left[e^j,e^k\right]::-:e^j:e^k\left[e_j,e_k\right]::\right)+Tu_{\ell},
\end{split}
\end{equation}
induce an embedding of the $N=2$ superconformal vertex algebra with central charge
\begin{equation*}
c=3\dim \ell + 6\left\langle\varepsilon,\varepsilon\right\rangle \in \C
\end{equation*}
into the space of global sections of the chiral de Rham complex $\Omega^{\ch}_E$.
\end{theorem}

\begin{proof}
This follows from Remark~\ref{rem:tech}, Proposition~\ref{teo:mainresult}, Lemma~\ref{lem:locura} and Proposition~\ref{teo:teoglobalespe}.
\end{proof}
      
\subsection{Global embeddings}\label{sec:mainglobal}  

We can now prove the main results of this paper, Theorems~\ref{teo:mainTCA} and~\ref{teo:mainTCAcHE}, which provide embeddings of the $N=2$ superconformal vertex algebra in the chiral de Rham complex, respectively constructed from solutions of the twisted Hull--Strominger system \eqref{eq:tHS} and the coupled Hermitian-Einstein equation \eqref{eq:cHE}.

We start with the twisted Hull--Strominger system. Let $M$ be a $2n$-dimensional spin manifold endowed with a principal $K$-bundle $P$. Suppose the Lie algebra $\mathfrak{k}$ is equipped with a non-degenerate bi-invariant symmetric bilinear form $\left\langle\cdot,\cdot\right\rangle: \mathfrak{k} \otimes \mathfrak{k} \to \mathbb{R}$, such that
$$
p_1(P) = 0 \in H^4_{dR}(M,\mathbb{R}).
$$
%In this section we complete the proof of Theorem \ref{teo:mainTCAintro}, showing that a solution of the twisted Hull--Strominger system \eqref{eq:tHS} induces an embedding of the $N=2$ superconformal vertex algebra on the chiral de Rham complex associated to the solution.
Let $(\Psi,\omega,A)$ be a solution of the twisted Hull--Strominger system on $(M,P)$ (see Definition \ref{def:tHS}). By Example \ref{def:E0}, we have an associated string algebroid $E = E_{P,-d^c\omega,A}$, with underlying vector bundle
$$
T \oplus \ad P \oplus T^*,
$$
pairing \eqref{eq:pairingE0}, and bracket \eqref{eq:bracketE0} for $H = - d^c \omega$. %Applying Proposition \ref{prop:KSEqevendimCA}, the tuple $(\Psi,\omega,A,0)$ (with vanishing $B$-field $b = 0$) determines a solution $(V_+,\operatorname{div},\eta)$ of the Killing spinor equations on $E$, with generalized metric as in \eqref{eq:Vpm} and closed divergence operator 
%$$
%\operatorname{div} = \operatorname{div}_0 + \langle 2\theta_\omega , \rangle. 
%$$
Consider the generalized metric
\begin{equation*}
V_+ = \{X + g(X), X \in T\}, \quad V_- = \{X - g(X) + r, X \in T, r \in \ad P\},
\end{equation*}
where $g = \omega(,J)$, for the orthogonal integrable complex structure $J$ determined by $\Psi$. %The pure spinor line $\langle \eta \rangle \subset \Omega^0(S^+)$ or, equivalently, 
The complex structure $J$ gives a decomposition
\begin{equation}\label{eq:Ecdecomglobal}
E \otimes \C = \ell \oplus \overline{\ell} \oplus (V_- \otimes \C)
\end{equation}
which satisfies the F-term condition \eqref{eq:Ftermabs} (see Lemma \ref{lem:Ftermexp}), where
$$
\ell = e^{-i\omega}(T^{1,0}), \qquad \overline{\ell} = e^{i\omega}(T^{0,1}).
$$
Consider the holomorphic atlas $\{(U_\alpha,\psi_\alpha)\}_{\alpha \in \mathcal{A}}$ constructed in Lemma \ref{lem:holconstdetatlas} from the pair $(\Psi,\omega)$ (regarded as a solution of \eqref{eq:confbal}). Arguing as in Section \ref{sec:Dterm}, for each open coordinate patch $U_\alpha \subset M$ with holomorphic coordinates $(z_1,\ldots,z_n)$, we have well-defined local frames
$$
\left\{ \epsilon_j^\alpha = e^{-i\omega}g^{-1}d\overline{z}_j\right\}_{j=1}^n, \qquad \left\{ \overline{\epsilon}_j^\alpha = e^{i\omega}\dbar_j\right\}_{j=1}^n,
$$ 
of $\ell$ and $\overline{\ell}$, respectively, where $\dbar_j : = \frac{\partial}{\partial\overline{z}_j}$. By Lemma \ref{lem:frameprop}, we have the relations
$$
\langle \epsilon_j^\alpha, \overline \epsilon_k^\alpha \rangle  = \delta_{jk}, \qquad \langle \epsilon_j^\alpha, \epsilon_k^\alpha \rangle  = \langle \overline \epsilon_j^\alpha, \overline \epsilon_k^\alpha \rangle = 0, \qquad[\overline \epsilon_j^\alpha, \overline \epsilon_k^\alpha ] = 0.
$$
Furthermore, by Proposition \ref{prop:DtermTHS}, we have a family of solutions of the D-term equation \eqref{eq:DtermECA}:
$$
\frac{1}{2}\sum_{j=1}^n\left[\overline{\epsilon}_j^\alpha,\epsilon_j^\alpha\right] = \theta_\omega - (\theta_\omega)_{\ell},
$$
where $\varepsilon = - \theta_\omega \in \Omega^0(E \otimes \C)$ satisfies $\left[\varepsilon,\cdot\right]=0$ and $\langle \varepsilon , \varepsilon \rangle = 0$. We are therefore under the hypotheses of Propositions~\ref{teo:mainresult} and~\ref{teo:teoglobalespe}, so setting
$$
e^j_\alpha = \Pi  \epsilon_j^\alpha \in \Omega^0(\Pi \ell|_{U_\alpha}), \qquad e_j^\alpha = \Pi  \overline{\epsilon}_j^\alpha \in \Omega^0(\Pi \overline{\ell}|_{U_\alpha}),
$$
we obtain local sections $J^\alpha,H^\alpha \in \Omega^{\ch}_{E\otimes\C}(U_\alpha)$ of the chiral de Rham complex, defined by \eqref{eq:JHlocal2}, satisfying
\begin{equation}\label{eq:N=2aux}
\begin{split}
\left[J^\alpha_\Lambda J^\alpha \right] &= -\left(H^\alpha+\frac{\lambda\chi}{3}(3n)\right)\\
\left[{H^\alpha}_\Lambda{J^\alpha}\right] & = \left(2\lambda+2T+\chi S\right)J^\alpha + \frac{i}{4}\left(TS\mathcal{D}R^\alpha + \lambda F^\alpha \right),
\end{split} 
\end{equation}
where $F^\alpha = :F^{ij}_\alpha:e_j^\alpha e_i^\alpha::$ and
\begin{equation*}
\begin{split}
% R^\alpha&=16\left\langle u_-,u_-\right\rangle+8\left\langle u_\ell,u_{\overline{\ell}}\right\rangle-4\langle [e_j,u],e^j\rangle,
%\\&=3\langle[e^j_\alpha,e_j^\alpha],[e^k_\alpha,e_k^\alpha]\rangle-4\langle [e_j^\alpha,u],e^j_\alpha\rangle+4\langle u,u\rangle,
R^\alpha&= - 4\langle [e^j_\alpha,u],e_j^\alpha\rangle + 8\left\langle u_-,u_-\right\rangle,
\\
F^{ij}_\alpha & = \left\langle\mathcal{D}\left\langle\left[e^i_\alpha,e^j_\alpha\right],e_k^\alpha\right\rangle,e^k_\alpha\right\rangle,\quad \text{for } i,j \in \{1,\ldots,n\}.
\end{split}
\end{equation*}

\begin{lemma}\label{lem:globalsec}
The family of local sections 
$$
J^\alpha, H^\alpha, F^\alpha \in \Omega^{\ch}_{E\otimes\C}(U_\alpha)
$$ 
agree on overlaps $U_\alpha \cap U_\beta$, and therefore define global sections, denoted $J$, $H$ and $F$, respectively, of the chiral de Rham complex $\Omega^{\ch}_{E\otimes\C}$.
Furthermore, the family of local functions $R^\alpha$ agree on overlaps $U_\alpha \cap U_\beta$, and therefore define a global function $R\in C^\infty(M)$. 
% Furthermore, $R^\alpha$ is given by restriction to $U_\alpha$ of the global function
% $$
% R = 2 d^*\theta_\omega - |\theta_\omega|_\omega^2 \in C^\infty(M).
% $$ 
\end{lemma}

\begin{proof}
We will use basic identities on the chiral de Rham complex collected in Appendix~\ref{app:1}. 
Let $U_\alpha$ and $U_\beta$ be open coordinate patches in our atlas. To simplify notation, set
$$
e^j := e^j_\alpha, \quad e_j = e^j_\alpha, \qquad f^j := e^j_\beta, \quad f_j := e_j^\beta.
$$
Then there exist matrix-valued functions
$$
A = \left(A_j^k\right)_{j,k = 1}^n,
\;
B = \left(B_j^k\right)_{j,k = 1}^n \in M_{n\times n}\left(C^\infty(U_\alpha \cap U_\beta,\C)\right),
$$ 
such that
\begin{equation*}
f_j = \sum_{k=1}^n A^k_je_k \text{ and } f^{j} = \sum_{k=1}^n B_j^ke^k, \quad \text{for } j \in \{1,\ldots,n\}.
\end{equation*}
By the properties of our atlas (see Lemma \ref{lem:holconstdetatlas}), and the definition of the frames, we have
\begin{equation}
\label{eq:constdetyhol}
B= A^{-1}, \quad \mathcal{D}\det A = 0, \quad \pi\left(e^k\right)\left(B_j^r\right) = 0 = \pi\left(f^r\right)\left(A_k^j\right)\!,\; \text{for } j,k,r \in \{1,\ldots,n\}.
\end{equation}
By the basic identities~\eqref{eq:14}, \eqref{eq:affa}, \eqref{eq:abffab}, \eqref{eq:abfabf} \eqref{eq:fgabagab} and \eqref{eq:fagbfagb}, we have 
\begin{equation*}
\begin{split}
:f^{j}f_j: & = :e^je_j:+:B^j_k\left(TA_j^k\right):.
\end{split}
\end{equation*}
Now, by Lemma \ref{lem:matrixid}, the second term of the right-hand side is
\begin{equation*}
:B^j_k(TA_j^k):=\tr\left(:A^{-1}\left(TA\right):\right)=T(\log\det A)=\frac{1}{2}S\mathcal{D}(\log\det A)= 0,
\end{equation*}
so $J^\alpha = J^\beta$ on $U_\alpha \cap U_\beta$. By the first relation in \eqref{eq:N=2aux}, the family of local sections $H^\alpha$ also agree on overlaps.

As for the family of local sections $F^\alpha$, we need to prove that
\begin{equation*}
:\left\langle\mathcal{D}\left\langle\left[e^i,e^j\right],e_k\right\rangle,e^k\right\rangle:e_je_i:: = :\left\langle\mathcal{D}\left\langle\left[f^i,f^j\right],f_k\right\rangle,f^k\right\rangle:f_jf_i::.
\end{equation*}
Indeed, by \eqref{eq:abfabf} \eqref{eq:fgabagab} and \eqref{eq:fagbfagb}, notice that
\begin{equation*}
\begin{split}
::A^q_je_q::A^r_ie_r:: &= :\left(A^q_jA^r_i\right):e_qe_r::,
\end{split}
\end{equation*}
while by the Courant algebroid axioms and \eqref{eq:14},
\begin{equation*}
\begin{split}
\left\langle\mathcal{D}\left\langle\left[B^i_re^r,B^j_me^m\right],A_k^se_s\right\rangle,B^k_pe^p\right\rangle & = B^k_pA^j_mB^i_r\left\langle\left[e^r,e^m\right],e_s\right\rangle\left\langle\mathcal{D}A_k^s,e^p\right\rangle\\
&+B^i_r\left\langle\left[e^r,e^m\right],e_s\right\rangle\left\langle\mathcal{D}B^j_m,e^s\right\rangle\\
&+B^j_m\left\langle\left[e^r,e^m\right],e_s\right\rangle\left\langle\mathcal{D}B^i_r,e^s\right\rangle\\
&+B^j_mB^i_r\left\langle\mathcal{D}\left\langle\left[e^r,e^m\right],e_s\right\rangle,e^s\right\rangle\\
&+B^k_pB^j_m\left\langle\mathcal{D}B^i_s,e^m\right\rangle \left\langle\mathcal{D}A_k^s,e^p\right\rangle\\
&+\left\langle\mathcal{D}B^i_k,e^m\right\rangle\left\langle\mathcal{D}A^j_m,e^k\right\rangle+\left\langle\mathcal{D}B^j_m,e^m\right\rangle \left\langle\mathcal{D}B^i_k,e^k\right\rangle\\
&+B^j_m\left\langle\mathcal{D}\left\langle\mathcal{D}B^i_k,e^m\right\rangle,e^k\right\rangle+B^i_k\left\langle\mathcal{D}\left\langle\mathcal{D}B^j_m,e^m\right\rangle,e^k\right\rangle\\
&+B^k_pB^i_s\left\langle\mathcal{D}B^j_m,e^m\right\rangle \left\langle\mathcal{D}A_k^s,e^p\right\rangle,
\end{split}
\end{equation*}
which implies the required identity, thanks to third equation in \eqref{eq:constdetyhol}. 

The result about $R^\alpha$ follows by an explicit calculation in the local frames. Since we are not going to use this result, we refer to \cite[Lemma 10.3.5]{Arriba} for a complete proof. 
% The formula for $R^\alpha$ follows by an explicit calculation using the definition of the local frames. Since we are not going to use this formula, we refer to \cite[Lemma 10.3.5]{Arriba} for a complete proof.
\end{proof}

In the next result, we establish the relation between the global section $F$ of $\Omega^{\ch}_{E\otimes\C}$ and the torsion bivector field of $(\Psi,\omega)$, defined in Lemma \ref{lem:sigmaomega}.

\begin{lemma}\label{lem:dcomega}
Let $\sigma_\omega$ be the torsion bivector field of $(\Psi,\omega)$, defined in Lemma \ref{lem:sigmaomega}. Then, in the holomorphic coordinates provided by Lemma \ref{lem:holconstdetatlas}, we have that
	\begin{equation*}
	\begin{split}
	\left(\sigma_{\omega}\right)_{\overline{ij}} & = F^{ij}, \quad \text{for } i,j \in \{1,\ldots,n\}.
	\end{split}
	\end{equation*}
\end{lemma}
\begin{proof}
Arguing as in the last part of the proof of Lemma \ref{lem:braketsum}, we have that
\begin{equation*}
\begin{split}
\left\langle\left[e^i,e^j\right],e_k\right\rangle & = -\left\langle e^j,\left[e^i,e_k\right]\right\rangle\\
& = - \left\langle \epsilon_j,\left[\sigma_-\left(g^{-1}d\overline{z}_i\right),\overline{\epsilon}_k\right]\right\rangle\\
& = -d\overline{z}_j\left(\nabla^B_{g^{-1}d\overline{z}_i}\overline{\partial}_k\right)\\
& = -d\overline{z}_j\left(\nabla^+_{g^{-1}d\overline{z}_i}\overline{\partial}_k-\nabla^+_{\overline{\partial}_k}\left(g^{-1}d\overline{z}_i\right)\right)\\
& = -d^c\omega\left(g^{-1}d\overline{z}_i,g^{-1}d\overline{z}_j,\overline{\partial}_k\right),
\end{split}
\end{equation*}
for all $i,j,k \in \{1,\ldots,n\}$, since $\nabla^Bg^{-1}d\overline{z}_i \in \Omega^{1,0}$  and
\begin{equation*}
\left[g^{-1}d\overline{z}_i,\overline{\partial}_k\right]^{0,1}=\overline{\left[g^{-1}dz_i,\partial_k\right]^{1,0}}=\overline{\overline{\partial}_{g^{-1}dz_i}\left(\partial_k\right)}=0.
\end{equation*}
The proof follows now by comparing the explicit formula \eqref{eq:sigmaomegaexp} with 
\[
F^{ij} %= \left\langle\mathcal{D}\left\langle\left[e^i,e^j\right],e_k\right\rangle,e^k\right\rangle 
= -\pi(\epsilon_k)(d^c\omega\left(g^{-1}d\overline{z}_i,g^{-1}d\overline{z}_j,\overline{\partial}_k\right)) = ig^{m \overline{k}}\partial_m(\partial \omega\left(g^{-1}d\overline{z}_i,g^{-1}d\overline{z}_j,\overline{\partial}_k\right)).
\qedhere
\]
\end{proof}
      
We are ready to prove our first main result.
      
\begin{theorem}\label{teo:mainTCA}
Let $M$ be a $2n$-dimensional smooth manifold endowed with a principal bundle $P$. Let $(\Psi,\omega,A)$ be a solution of the twisted Hull--Strominger system \eqref{eq:tHS} on $(M,P)$, and consider the associated string algebroid $E = E_{P,-d^c\omega,A}$ and the chiral de Rham complex $\Omega^{\ch}_{E \otimes \C}$. Then the following expressions, defined by the frames \eqref{eq:localisoframe} (see also \eqref{eq:framesodd}) induced by the atlas in Lemma \ref{lem:holconstdetatlas},
\begin{equation}\label{eq:JHglocalmain}
\begin{split}
J & := \frac{i}{2}\sum_{j=1}^n:e^je_j: + Si\Pi \theta_\omega,\\
H & := \frac{1}{2}\sum_{j=1}^n\left(:e_j\left(Se^j\right):+:e^j\left(Se_j\right):\right) + \frac{1}{4}\sum_{j,k=1}^n \Big{(}:e_j:e^k[e^j,e_k]::\\
& + :e^j:e_k[e_j,e^k]:: - :e_j:e_k[e^j,e^k]:: - :e^j:e^k [e_j,e_k]::\Big{)} - T \Pi g^{-1} \theta_\omega^{0,1},
\end{split}
\end{equation}
define global sections of $\Omega^{\ch}_{E \otimes \C}$. Furthermore, the sections $J$ and $H$ induce an embedding of the $N=2$ superconformal vertex algebra with central charge $c = 3n$ into the space of global sections of $\Omega^{\ch}_{E\otimes\C}$.
\end{theorem} 

\begin{proof}
By Lemma \ref{lem:globalsec}, the sections $J$ and $H$ are global sections of $\Omega^{\ch}_{E \otimes \C}$ and satisfy the relations \eqref{eq:N=2aux}, for the sections $F$ and $R$ defined in Lemma \ref{lem:globalsec}. By Lemma \ref{lem:holonomia}, we have $\rho_B = 0$, and therefore, by Proposition \ref{prop:rho20sigma}, the torsion bivector field $\sigma_\omega$ vanishes identically. Using this, combined with Lemma \ref{lem:dcomega}, we have that $F = 0$ in \eqref{eq:N=2aux}. Finally, the proof follows from Lemma \ref{lem:locura}, which implies $TS\mathcal{D}R = 2T^2 R = 0$, because $TF = 0$.
\end{proof}

We construct next a new embedding from solutions of the coupled Hermitian-Einstein equation \eqref{eq:cHE}. The construction is similar to that of Theorem \ref{teo:mainTCA}, and hence we should just emphasize the main points. Let $(M,J)$ be a complex $n$-dimensional manifold which admits a holomorphic volume form $\Omega$. We fix a string algebroid $E$ over $(M,J)$ endowed with a generalized metric $V_+ \subset E$ such that the associated decomposition \eqref{eq:lbarl} satisfies $\langle \ell, \ell \rangle = \langle \overline \ell, \overline \ell \rangle = 0$ and the F-term condition \eqref{eq:Fterm}. In this setup, Lemma \ref{lem:braketsum0} gives a global section for arbitrary holomorphic coordinates
$$
\mu = \frac{1}{2}\sum_{j=1}^n\left[\overline{\epsilon}_j,\epsilon_j\right] \in \Omega^0(E \otimes \C),
$$
and we have a solution of the D-term equation \eqref{eq:DtermECA} with
\begin{equation}\label{eq:varepsilonDtermbis}
\varepsilon = d \log \|\Omega\|_\omega + i \left(d^*\omega - d^c\log \|\Omega\|_\omega \right) - \frac{i}{2} \Lambda_\omega F_A \in \Omega^0(E \otimes \C).
\end{equation}
Assuming that $V_+$ solves the coupled Hermitian-Einstein equation \eqref{eq:cHE}, by Proposition \ref{prop:DtermCHE} we have that $\left[\varepsilon,\cdot\right]=0$ and $\langle \varepsilon , \varepsilon \rangle \in \CC$. We are therefore under the hypotheses of Proposition \ref{teo:mainresult} and Proposition \ref{teo:teoglobalespe}, and for any holomorphic coordinate patch $U \subset M$ we obtain local sections $J_U,H_U \in \Omega^{\ch}_{E\otimes\C}(U)$ of the chiral de Rham complex, defined by \eqref{eq:JHlocal2} and satisfying the identities in Proposition \ref{teo:mainresult} and Proposition \ref{teo:teoglobalespe}. Assuming that we choose holomorphic coordinates such that 
$$
\Omega = dz_1 \wedge \ldots \wedge dz_n,
$$ 
the analogue of Lemma \ref{lem:globalsec} now holds, with the same proof, yielding global sections $J,H,F,R$ of $\Omega^{\ch}_{E\otimes\C}$ (the only difference being the formula for $R \in C^\infty(M)$, which is irrelevant for our applications). Furthermore, Lemma \eqref{lem:dcomega} also applies, for the given holomorphic coordinates, establishing the relation between the global section $F$ of $\Omega^{\ch}_{E\otimes\C}$ and the torsion bivector field $\sigma_\omega$, defined in Lemma \ref{lem:sigmaomega}.

We have arrived at our second main result.

\begin{theorem}\label{teo:mainTCAcHE}
Let $E$ be a string algebroid over an $n$-dimensional complex manifold $(M,J)$ endowed with a holomorphic volume form $\Omega$.  Let $V_+ \subset E$ be a generalized metric such that the associated decomposition \eqref{eq:lbarl} satisfies $\langle \ell, \ell \rangle = \langle \overline \ell, \overline \ell \rangle = 0$ and the F-term condition \eqref{eq:Fterm}. Assume that $V_+$ is a solution of the coupled Hermitian-Einstein equation \eqref{eq:cHE}. Then, defining $\varepsilon \in \Omega^0(E \otimes \C)$ by \eqref{eq:varepsilonDtermbis}, the following expressions, defined by the frames \eqref{eq:localisoframe} (see also \eqref{eq:framesodd}) for any choice of holomorphic coordinates such that $\Omega = dz_1 \wedge \ldots \wedge dz_n$,
\begin{equation}\label{eq:JHgmaincHE}
\begin{split}
J = & \frac{i}{2}\sum_{j=1}^n:e^je_j: - Si\Pi \varepsilon,\\
H = & \frac{1}{2}\sum_{j=1}^n\left(:e_j\left(Se^j\right):+:e^j\left(Se_j\right):\right)+\frac{1}{4}\sum_{j,k=1}^n\left(:e_j:e^k\left[e^j,e_k\right]::\right.\\
+ & \left.:e^j:e_k\left[e_j,e^k\right]::-:e_j:e_k\left[e^j,e^k\right]::-:e^j:e^k\left[e_j,e_k\right]::\right)+T\Pi\varepsilon_{\ell},
\end{split}
\end{equation}
induce an embedding of the $N=2$ superconformal vertex algebra into the space of global sections of $\Omega^{\ch}_{E\otimes\C}$, with central charge
\begin{equation*}
c=3n -\frac{3}{2} \langle \Lambda_\omega F_A ,\Lambda_\omega F_A \rangle
\in \R.
\end{equation*}
\end{theorem} 

\begin{proof}
The proof is analogous to that of Theorem \ref{teo:mainTCA}. Notice that for any solution of the coupled Hermitian-Einstein equation \eqref{eq:cHE} one has $\rho_B^{2,0} = 0$ (see Proposition \ref{prop:cHECA}), and therefore the associated torsion bivector field vanishes $\sigma_\omega = 0$, by Proposition \ref{prop:rho20sigma}.
\end{proof}

\begin{remark}
It should be emphasized that, for solutions of the Hull--Strominger system \eqref{eq:HS}, Theorem \ref{teo:mainTCA} and Theorem \ref{teo:mainTCAcHE} apply and give the same embedding. This follows from the existence of a global holomorphic volume $\Omega$ in this case, which implies that the atlas in Lemma \ref{lem:holconstdetatlas} satisfies $\Omega = dz_1 \wedge \ldots \wedge dz_n$, combined with formula \eqref{eq:varepsilonDtermbis}, which specializes to
$$
\varepsilon = d \log \|\Omega\|_\omega = - \theta_\omega.
$$
\end{remark}

\section{Examples}\label{sec:geoexam}

\subsection{Complex domains in $\C^2$}\label{sec:exdomain}

Our first class of examples where Theorem \ref{teo:mainTCA} and Theorem \ref{teo:mainTCAcHE} apply is a local situation in complex dimension $2$. For simplicity, we will consider the twisted Hull--Strominger system \eqref{eq:tHS} with trivial structure group $K = \{1\}$. In this case, the system becomes the so-called \emph{twisted Calabi--Yau equations} (see \cite[Definition 2.3]{grst})
\begin{equation}\label{eq:tCY}
\begin{split}
	d \Psi - \theta_\omega \wedge \Psi & = 0,\\
	d \theta_\omega & = 0,\\
	dd^c \omega  & = 0.
	\end{split}
\end{equation}
Observe that these equations include the condition that $\omega$ is a pluriclosed, that is, $dd^c \omega =0$. The application of Theorem \ref{teo:mainTCA} in this situation will lead us to embeddings of the $N=2$ superconformal vertex algebra into the chiral de Rham complex $\Omega^{\ch}_{E \otimes \C}$ of the exact Courant algebroid $E = E_{-d^c \omega}$. 

Let $X \subset \C^2$ be a complex domain. We fix the flat $\SU(2)$-structure
$$
\omega_0 = \frac{i}{2}\left(dz_1\wedge d\overline{z}_1+dz_2\wedge d\overline{z}_2\right), \qquad \Psi_0 = \frac{1}{2}dz_1 \wedge dz_2.
$$
Given a positive smooth harmonic function $f \in C^\infty(X)$, that is,
$$
\Delta_{\omega_0} f := 2i \Lambda_{\omega_0} \dbar \partial f = 0,
$$
we define 
\begin{equation*}\label{eq:TCYCS}
\omega = f \omega_0, \qquad \Psi = f \Psi_0.
\end{equation*}
Since $d\omega = d\log f \wedge \omega$, the associated Lee form is
\begin{equation*}\label{eq:LFCS}
\theta_\omega := d\log f \in \Omega^1.
\end{equation*}

\begin{lemma}\label{lem:SU(2)local}
The pair $(\Psi,\omega)$ defines an $\SU(2)$-structure solving the twisted Calabi--Yau equations \eqref{eq:tCY}.
\end{lemma}

\begin{proof}
Notice that $\Psi \wedge \omega = 0$ and 
$$
\Psi \wedge \overline{\Psi} = f^2 \Psi_0 \wedge \overline{\Psi}_0 = f^2 \frac{\omega^2_0}{2} = \frac{\omega^2}{2},
$$
so we have an $\SU(2)$-structure. Now, since
\begin{equation*}\label{eq:dcomegaCS}
d^c\omega = d^cf \wedge \omega_0 = d^c\log f \wedge \omega
\end{equation*}
and $\Delta_{\omega_0}f=0$, we conclude
\begin{equation*}
\begin{split}
d \Psi & = d \log f \wedge \Psi = \theta_\omega \wedge \Psi,\\
dd^c\omega & = dd^cf \wedge \omega_0 = - \Delta_{\omega_0} f \frac{\omega_0^2}{2} = 0.
\qedhere
\end{split}
\end{equation*}
\end{proof}

%In order to apply Theorem \ref{teo:mainTCA}, the torsion bivector field $\sigma_{\omega}$ defined in Lemma \ref{lem:sigmaomega} must vanish. The proof of our next result follows from the calculation in Example \ref{ex:sigmanonzero}.

%\begin{lemma}\label{lem:CSlem}
%Let $(\Psi,\omega)$ be as in Lemma \ref{lem:SU(2)local}. Then
%\begin{equation*}
%\sigma_{\omega} = 0.
%\end{equation*}
%\end{lemma}

%\begin{proof}
%We denote $\partial_j = \frac{\partial}{\partial z_j}$ and $\overline{\partial}_j = \frac{\partial}{\partial\overline{z _j}$. Notice that
%\begin{equation*}
%g^{-1}d\overline{z}_j = 2 f^{-1}\partial_j, \quad \text{for } j \in \{1,2\}.
%\end{equation*}
%Then, by direct calculation using Lemma \ref{lem:sigmaomega}, we have
%\begin{align*}
%\left(\sigma_\omega\right)_{\overline{11}} & = \left(\sigma_\omega\right)_{\overline{22}} = 0\\
%\left(\sigma_\omega\right)_{\overline{12}} & = i\sum_{k=1}^2g^{-1}d\overline{z}_k\left(4 f^{-2}\partial f \wedge \omega_0 \left(\partial_1,\partial_2,\overline{\partial}_k\right)\right)\\
%& = 2i\sum_{k=1}^2 f^{-1}\partial_k\left(- 4 \partial f^{-1} \wedge \omega_0 \left(\partial_1,\partial_2,\overline{\partial}_k\right)\right)\\
%& = - 4 f^{-1}\partial_1\left(\partial_2 f^{-1} \right) + 4f^{-1}\partial_2\left(\partial_1 f^{-1} \right) = 0,
%\end{align*}
%and the rest of components vanish by skew-symmetry.
%\end{proof}

The hypotheses of Theorem \ref{teo:mainTCA} are therefore satisfied, and we conclude the following. 

\begin{proposition}\label{prop:N2CS}
Let $X \subset \C^2$ be a complex domain, and consider a positive harmonic function $f\in C^\infty(X)$ with respect to the standard flat K\"ahler metric. Then the solution of the twisted Calabi--Yau equations $(\Psi,\omega)$ in Lemma \ref{lem:SU(2)local} induces an embedding of the $N=2$ superconformal vertex algebra of central charge $c = 6$ into the space of global sections of the chiral de Rham complex $\Omega^{\ch}_{E\otimes\C}$, for the exact Courant algebroid $E=E_{-d^c\omega}$. The generators of this embedding are given by \eqref{eq:JHglocalmain}, for the frames
$$
e^j = \Pi \left(2f^{-1}\partial_j + d\overline{z}_j\right), \qquad e_j = \Pi \left(\overline{\partial}_j + \tfrac{1}{2}f dz_j\right).
$$
\end{proposition}

\begin{proof}
The existence of the embedding follows from Lemma \ref{lem:SU(2)local} and Theorem \ref{teo:mainTCA}. The formula for the frames follows by construction, observing that the canonical coordinates in $\C^2$ belong to the canonical atlas associated to $(\Psi,\omega)$, given that $\theta_\omega$ is globally exact.
\end{proof}

We next analyze an interesting particular case of the previous result. Fix points $p_1, \ldots, p_k \in \CC^2$ and positive integers $n_1,\ldots,n_k$. Consider a distributional solution of the equation
$$
\Delta_{\omega_0} f = 4\pi^2 \sum_{j=1}^k n_j \delta_{p_j},
$$
on $\C^2$, where $\delta_{p_j}$ is the Dirac delta distribution located at $p_j$. Explicitly, this is given by
$$
f = C_0 + \sum_{j=1}^k \frac{n_j}{|z-p_j|^2},
$$
for some undetermined constant $C_0\in \RR$. Consider the domain
$$
X = \C^2 \backslash \{p_1,\ldots,p_k\}.
$$
Restricting $f$ to $X$, we obtain a positive harmonic function, for every $C_0 \geqslant 0$, and hence Proposition \ref{prop:N2CS} applies. The first generator of supersymmetry for the corresponding $N=2$ structure can be therefore written explicitly as
\begin{align}
\nonumber
J & := \frac{i}{2}:\Pi \left(2f^{-1}\partial_1 + d\overline{z}_1\right)\Pi\left(\overline{\partial}_1 + \tfrac{1}{2}f dz_1\right):
\\\label{eq:J5brane}
& + \frac{i}{2}:\Pi \left(2f^{-1}\partial_2 + d\overline{z}_2\right)\Pi\left(\overline{\partial}_2 + \tfrac{1}{2}f dz_2\right): + Si\Pi d \log f
\\\nonumber
& = \frac{i}{2}\sum_{j=1}^2 \left(2:f^{-1}:\Pi\partial_j\Pi\overline{\partial}_j:: + \frac{1}{2}:f:\Pi d\overline{z}_j\Pi dz_j:: + :\Pi\partial_j\Pi dz_j: + :\Pi d\overline{z}_j\Pi\overline{\partial}_j:\right),
\end{align}
where the last equality follows applying \eqref{eq:affa}, \eqref{eq:abffab}, \eqref{eq:abfabf}, \eqref{eq:afbafb} and \eqref{eq:fgabagab}.

The previous construction can be coupled to Hermitian-Yang--Mills connections, showing an interesting link to physics. For instance, given that the Bismut connection $\nabla^B$ of the $\SU(2)$-structure $(\Psi,\omega)$ associated to $f$ has vanishing Bismut Ricci form, it follows that
$$
\nabla^- = \nabla^g + \frac{1}{2}g^{-1}d^c\omega
$$
is an anti-self-dual orthogonal connection on $(T,g)$ (see, e.g., \cite[Remark 9.53]{GFStreets}). Taking the \emph{standard embedding ansatz} $A = \nabla^- \oplus \nabla^-$ on the $\operatorname{Spin}(4) \times \operatorname{Spin}(4)$-bundle of $V_0 \oplus V_1$, with  $V_j = (T,g)$ and the pairing
$$
\langle , \rangle = \tr_{V_0} - \tr_{V_1},
$$
we obtain a solution $(\omega,A)$ of the Hull--Strominger system \eqref{eq:HS} with $\Omega = \tfrac{1}{2}dz_1 \wedge dz_2$ for which Theorem \ref{teo:mainTCA} applies, giving the same supersymmetry generator $J$ as in \eqref{eq:J5brane}. These geometries corresponds to the `symmetric' \emph{heterotic string solitons} constructed by Callan, Harvey and Strominger \cite{CHS}. In the physical interpretation of these solutions, the Dirac delta distributions correspond to $k$ fivebrane sources on $\mathbb{R}^6 \times X$ with quantized charges $n_j \in \mathbb{N}$ and world-volume $\R^6 \times \{p_j\}$. It is interesting to observe that Theorem \ref{teo:mainTCA} matches the value of the central charge $c= 6$ for the corresponding sigma model, predicted by string theory. We expect that our main theorem applied to this class of examples provides a rigorous construction of the chiral part of the superconformal field theory studied in \cite[Section 7]{CHS}. According to physics, this has an enhanced $(4,4)$-supersymmetry, which should recover, in particular, our $(0,2)$-model. A rigorous construction should follow along the lines of \cite[Section 5.1]{AAGa} (see Remark \ref{rem:Hopfcompare}).

\subsection{Compact complex surfaces}\label{sec:surfaces}

In this section we analyse in further detail the general result obtained by direct application of Theorem \ref{teo:mainTCA} to solutions of the twisted Hull--Strominger system \eqref{eq:tHS} on compact complex surfaces.

\begin{theorem}\label{th:CCSemb}
Let $(M,J)$ be a smooth compact complex surface endowed with a principal bundle $P$. Let $(\Psi,\omega,A)$ be a solution of the twisted Hull--Strominger system \eqref{eq:tHS} on $(M,P)$, and consider the associated string algebroid $E = E_{P,-d^c\omega,A}$ and the chiral de Rham complex $\Omega^{\ch}_{E \otimes \C}$. Then there exists an embedding of the $N=2$ superconformal vertex algebra of central charge $c = 6$ into the space of global sections of the chiral de Rham complex $\Omega^{\ch}_{E\otimes\C}$. The generators of this embedding are given by \eqref{eq:JHglocalmain} for the frames defined in \eqref{eq:localisoframe} (see also \eqref{eq:framesodd}) induced by the atlas in Lemma \ref{lem:holconstdetatlas}.
\end{theorem}

The twisted Hull--Strominger system in complex dimension two is rather rigid. For instance, arguing as in the proof of \cite[Proposition 2.16]{grst}, for any solution $(\Psi,\omega,A)$ there exists a smooth function $\phi$ such that 
$$
\Psi' = e^{\phi}\Psi, \qquad \omega' = e^{\phi}\omega
$$
is a solution of the twisted Calabi--Yau equation \eqref{eq:tCY}. Here, $\omega'$ is a Gauduchon metric in the conformal class of $\omega$, which exists by Gauduchon's Theorem~\cite{Gauduchon.1984} (the rest of equations easily follow). Applying now the classification result in \cite[Proposition 2.10]{grst}, it follows that either $\omega'$ is K\"ahler (and Ricci flat) or $(M,J)$ is a quaternionic Hopf surface and $\omega'$ is a constant rescaling of the Boothby metric. In the first case, $\omega$ is conformally K\"ahler, while in the second case the metric $\omega'$ is locally conformally flat, since its pull-back to the universal cover $\widetilde M = \C^2 \backslash \{0\}$ is 
$$
\tilde\omega' = \frac{ai\left(dz_1\wedge d\overline{z}_1+dz_2\wedge d\overline{z}_2\right)}{|z|^2}
$$ 
for some positive real constant $a > 0$. %In particular, in the second case $\omega$ is locally conformally K\"ahler, and the result follows. 

\begin{remark}\label{rem:Hopfcompare}
In the case that $(M,J)$ is a quaternionic Hopf surface, there is a diffeomorphism $M \cong \SU(2) \times \U(1)$ and our result in Theorem \ref{th:CCSemb} provides a vast generalization of the $N=2$ superconformal embeddings constructed in \cite{AAGa} from homogeneous (that is, left-invariant) solutions of the Killing spinor equations. In fact, for a homogeneous solution of the twisted Calabi--Yau equations on $M$, there is a compatible hyperholomorphic structure and \cite[Proposition 5.2]{AAGa} applies, giving an embedding of the $N=4$ superconformal vertex algebra with central charge $c = 6$ into  $\Omega^{\ch}_{E\otimes\C}$. In this setup, we expect that the embedding in Theorem \ref{th:CCSemb} recovers one of the $N=2$ subalgebras in this family.
\end{remark}

By the previous discussion, any compact solution of the twisted Hull--Strominger system in complex dimension two is, in particular, locally conformally K\"ahler. For completeness, in the next result we show that any locally conformally K\"ahler manifold has vanishing torsion bivector field or, equivalently, $\rho_B^{2,0} = 0$. This illustrates the fact that this condition is strictly weaker than the existence of solutions to the (twisted) Hull-Strominger system.

\begin{proposition}\label{prop:lcK}
Let $(M,J)$ be a complex manifold of arbitrary dimension endowed with a locally conformally K\"ahler Hermitian form $\omega$. Then the torsion bivector field of $\omega$, defined in Lemma \ref{lem:sigmaomega}, identically vanishes: $\sigma_\omega = 0$. 
\end{proposition}

\begin{proof}
Let $p \in M$. By hypothesis, there exists an open set $U \subset M$ containing $p$ such that
$$
\omega = e^u \omega_0,
$$
for a locally defined smooth function $u$ and a local K\"ahler metric $\omega_0$. Applying \cite[Proposition 2.2]{Ruan}, after restricting to a possibly smaller neighbourhood, we can find holomorphic coordinates centered at $p$ such that
$$
\omega_0 = \frac{i}{2}\sum_{j=1}^n dz_j \wedge d \overline{z}_j + \frac{i}{2}\sum_{1 \leq k < l \leq n} z_k \overline{z}_l R_{i\overline{j}k\overline{l}} dz_i \wedge d \overline{z}_j + O(|z|^3),
$$
where $R_{i\overline{j}k\overline{l}}$ are the coefficients of the Chern curvature of $\omega_0$. From this, we obtain the following identities for the components of the K\"ahler metric $g_0$: 
$$
(g_0)_{m\overline{j}}(p) = \delta_{mj}, \qquad \partial((g_0)_{m\overline{j}})(p) = 0.
$$
Observe also that
\begin{align*}
i\partial \omega(g^{-1}d\overline{z}_i,g^{-1}d\overline{z}_j,\overline{\partial}_k) & = - (i\partial f \wedge \omega_0)(g^{-1}_0 d\overline{z}_i,g^{-1}_0 d\overline{z}_j,\overline{\partial}_k)\\
& = g_0^{m\overline{i}}\delta_{jk}\partial_m f  - g_0^{m\overline{j}}\delta_{ik}\partial_m f,
\end{align*}
where $f = e^{-u}$. Finally, using the Einstein summation convention we can directly calculate
\begin{align*}
(\sigma_\omega)_{\overline{i}\overline{j}}(p) & = f g^{s\overline{k}}_0\partial_s\left(g_0^{m\overline{i}}\delta_{jk}\partial_m f  - g_0^{m\overline{j}}\delta_{ik}\partial_m f \right)\\
& = f(\delta_{sk}\delta_{mi}\delta_{jk}\partial_s\partial_m f - \delta_{sk}\delta_{mj}\delta_{ik}\partial_s\partial_m f)\\
& = f(\partial_j\partial_i f - \partial_i\partial_j f)(p)\\
& = 0.
\qedhere
\end{align*}
\end{proof}

\subsection{Locally homogeneous solutions}\label{subsubsec:Iwasawa}

As mentioned in Section \ref{sec:intro}, the method developed in \cite{AAGa} in the homogeneous setup provides an embedding of the $N=2$ superconformal vertex algebra for all the known homogeneous solutions of the Hull--Strominger system with $\nabla$ Hermitian-Yang--Mills \cite{FeiYau,FIUV,GFGM,OU,OUVi}. A drawback of that method is that the embedding can only be constructed, a priori, for the chiral de Rham complex on the universal cover of the manifold, which is non-compact. By direct application of Theorem \ref{teo:mainTCA}, we obtain a refinement of this result for the compact, locally homogeneous, solutions constructed in the aforementioned references. To illustrate this general result, in this section we apply Theorem \ref{teo:mainTCA} to a natural family of homogeneous solutions of the Hull--Strominger system on the Iwasawa manifold, constructed in \cite{GFGM}. We should emphasize that our construction does not rely on any homogeneous assumption, unlike the previous construction undertaken in \cite[Section 5.2]{AAGa} for the universal cover.

Consider the \textit{Iwasawa manifold} $M =\Gamma \backslash H_\CC$, given by the compact quotient of the complex Heisenberg Lie group
$$
H_\CC =\left\{ \left(\begin{array}{c c c}
1 & z_2 & z_3\\
0 & 1 & z_1\\
0 & 0 & 1\\
\end{array}\right) \hspace{1mm} |\hspace{1mm} z_i \in \mathbb{C}\right\}
$$
by the lattice $\Gamma \subset \CC$ of Gaussian matrices (with entries in $\mathbb{Z}[i]$). The manifold $M$ has a natural complex structure $J$, induced by the bi-invariant complex structure on $H_\CC$. Let 
\begin{equation}\label{eq:Iwasawarel}
\omega_1=dz_1 \hspace{2mm} , \hspace{2mm} \omega_2=dz_2 \hspace{2mm} , \hspace{2mm} \omega_3=dz_3-z_2 dz_1,
\end{equation}
be left-invariant $1$-forms on the Lie group $H_\CC$. These forms are lattice invariant and descend to $(M,J)$. Moreover, they define a global basis of $T^*_{1,0}$ and satisfy the structure equations
$$
d\omega_1=d\omega_2=0 \hspace{3mm}, \hspace{3mm} d\omega_3=\omega_{12}.
$$
For any choice of
$$
(m,n,p)\in \mathbb{Z}^{3}\backslash \{0\},
$$
we consider the following purely imaginary $(1,1)$-form on $(M,J)$:
\begin{equation}\label{lbtorus}
F=\pi (m(\omega_{1\overline{1}}-\omega_{2\overline{2}})+n(\omega_{1\overline{2}}+\omega_{2\overline{1}})+ip(\omega_{1\overline{2}}-\omega_{2\overline{1}})).
\end{equation}
Note that $\tfrac{i}{2\pi} F$ has integral periods and hence, by general theory, this is the curvature form of the Chern connection of a holomorphic Hermitian line bundle $(\mathcal{L},h)\rightarrow (M,J)$.

Fix $(m_0,n_0,p_0), (m_1,n_1,p_1) \in \mathbb{Z}^3 \backslash \{0\}$ and consider the associated holomorphic Hermitian bundles $(\mathcal{L}_j,h_j) \to (M,J)$ with Chern connections $A_{h_j}$, for $j = 0,1$. Let $P$ denote the $(\U(1)\times \U(1))$-bundle of unitary frames of $(\mathcal{L}_0,h_1) \oplus (\mathcal{L}_1,h_1)$. We choose the holomorphic volume form $\Omega=\omega_{123}$ and $0\neq\alpha \in \RR$, and consider the Hull--Strominger system on $(M,P)$
\begin{equation}\label{eq:HSIwasawa}
\begin{split}
F_{A_0}^{0,2} = 0, \qquad F_{A_0} \wedge \omega^2 & = 0,\\
F_{A_1}^{0,2} = 0, \qquad F_{A_1} \wedge \omega^2 & = 0,\\
d(\|\Omega\|_\omega \omega^2) & = 0,\\
dd^c \omega - \alpha F_{A_0} \wedge F_{A_0} +  \alpha F_{A_1} \wedge F_{A_1} & = 0.
\end{split}
\end{equation}
where we use the product unitary connections $A = A_0 \oplus A_1$ (cf. \eqref{eq:HS}). Consider the hermitian form defined by 
\begin{align*}
\omega=\tfrac{i}{2}(\omega_{1\overline{1}}+\omega_{2\overline{2}}+\omega_{3\overline{3}}). 
\end{align*}
Note that $\omega$ is a balanced Hermitian metric.

\begin{proposition}[{\cite[Proposition 4.1]{GFGM}}]\label{prop:HSsol}
With the notation above, $(\omega,A_{h_0},A_{h_1})$ is a solution of the Hull--Strominger system \eqref{eq:HSIwasawa} if and only if
\begin{equation*}\label{eq:alphasol}
\alpha=\frac{1}{2\pi^2(m_0^2+n_0^2+p_0^2-m_1^2-n_1^2-p_1^2)}. 
\end{equation*}
\end{proposition}

For the sake of completeness, we provide next a direct calculation of the torsion bivector field $\sigma_{\omega}$ of the hermitian metric $\omega$ in the family of solutions in Proposition \ref{prop:HSsol}. In particular, by Proposition \ref{prop:rho20sigma} this gives an independent proof of $\rho_B^{2,0} = 0$ for this metric.

\begin{lemma}\label{lem:IwaF}
The torsion bivector field  of $\omega=\tfrac{i}{2}(\omega_{1\overline{1}}+\omega_{2\overline{2}}+\omega_{3\overline{3}})$ identically vanishes.
\end{lemma}

\begin{proof}
Applying \eqref{eq:Iwasawarel}, one can check that $d\omega = \frac{i}{2}\left(\omega_{12\overline{3}}-\omega_{3\overline{12}}\right)$, and hence
\begin{align*}
d^c\omega = -d\omega(J,J,J) = \frac12\left(\omega_{12\overline{3}}+\omega_{3\overline{12}}\right).
\end{align*}
Now, notice that
\begin{equation*}
\begin{split}
g^{-1}d\overline{z}_1 &= 2\left(\partial_1 + z_2 \partial_3\right),\\
g_u^{-1}d\overline{z}_2 &= 2\partial_2,\\
g_u^{-1}d\overline{z}_3 &= 2\left(\overline{z}_2 \partial_2+\left(1+|z_2|^2\right)\partial_3\right).
\end{split}
\end{equation*}
So we obtain 
\begin{equation*}
\begin{split}
i\partial\omega\left(g^{-1}d\overline{z}_1,g^{-1}d\overline{z}_2,\overline{\partial}_1\right) & = -2\overline{z}_2,\\
i\partial\omega\left(g^{-1}d\overline{z}_1,g^{-1}d\overline{z}_2,\overline{\partial}_2\right) & = 0,\\
i \partial \omega\left(g^{-1}d\overline{z}_1,g^{-1}d\overline{z}_2,\overline{\partial}_3\right) & = 2.
\end{split}
\end{equation*}
and therefore, by direct application of Lemma \ref{lem:dcomega},
\begin{equation*}
\left(\sigma_{\omega}\right)_{\overline{1}\overline{2}} = -2\left(g^{-1}d\overline{z}_1(\overline{z}_2)-g^{-1}d\overline{z}_3(2)\right)=0.
\end{equation*}
Similarly, we obtain that
$$
i\partial\omega\left(g^{-1}d\overline{z}_1,g^{-1}d\overline{z}_3,\overline{\partial}_j\right)= 0,
$$
and therefore $\left(\sigma_{\omega}\right)_{\overline{1}\overline{3}} = 0$. Finally, we have
\begin{equation*}
\begin{split}
i \partial \omega\left(g^{-1}d\overline{z}_2,g^{-1}d\overline{z}_3,\overline{\partial}_1\right) &= 2\overline{z}_2^2,\\
i \partial \omega\left(g^{-1}d\overline{z}_2,g^{-1}d\overline{z}_3,\overline{\partial}_2\right) &= 0,\\
i\partial \omega\left(g^{-1}d\overline{z}_2,g^{-1}d\overline{z}_3,\overline{\partial}_3\right) & = -2\overline{z}_2,
\end{split}
\end{equation*}
and therefore
\begin{equation*}
\left(\sigma_{\omega}\right)_{\overline{23}} = 2\left(g^{-1}d\overline{z}_1(\overline{z}_2^2)+g^{-1}d\overline{z}_3(\overline{z}_2)\right) = 0.
\qedhere
\end{equation*}
\end{proof}

We have arrived at the following result.

\begin{proposition}
\label{prop:Iwas2}
For $(m_0,n_0,p_0),(m_1,n_1,p_1) \in \Z^3 \backslash \{0\}$, consider the solution of the Hull--Strominger system $(\omega,A_{h_0},A_{h_1})$ in Proposition \ref{prop:HSsol}, the associated string algebroid $E$, and the chiral de Rham complex $\Omega^{\ch}_{E \otimes \C}$. Then there exists an embedding of the $N=2$ superconformal vertex algebra of central charge $c = 9$ into the space of global sections of the chiral de Rham complex $\Omega^{\ch}_{E\otimes\C}$. The generators of this embedding are given by \eqref{eq:JHglocalmain} for the frames (see also \eqref{eq:framesodd}) 
\begin{equation*}
\begin{split}
\epsilon_j & = g^{-1}d\overline{z}_j + d\overline{z}_j \in \ell,\\
\overline \epsilon_j & = \dbar_j + g \dbar_j  \in \overline \ell,
\end{split}
\end{equation*}
induced by the natural coordinates in the universal cover.
\end{proposition}

\begin{proof}
This is a direct consequence of Theorem \ref{teo:mainTCA} and Proposition \ref{prop:HSsol}. The formula for the frame follows because $\Omega = dz_1 \wedge dz_2 \wedge dz_3$ in the given coordinates, and $\theta_\omega = 0$.
\end{proof}

\subsection{Torus fibrations over Calabi--Yau surfaces}\label{sec:Picard}

In this section, we complete our study of the landscape of solutions of the Hull-Strominger with $\nabla$ Hermitian-Yang-Mills on compact non-K\"ahler threefolds. For this, we apply Theorem \ref{teo:mainTCA} to an interesting class of non-homogeneous solutions of the Hull--Strominger system \eqref{eq:HS} constructed by the third author in \cite{GF4}, inspired by the seminal work by Fu and Yau \cite{FuYau}. The underlying geometry is given by a $\mathbb{T}^2$-fibration over a K\"ahler $K3$ surface, considered originally in the physics literature \cite{DRS,GoPro}. Interestingly, these solutions are invariant under T-duality, and we speculate that they actually provide examples of $(0,2)$-mirror symmetry, following the methods in \cite{AAGa}.

Let $S$ be a $K3$ surface endowed with a K\"ahler Ricci-flat metric $\omega_S$ and a holomorphic volume form $\Omega_S$. We do not assume that $S$ is algebraic. Let $\omega_1$ and $\omega_2$ be anti-self-dual $(1,1)$-forms on $S$ such that 
$$
[\omega_i/2\pi] \in H^2(S,\ZZ).
$$
Let $M$ be the total space of the fibred product of the principal $U(1)$-bundles determined by $[\omega_1/2\pi]$ and $[\omega_2/2\pi]$. Then $M$ is a principal $\mathbb{T}^2$-bundle over $S$,
\begin{equation*}\label{eq:T2fibration}
p \colon M \longrightarrow S,
\end{equation*}
and we can choose a connection $\upsilon$ on $M$ with curvature 
$$
i F_\upsilon = (\omega_1,\omega_2) \in \Omega^2(S,\RR^2).
$$
Identifying $\RR^2 \cong \CC$ we construct a $\mathbb{T}^2$-invariant complex one-form $\sigma$ on $M$ such that
\begin{equation*}\label{eq:sigma}
d\sigma = \omega_1 + i \omega_2.
\end{equation*}
Consider $M$ endowed with the almost complex structure $J$ determined by the complex $3$-form
$$
\Omega = p^*\Omega_S \wedge \sigma.
$$
This complex structure is integrable, that is, $d\Omega = 0$, and %, by a theorem of Goldstein and Prokushkin, 
the corresponding complex threefold $(M,J)$ is non-K\"ahler unless $\omega_1 = \omega_2 = 0$ (see \cite[Theorem 1]{GoPro}).

Let $u$ be a smooth function on $M$. Then the hermitian metric 
\begin{equation}\label{eq:GPansatz}
\omega_{u} = p^*(e^u \omega_S) + \frac{i}{2} \sigma \wedge \overline{\sigma},
\end{equation}
satisfies the conformally balanced equation \cite{GoPro}
$$
d(\|\Omega\|_{\omega_u} \omega_{u} \wedge \omega_{u}) = 0.
$$
Consider the smooth complex vector bundle $T^{1,0} \to M$, and fix some background Hermitian metric $h_0$. We stress the fact that we do not consider any preferred holomorphic structure on $T^{1,0}$.

Consider the bilinear form defined by Poincar\'e duality
$$
Q \colon H^2(S,\ZZ) \times H^2(S,\ZZ) \longrightarrow \ZZ
$$
and denote $Q(\beta,\beta) = Q(\beta)$, for $\beta \in H^2(S,\ZZ)$. Let $\alpha \in \mathbb{R} \backslash \{0\}$ be such that
\begin{equation}\label{eq:integral}
\frac{Q([\omega_1/2\pi]) + Q([\omega_2/2\pi])}{\alpha}\in \mathbb{Z}.
\end{equation}

Let $(\mathbb{W},h_1)$ be a smooth Hermitian vector bundle over $S$ with rank $r$, $c_1(\mathbb{W}) = 0$ and second Chern class
\begin{equation}\label{eq:c2Wbis}
c_2(\mathbb{W}) =  24 + \frac{1}{\alpha}(Q([\omega_1/2\pi]) + Q([\omega_2/2\pi])).
\end{equation}
Definition \eqref{eq:c2Wbis} implies the more familiar `anomaly cancellation' condition for the pull-back bundle $p^*\mathbb{W}$ on the cohomology of $M$: 
$$
c_2(p^*\mathbb{W}) - c_2(M) = 0 \in H^4(M,\RR).
$$
%Note that $H^2(X,\RR) \cong H^2(S,\RR)/\langle [\omega_1,\omega_2] \rangle$ (see \cite{GoPro}).
%For the case that the $\mathbb{T}^2$-bundle $X$ is non-trivial, this follows from the isomorphism $H^2(X,\RR) \cong H^2(S,\RR)/\langle [\omega_1,\omega_2] \rangle$ (see \cite{GoPro}).

\begin{theorem}[\cite{GF4}]\label{thmGP}
Let $\alpha \in \Q$ satisfying \eqref{eq:integral}, and let $(\mathbb{W},h_1)$ be a smooth complex hermitian vector bundle over $S$ with rank $r$, $c_1(\mathbb{W}) = 0$, second Chern class \eqref{eq:c2Wbis}, and satisfying 
$$
r \leqslant c_2(\mathbb{W}).
$$
Then there exists a smooth function $u$ on $S$ and Hermitian-Yang--Mills unitary connections $\nabla$ on $T^{1,0}$ and $A_1$ on $p^*\mathbb{W}$ such that $(\omega_{u},\nabla,A_1)$ is a solution of the Hull--Strominger system 
\begin{equation}\label{eq:HSGP}
\begin{split}
	R_{\nabla}^{0,2} = 0, \qquad R_{\nabla} \wedge \omega^2 & = 0,\\
	F_{A_1}^{0,2} = 0, \qquad F_{A_1} \wedge \omega^2 & = 0,\\
	d(\|\Omega\|_\omega \omega^2) & = 0,\\
	dd^c \omega - \alpha \tr R_{\nabla} \wedge R_{\nabla} +  \alpha \tr F_{A_1} \wedge F_{A_1} & = 0,
	\end{split}
\end{equation}
where $R_{\nabla}$ denotes the curvature of $\nabla$. Furthermore, we can assume that the connection $\nabla$ on $T^{1,0}$ is $\omega_{u}$-unitary.
\end{theorem}

%For completeness, in the next result we check that the associated torsion bivector field of the solutions in Theorem \ref{thmGP} identically vanishes. As one can observe from the proof, this is a direct consequence of the Goldstein--Prokushkin ansatz \eqref{eq:GPansatz}, disregardless of the choice of function $u$.

%\begin{lemma}\label{lem:GPsigma}
%The torsion bivector field of $\omega_u$ in \eqref{eq:GPansatz} identically vanishes.
%\end{lemma}

Our next result is a direct consequence of Theorem \ref{teo:mainTCA}. % and Lemma \ref{lem:GPsigma}.

\begin{proposition}\label{prop:GPchiral}
Let $(\omega_u,\nabla,A_1)$ be the solution of the Hull--Strominger system in Theorem \ref{thmGP}, and consider the associated string algebroid $E$ and the chiral de Rham complex $\Omega^{\ch}_{E \otimes \C}$. Then there exists an embedding of the $N=2$ superconformal vertex algebra of central charge $c = 9$ into the space of global sections of the chiral de Rham complex $\Omega^{\ch}_{E\otimes\C}$. The generators of this embedding are given by \eqref{eq:JHglocalmain} for the frames \eqref{eq:localisoframe} induced by the atlas in Lemma \ref{lem:holconstdetatlas}.
\end{proposition}

To finish this section, we provide a direct calculation of the torsion bivector field $\sigma_{\omega}$ in an explicit example, originally considered by Picard in \cite{Picard}. This is given by a $\mathbb{T}^2$-fibration over a torus $\mathbb{T}^4$, endowed with a Hermitian metric with the Goldstein--Prokushkin ansatz \eqref{eq:GPansatz}. In particular, by Proposition \ref{prop:rho20sigma} this gives an independent proof of $\rho_B^{2,0} = 0$ for these metrics.

\begin{example}
Consider the compact complex $3$-dimensional manifold given by $(M,J):=\C^3/\sim$, where we define
\begin{equation*}
(z_1,z_2,z_3) \sim (z_1+a,z_2+c,z_3 + \overline{a}z_2+b)
\end{equation*}
for $a,b,c \in \Z[i]$. Observe that there is a holomorphic fibration structure
$$
p\colon (M,J) \longrightarrow  \C/\Lambda \times \C/\Lambda \cong \mathbb{T}^4\colon \left[\left(z_1,z_2,z_3\right)\right] \longmapsto \left([z_1],[z_2]\right),
$$
whose fibres are complex tori, where $\Lambda$ is the lattice generated by $\langle 1,i\rangle$. The following $1$-forms
$$
\omega_1=dz_1, \qquad \omega_2 = dz_2, \qquad \omega_3 = dz_3-\overline{z}_1dz_2,
$$
define a global holomorphic frame of $T_{1,0}^*M$, and satisfy the structure equations
$$
d\omega_1= 0, \qquad d\omega_2=0,\qquad d\omega_3=\omega_{2\overline{1}}.
$$
Consider the flat metric on the base torus
$$
\omega_{\mathbb{T}^4} = \frac{i}{2}\left(\omega_{1\overline{1}}+\omega_{2\overline{2}}\right).
$$
Then, for a choice of smooth function $u \in C^\infty(M)$, we define the Hermitian form (cf. \eqref{eq:GPansatz})
$$
\omega_u = p^* e^u\omega_{\mathbb{T}^4} +\frac{i}{2}\omega_{3\overline{3}}.
$$
We check next that $\sigma_{\omega_u} = 0$. Observe that we have
\begin{equation*}
\begin{split}
g_u^{-1}d\overline{z}_1 &= 2e^{-u}\partial_1,\\
g_u^{-1}d\overline{z}_2 &= 2\left(e^{-u}\partial_2 + e^{-u}\overline{z}_1\partial_3\right),\\
g_u^{-1}d\overline{z}_3 &= 2\left(z_1e^{-u}\partial_2 + \left(1+e^{-u}|z_1|^2\right)\partial_3\right),
\end{split}
\end{equation*}
and also $\omega_3 = dz_3 + \partial(-\overline{z}_1z_2)$ and
$$
\partial\omega_u = e^u \partial u \wedge \omega_{\mathbb{T}^4} + \tfrac{i}{2}(dz_3 -\overline{z}_1dz_2) \wedge dz_1 \wedge d \overline{z}_2.
$$
By Lemma \ref{lem:dcomega}, we obtain
\begin{equation*}
\begin{split}
i\partial\omega_u\left(g^{-1}_ud\overline{z}_1,g^{-1}_ud\overline{z}_2,\overline{\partial}_1\right) & = -2\partial_2\left(e^{-u}\right),\\
i\partial\omega_u\left(g^{-1}_ud\overline{z}_1,g^{-1}_ud\overline{z}_2,\overline{\partial}_2\right) &= 2\partial_1\left(e^{-u}\right),\\
i\partial\omega_u\left(g^{-1}_ud\overline{z}_1,g^{-1}_ud\overline{z}_2,\overline{\partial}_3\right) & = 0.
\end{split}
\end{equation*}
From this, we have that
\begin{equation*}
\left(\sigma_{\omega}\right)_{\overline{12}} : = 4e^{-u}\left(\partial_1\partial_2\left(e^{-u}\right) -\partial_2\partial_1\left(e^{-u}\right)\right)=0.
\end{equation*}
Similarly, we obtain
\begin{equation*}
\begin{split}
i\partial \omega_u\left(g^{-1}_ud\overline{z}_1,g^{-1}_ud\overline{z}_3,\overline{\partial}_1\right) & = -2\overline{\partial}_2\left(z_1e^{-u}\right),\\
i\partial\omega_u\left(g^{-1}_ud\overline{z}_1,g^{-1}_ud\overline{z}_3,\overline{\partial}_2\right) & = 2\overline{\partial}_1\left(z_1e^{-u}\right),\\
i\partial\omega_u\left(g^{-1}_ud\overline{z}_1,g^{-1}_ud\overline{z}_3,\overline{\partial}_3\right) & = 0.
\end{split}
\end{equation*}
and therefore
\begin{equation*}
\left(\sigma_{\omega}\right)_{\overline{13}} : = 4e^{-u}\left(\partial_1\partial_2\left(z_1e^{-u}\right) -\partial_2\partial_1\left(z_1e^{-u}\right)\right)=0.
\end{equation*}
By direct calculation, for $j = 1,2,3$ we also have that
$$
\partial\omega_u\left(g^{-1}_ud\overline{z}_2,g^{-1}_ud\overline{z}_3,\overline{\partial}_j\right) = 0,
$$
and therefore $\left(\sigma_{\omega}\right)_{\overline{23}} = 0$, which concludes the proof, by antisymmetry.
\end{example}

\begin{remark}
Observe that the complex 3-manifold $(M,J)$ of the previous example has a natural holomorphic volume form, given by $\Omega = \omega_{123}$, and furthermore, $\omega_u$ is balanced, for any $u$. Similarly as in Theorem \ref{thmGP}, it is not difficult to find a bundle $\mathbb{W} \to \mathbb{T}^4$ and Hermitian-Yang--Mills connections $\nabla$ on $T^{1,0}$ and $A_1$ on $p^*\mathbb{W}$ such that $(\omega_u,\nabla,A_1)$ is a solution of the Hull--Strominger system \eqref{eq:HSGP} (for suitable $u$ and $\alpha$). %By the previous calculation, $\sigma_{\omega_u} = 0$ and 
Therefore, Theorem \ref{teo:mainTCA} applies giving embeddings of the $N=2$ superconformal vertex algebra on the chiral de Rham complex of the solution.
\end{remark}

\appendix

\section{Basic identities}\label{app:1}

In this appendix we collect various identities for the chiral de Rham complex that are extensively used in the core part of the text. We use the Einstein summation convention for repeated indices.

\subsection{Basic chiral identities}

Let $\Omega^{\ch}_E$ be the chiral de Rham complex of a complex Courant algebroid $E$ over a smooth manifold $M$ (see Section \ref{sec:localgen}).

% \begin{lemma}
%For $a \in \Gamma(E)$, the following identity holds:
%\begin{align}
%2\left[a,a\right] & = \mathcal{D}\left\langle a,a\right\rangle.\label{eq:identech2}
%\end{align}
%\end{lemma}
%\begin{proof}
%It is straigthforward from quasiantisymmetry axiom for Courant algebroids.
%\end{proof}

\begin{lemma}
For $f \in C^\infty(M,\C)$ and $a,b \in \Omega^0(E)$, the following identity holds:
\begin{align}
\label{eq:14}
\left[fa,b\right] & = f\left[a,b\right]-f\mathcal{D}\left\langle a,b\right\rangle-\left\langle\mathcal{D}f,b\right\rangle a+\mathcal{D}\left\langle b,fa\right\rangle.
\end{align}
\end{lemma}

\begin{proof}
This follows easily from the Courant algebroid axioms (Definition \ref{defn:CA}).
\end{proof}

\begin{lemma}\label{lem:tech2bisbis}
For all $f,g \in C^\infty(M,\C)$, the following identity holds:
\begin{align*}
:\left(T^mf\right)\left(T^ng\right): &= :\left(T^ng\right)\left(T^mf\right):, \quad \text{for } m,n \in \N.
\end{align*}
\end{lemma}
\begin{proof}
Immediate, by the quasicommutativity axiom~\eqref{eq:cuasicon}.
\end{proof}

\begin{lemma}\label{lem:tech2bis}
For all $f \in C^\infty(M,\C)$ and $a \in \Omega^0(\Pi E)$, the following identities hold:
\begin{align}
\left[f_\Lambda a\right] & = \left\langle\mathcal{D}f,a\right\rangle,\label{eq:immedi}\\
S\left(fa\right)&= :f\left(Sa\right):+\frac12:\left(\mathcal{D}f\right)a:\label{eq:claveAi},\\
%:af: &= :fa:,\\
:\left(T^ma\right)\left(T^nf\right): &= :\left(T^nf\right)\left(T^ma\right):, \quad \text{for } m,n \in \N\label{eq:affa}.
\end{align}
\end{lemma}

\begin{proof}
The first identity is a direct consequence of the skew-symmetry axiom~\eqref{eq:comLambda} of the $\Lambda$-bracket, the second one follows because $S$ is an odd derivation for the normally ordered product~\eqref{eq:VA-equiv.1}, and the last one is immediate from quasicommutativity.
\end{proof}

\begin{lemma}
For all $a,b \in \Omega^0(\Pi E)$, the following identities hold:
\begin{align}
:ab:+:ba: & = 2T\left\langle a,b\right\rangle,\label{eq:cuasicom}\\
:\left(Sa\right)b: & = :b\left(Sa\right):+T\left[a,b\right],\label{eq:cuasicom2}\\
:a\left(Tb\right):+:\left(Tb\right)a: &= T^2\left\langle a,b\right\rangle.\label{eq:20}
\end{align}
\end{lemma}

\begin{proof}
Immediate, by quasicommutativity.
\end{proof}

\begin{lemma}\label{lem:tech3bis}
For $f \in C^\infty(M,\C)$ and $a,b \in \Omega^0(\Pi E)$, the following identities hold:
\begin{align}
::ab:f: &= :f:ab::\label{eq:abffab},\\
:a:fb:: &= ::fb:a:+2T\left(:\!f\left\langle a,b\right\rangle\!:\right),\label{eq:bfafab}\\
::ab:f: &= :a:bf::\label{eq:abfabf},\\
::af:b: &= :a:fb::+2:\left(Tf\right)\left\langle a,b\right\rangle:\label{eq:afbafb},\\
::af:\left(Sb\right): & = :a:f\left(Sb\right)::-:\left(Ta\right)\left\langle\mathcal{D}f,b\right\rangle:-:\left(Tf\right)\left(\mathcal{D}\left\langle a,b\right\rangle-\left[a,b\right]\right):\label{eq:13},\\
::a\left(Tf\right):b: &= :a:\left(Tf\right)b::+2:\left(T^2f\right)\left\langle a,b\right\rangle:\label{eq:17},\\
::\left(Ta\right)f:b: &= :\left(Ta\right):fb::-:\left(T^2f\right)\left\langle a,b\right\rangle:\label{eq:21},\\
:\left(T\left(:ab:\right)\right)f: &= :f\left(T\left(:ab:\right)\right):\label{eq:TabffTab}.
\end{align}
\end{lemma}

\begin{proof}
The first, second and last identities follow directly from quasicommutativity, while the other ones follow from the quasiassociativity axiom~\eqref{eq:cuasiaso}. The proof of some of these identities requires the non-commutative Wick formula~\eqref{eq:Wick} as well.
\end{proof}

\begin{lemma}
\label{lem:tech3bisbis}
For $f,g \in C^\infty(M,\C)$ and $a,b \in \Omega^0(\Pi E)$, the following identities hold:
\begin{align}
::fg::ab:: & = :f:g:ab:::\label{eq:fgabagab},\\
::fa::gb:: & = :f:a:gb:::+2:\left(Tf\right):g\left\langle a,b\right\rangle::\label{eq:fagbfagb}.
\end{align}
\end{lemma}
\begin{proof}
Both identities follow from quasiassociativity.
\end{proof}

\begin{lemma}
For all $a,b,c \in \Omega^0(\Pi E)$, the following identities hold:
\begin{align}
:a:bc:: & = ::bc:a: + 2T\left(\left\langle a,b\right\rangle c-\left\langle a,c\right\rangle b\right),\label{eq:cuasicom3}\\
::ab:c: & = :a:bc::+2\left(:\left(Ta\right)\left\langle b,c\right\rangle:-:\left(Tb\right)\left\langle a,c\right\rangle:\right),\label{eq:cuasiaso2}\\
::a\left(Sb\right):c: &= :a:\left(Sb\right)c::+:\left(Ta\right)\left[b,c\right]:+2:\left(TSb\right)\left\langle a,c\right\rangle:,\label{eq:cuasiaso3}\\
:a:bc:: & = :b:ca::+2\left(:c\left(T\left\langle a,b\right\rangle\right):-:b\left(T\left\langle a,c\right\rangle\right):\right)
% = -:b:ac::
,\label{eq:cuasicomaso}\\
:a:bc:: & = :c:ab::+2\left(:a\left(T\left\langle b,c\right\rangle\right):-:b\left(T\left\langle a,c\right\rangle\right):+:c\left(T\left\langle a,b\right\rangle\right):\right).\label{eq:cuasiasocom}
\end{align}
\end{lemma}

\begin{proof}
The identity \eqref{eq:cuasicom3} follows directly from quasicommutativity, while the identities \eqref{eq:cuasiaso2} and \eqref{eq:cuasiaso3} follow from quasiassociativity. The identity~\eqref{eq:cuasicomaso} follows applying \eqref{eq:cuasicom3} and \eqref{eq:cuasiaso2}, while~\eqref{eq:cuasiasocom} follows applying \eqref{eq:cuasiaso2} and \eqref{eq:cuasicom3} in that order.
\end{proof}

\begin{lemma}\label{lem:tech4bis}
For $f \in C^\infty(M,\C)$ and $a,b,c \in \Omega^0(\Pi E)$, the following identities hold:
\begin{align}
::fa::bc:: &= :a:b:fc:::+2\left(:\left(Tf\right)\left(:\left\langle a,b\right\rangle c:-:b\left\langle a,c\right\rangle:\right):\right),\label{eq:fabcabfc}\\
:a:b:fc::: &= -:b:a:fc:::+2:\left(T\left\langle a,b\right\rangle\right):fc::,\label{eq:15}\\
:a:b:fc::: &= :a::bf:c::-2:\left(Tf\right):\left\langle b,c\right\rangle a::\label{eq:16}.
\end{align}
\end{lemma}

\begin{proof}
The first identity follows from quasiassociativity, and applying identities \eqref{eq:abffab}, \eqref{eq:abfabf} and \eqref{eq:affa} in that order. The second one follows from quasiassociativity,  \eqref{eq:cuasicom}, and quasiassociativity, in that order. The last one follows from quasiassociativity.
\end{proof}

\begin{lemma}\label{lem:tech5}
For $a,b,c,d \in \Omega^0(\Pi E)$, the following identities hold:
\begin{align}
::a:bc::d: & = :a::bc:d::+2\Big(\!:(Ta)\left(\left\langle d,c\right\rangle b-\left\langle d,b\right\rangle c\right):+:T\left(:bc:\right)\left\langle a,d\right\rangle:\!\Big),\label{eq:18}\\
:a::bc:d:: & = -:::bc:d:a:\nonumber\\
&\hspace*{-4ex}
+2\Big(\!:\left\langle a,b\right\rangle \left(T\left(:cd:\right)\right):-:\left\langle a,c\right\rangle \left(T\left(:bd:\right)\right):+:\left(T\left(:bc:\right)\right)\left\langle a,d\right\rangle:\!\Big).
\label{eq:19}
\end{align}
\end{lemma}
\begin{proof}
The first identity follows from quasiassociativity, while the second one follows from quasicommutativity. In both cases, we use the non-commutative Wick formula, combined with \eqref{eq:abffab} and \eqref{eq:TabffTab}.
\end{proof}

We will need a variant of Jacobi's classical formula for the sheaf $M_{n\times n}(\Omega^{\ch}_E)$ of matrix-valued sections of the chiral de Rham complex $\Omega^{\ch}_E$. 
We extend $S,T\colon\Omega^{\ch}_E\to\Omega^{\ch}_E$ to 
\[
S,T\colon M_{n\times n}(\Omega^{\ch}_E)\lto M_{n\times n}(\Omega^{\ch}_E),
\]
given by their action on the matrix entries, that is, $SA=(SA_{jk}), TA=(TA_{jk})$, for all local sections $A=\(A_{jk}\)_{1\leq j,k\leq n}$ of $M_{n\times n}(\Omega^{\ch}_E)$, 
and the normally ordered product on $\Omega^{\ch}_E$ to
\[
M_{n\times n}(\Omega^{\ch}_E)\times M_{n\times n}(\Omega^{\ch}_E)\lto M_{n\times n}(\Omega^{\ch}_E), \quad (A,B)\longmapsto :AB:,
\]
given by combining the matrix multiplication and the normally ordered product of the matrix entries.
Furthermore, we extend the embeddings $\iota$ and $\jmath$ of Proposition~\ref{prop:CDRE} to
\[
\iota\colon M_{n\times n}(C^\infty(M,\C))\lhra M_{n\times n}(\Omega^{\ch}_E),
\quad
\jmath\colon M_{n\times n}(\Pi E)\lhra M_{n\times n}(\Omega^{\ch}_E).
\]
given by applying $\iota$ and $\jmath$ to the matrix entries, though the symbols $\iota$, $\jmath$ will be omitted.

% \begin{lemma}\label{lem:matrixid}
% For all $A \in M_{n\times n}(C^\infty(M,\C))$, we have
% \begin{equation*}
% T\det A=\tr(:\operatorname{adj}(A)(TA):),
% \end{equation*}
% where $\operatorname{adj}(A)\in M_{n\times n}\left(C^\infty(M,\C)\right)$ is the adjugate matrix of $A$. If moreover $\det A>0$ everywhere on $M$, then 
% \[
% T\log\det A=\tr(:A^{-1}(TA):).
% \]
% \end{lemma}

\begin{lemma}\label{lem:matrixid}
For all $A \in M_{n\times n}(C^\infty(M,\C))$ such that $\det A>0$ everywhere on $M$,
\[
T(\log\det A)=\tr(:A^{-1}(TA):).
\]
\end{lemma}

\begin{proof}
Observe that $S\colon\Omega^{\ch}_E\to\Omega^{\ch}_E$ satisfies the chain rule
\begin{equation}\label{eq:proof:lem:matrixid.1}
S{f}=\(\partial_jh\circ y\)\,Sy^j=:\(\partial_jh\circ y\)\,(Sy^j):,
\end{equation}
for all smooth $y=(y^1,\ldots,y^p)\colon M\to\C^p$, $h\colon \C^p\to\C$, where $f=h\circ y\in C^\infty(M,\C)$ and $\partial_jh=\partial h/\partial y^j$ (recall we are using the Einstein summation convention for repeated indices). 
The first identity in~\eqref{eq:proof:lem:matrixid.1} follows because $S$ restricts to $\frac{1}{2}\mathcal{D}\colon C^\infty(M,\C)\to\Omega^0(\Pi E)$, by part (2) of Proposition~\ref{prop:CDRE}, where the operator $\mathcal{D}$ in~\eqref{eq:parity-reversed-D.1} is the composite of the $C^\infty(M,\C)$-linear map $\Omega^1(\CC)\to\Omega^0(\Pi E)$ induced by $\pi^*$, $\left\langle\cdot,\cdot\right\rangle$ and $\Pi$, and the exterior differential $\du{}\colon C^\infty(M,\C)\to\Omega^1(\CC)$, that certainly satisfies the same chain rule.
The second identity in~\eqref{eq:proof:lem:matrixid.1} follows now from part (2) of Proposition~\ref{prop:CDRE} (recall we are omitting the symbols $\iota$ and $\jmath$).
Next we show that $T=S^2$ also satisfies the chain rule
\begin{equation}\label{eq:proof:lem:matrixid.2}
Tf=:\(\partial_jh\circ y\)\,(Ty^j):.
\end{equation}
Applying $S$ to~\eqref{eq:proof:lem:matrixid.1}, using the Leibniz rule for the derivation $S$ with respect to the normally ordered product, and the chain rule~\eqref{eq:proof:lem:matrixid.1} with $h$ replaced by $\partial_jh$, we obtain 
\begin{align*}
Tf&=S(Sf)=S\(:\!(\partial_jh\circ y)\,(Sy^j)\!:\)
=:\!\(\partial_jh\circ y\)\,S(Sy^j)\!:+:\!S\(\partial_jh\circ y\)\,(Sy^j)\!:
\\&
=:\(\partial_jh\circ y\)\,(Ty^j)\!:+::\!\(\partial_k\partial_jh\circ y\)(Sy^k)\!:\!(Sy^j)\!:=:\(\partial_jh\circ y\)\,(Ty^j):,
\end{align*}
as required, because~\eqref{eq:afbafb}, the identity $\langle Sy^k,Sy^j\rangle=\frac{1}{4}\langle\mathcal{D}y^k,\mathcal{D}y^j\rangle=0$ (by the Courant algebroid axioms), the symmetry $\partial_k\partial_jh=\partial_j\partial_kh$ and~\eqref{eq:cuasicom} imply
\begin{align*}
::\!(\partial_k&\partial_jh\circ y)(Sy^k)\!:\!(Sy^j)\!:
=:(Sy^k):\!(\partial_k\partial_jh\circ y)(Sy^j)\!::+2:\!(T(\partial_k\partial_jh\circ y))\langle Sy^k,Sy^j\rangle:
\\&
=:(Sy^k):\!(\partial_k\partial_j h\circ y)(Sy^j)\!::
=-::(\partial_k\partial_jh\circ y)(Sy^j)\!:Sy^k\!:+2\langle Sy^k,(\partial_k\partial_j h\circ y)(Sy^j)\rangle
\\&
=-::(\partial_j\partial_kh\circ y)(Sy^j)\!:(Sy^k)\!:.
\end{align*}
% and so $::\!(\partial_k\partial_jh\circ y)(Sy^k)\!:\!(Sy^j)\!:=0$, because we are summing over repeated indices.

% We recall now that $\det A$ is a polynomial in the entries of the matrix $A$, so expanding $\det A=\sum_\ell\operatorname{adj}(A)^{\rm{T}}_{j\ell}A_{j\ell}$ (for any fixed index $j$, without the Einstein summation convention), where $\operatorname{adj}(A)\in M_{n\times n}\left(C^\infty(M,\C)\right)$ is the adjugate matrix of $A$, we obtain
% \begin{equation}\label{eq:proof:lem:matrixid.3}
% \frac{\partial\det A}{\partial{A_{jk}}}=\sum_{\ell}\(\frac{\partial\operatorname{adj}(A)^{\rm{T}}_{j\ell}}{\partial{A_{jk}}}\,A_{j\ell}+\operatorname{adj}(A)^{\rm{T}}_{j\ell}\frac{\partial A_{j\ell}}{\partial{A_{jk}}}\)
% =\operatorname{adj}(A)^{\rm{T}}_{jk},
% \end{equation}
% because the $(j,\ell)$-minor $\operatorname{adj}(A)^{\rm{T}}_{j\ell}$ of $A$ is $(-1)^{j+\ell}$ times the determinant obtained by striking out the $j$th row and the $\ell$th column of $A$, and so it does not depend on the matrix entry $A_{jk}$, i.e., $\partial\operatorname{adj}(A)^{\rm{T}}_{j\ell}/\partial{A_{jk}}=0$.
%
% Applying now the chain rule~\eqref{eq:proof:lem:matrixid.2} and~\eqref{eq:proof:lem:matrixid.3}, we obtain
% \begin{align*}
% T{\det A}&=\sum_{j,k}:\!\frac{\partial\det A}{\partial{A_{jk}}}\,T{A_{jk}}\!:
% =\sum_{j,k}:\!\operatorname{adj}(A)^{\rm{T}}_{jk}(T{A_{jk}})\!:=\tr\left(:\!\operatorname{adj}(A)(T{A})\!:\right).
% \end{align*}
%
Finally we recall that $\det A$ is a polynomial in the entries of the matrix $A$, and if $\det A>0$, then %~\eqref{eq:proof:lem:matrixid.3} implies
Jacobi's classical formula can be formulated as
\[
\frac{\partial\log\det A}{\partial{A_{jk}}}%=(\det A)^{-1}\operatorname{adj}(A)^{\rm{T}}_{jk}
=(A^{-1})^{\rm{T}}_{jk},
\]
so the required identity follows by applying the chain rule~\eqref{eq:proof:lem:matrixid.2}:
\[
T(\log\det A)=:\frac{\partial\log\det A}{\partial{A_{jk}}}\,(T{A_{jk}}):
=:(A^{-1})^{\rm{T}}_{jk}(T{A_{jk}}):=\tr\left(:\!A^{-1}(T{A})\!:\right).
\qedhere
\]
\end{proof}

\subsection{Frame identities}

Consider a decomposition $E = \ell\oplus\overline{\ell}\oplus C_-$ as in \eqref{eq:abstractdecomp}. We fix an isotropic frame $\left\lbrace e^j,e_j\right\rbrace_{j=1}^n \subset \Pi \ell\oplus \Pi \overline{\ell}$. We use the notation of Section \ref{sec:localgen}.

\begin{lemma}
The following identity holds:
\begin{equation}\label{eq:clave}
\begin{split}
:\left[e^j,e_k\right]_-\left[e^k,e_j\right]_-: & = T\left\langle\left[e^j,e_k\right]_-,\left[e^k,e_j\right]_-\right\rangle.
\end{split}
\end{equation}
\end{lemma}

\begin{proof}
This follows from \eqref{eq:cuasicom}, using the Courant algebroid axioms (Definition~\ref{defn:CA}).
\end{proof}

\begin{lemma}
For all $a \in \Omega^0(\Pi E)$, the following identities hold:
\begin{align}
2T\left\langle\left[a,e^j\right],e_j\right\rangle & = :\left[a,e^j\right]_{\ell}e_j:+:e^j\left[a,e_j\right]_{\overline{\ell}}:,\label{eq:identech1}\\
:\left[I_+a_+,e^j\right]_+e_j:+:e^j\left[I_+a_+,e_j\right]_+: & =:\left\langle\left[a_{\overline{\ell}},e^j\right],e^k\right\rangle:e_ke_j::+:\left\langle\left[a_{\ell},e_j\right],e_k\right\rangle:e^ke^j::\nonumber\\
& + 2T\left\langle\left[I_+a_+,e^j\right],e_j\right\rangle.\label{eq:identech3}
\end{align}
\end{lemma}
\begin{proof}
The first identity follows from \eqref{eq:afbafb} and \eqref{eq:affa}, while the second one also needs \eqref{eq:abffab}, \eqref{eq:bfafab} and \eqref{eq:abfabf}. In both cases, we use the Courant algebroid axioms.
\end{proof}

As in Section \ref{sec:localgen}, we set $w := I_+\left[e^j,e_j\right]_+ \in \Omega^0\left(\Pi C_+\right)$.

\begin{lemma}
The following identity holds:
\begin{equation}\label{eq:extraprop}
\begin{split}
:e^j\left[e_j,w\right]_{\overline{\ell}}:+:\left[e^j,w\right]_{\ell}e_j: & = :e^k\left(\mathcal{D}\left\langle w,e_k\right\rangle\right)_{\overline{\ell}}:+:\left(\mathcal{D}\left\langle w,e^k\right\rangle\right)_{\ell}e_k:\\
&-2T\left\langle\left[w,e^j\right],e_j\right\rangle.
\end{split}
\end{equation}
\end{lemma}

\begin{proof}
This follows from the Courant algebroid axioms, \eqref{eq:affa} and \eqref{eq:afbafb}.
\end{proof}

In the next two lemmas, we obtain basic frame identities under the assumption that $\ell\oplus\overline{\ell}$ satisfies the F-term condition \eqref{eq:Ftermabs}.

\begin{lemma}\label{lem:applem}
Assume that $\ell\oplus\overline{\ell}$ satisfies the F-term condition \eqref{eq:Ftermabs}. For all $a \in \Omega^0(\Pi E)$ and $j \in \{1,\ldots,n\}$, the following identities hold:
\begin{align}
\left[a,e^j\right]_{\overline{\ell}} & = \left[a_{\overline{\ell}},e^j\right]_{\overline{\ell}} \qquad\text{ or, equivalently,}\hspace*{-3ex}&\left[e^j,a\right]_{\overline{\ell}} &= \left[e^j,a_{\overline{\ell}}\right]_{\overline{\ell}};\label{eq:prop1}\\
\left[a,e_j\right]_{\ell} & = \left[a_{\ell},e_j\right]_{\ell} \;\qquad\text{ or, equivalently,}\hspace*{-3ex}&\left[e_j,a\right]_{\ell} &= \left[e_j,a_{\ell}\right]_{\ell}.\label{eq:prop2}
\end{align}
\end{lemma}

\begin{proof}
By the F-term condition and the Courant algebroid axioms
\begin{align*}
\left[a_{\ell},e^j\right]_{\overline{\ell}} & = \left[a_\ell,e_j\right]_{\ell} = 0,\\
\left[a_-,e^j\right]_{\overline{\ell}} & = \left\langle\left[a_-,e^j\right],e^k\right\rangle e_k = \left\langle a_-,\left[e^j,e^k\right]_{\ell}\right\rangle e_k = 0,\\
\left[a_-,e_j\right]_{\ell} & = \left\langle\left[a_-,e_j\right],e_k\right\rangle e^k = \left\langle a_-,\left[e_j,e_k\right]_{\overline{\ell}}\right\rangle = 0.
\qedhere
\end{align*}
%In conclusion, we have obtained the desired identities.
\end{proof}

The following identities will be used in Appendix \ref{app:MRP}, specifically to prove Lemma \ref{lem:pasointermedio}.

\begin{lemma}\label{lem:importprop}
Assume that $\ell\oplus\overline{\ell}$ satisfies the F-term condition \eqref{eq:Ftermabs}. For each $i \in \{1,\ldots,n\}$, define
\begin{align}
a_i =& :e_j:\left[e_k,e^j\right]_{\overline{\ell}}\left[e_i,e^k\right]_{\overline{\ell}}::,\nonumber\\
b_i =& :e_j:\left(\mathcal{D}\left\langle\left[e^k,e_i\right],e^j\right\rangle\right)_{\overline{\ell}}e_k::,\nonumber\\
A_i =& :\left[e_i,e_j\right]_{\overline{\ell}}\left(Se^j\right):+:e_j\left(S\left[e_i,e^j\right]_+\right):+:\left[e_i,e^j\right]_+\left(Se_j\right):+:e^j\left(S\left[e_i,e_j\right]_{\overline{\ell}}\right):,\nonumber\\
B_i =& :\left[e_i,e_j\right]_{\overline{\ell}}:e_k\left[e^j,e^k\right]_{\ell}::+:e_j:\left[e_i,e_k\right]_{\overline{\ell}}\left[e^j,e^k\right]_{\ell}::\nonumber\\
& +:e_j:e_k\left[e_i,\left[e^j,e^k\right]\right]_+::+:\left[e_i,e^j\right]_+:e^k\left[e_j,e_k\right]_{\overline{\ell}}::\nonumber\\
& + :e^j:\left[e_i,e^k\right]_+\left[e_j,e_k\right]_{\overline{\ell}}::+:e^j:e^k\left[e_i,\left[e_j,e_k\right]\right]::,\nonumber\\
C_i =& 2:e_j:e_k\left[e^j,\left[e_i,e^k\right]_-\right]_+::+:e_j:e_k\left[e_i,\left[e^j,e^k\right]\right]_-::\nonumber\\
&+2:\left[e_i,e_j\right]_{\overline{\ell}}:e^k\left[e^j,e_k\right]_-::+:\left[e_i,e^j\right]_-:e^k\left[e_j,e_k\right]_{\overline{\ell}}::\nonumber\\
&+2:e_j:\left[e_i,e^k\right]\left[e^j,e_k\right]_-::+:e^j:\left[e_i,e^k\right]_-\left[e_j,e_k\right]_{\overline{\ell}}::\nonumber\\
&+2:e_j:e^k\left[e_i,\left[e^j,e_k\right]_-\right]::\nonumber.
\end{align}
Then the following identities hold:
\begin{align}\label{eq:parte0}
0 &= a_i+b_i,\\
A_i &= 2:\left(T\left\langle e^k,\left[e_i,e_j\right]\right\rangle\right)\left[e_k,e^j\right]:+:\left(T\left\langle\left[e^k,e_i\right],e^j \right\rangle\right)\left[e_k,e_j\right]_{\overline{\ell}}:\nonumber\\
& + T\Big(:e_j\left(\left\langle\mathcal{D}\left\langle\left[e^k,e_i\right],e^j\right\rangle,e_k\right\rangle-\left\langle\mathcal{D}\left\langle e^j,\left[e_i,e_k\right]\right\rangle,e^k\right\rangle\right):\nonumber\\
& +:e^j\left(\left\langle\mathcal{D}\left\langle\left[e^k,e_i\right],e_j\right\rangle,e_k\right\rangle\right):\Big)+\frac{1}{2}b_i\nonumber\\
&+\left.:e^j:e^k\left(\left\langle\mathcal{D}\left\langle\left[e_i,e_j\right],e^r\right\rangle,e_k\right\rangle e_r\right)::\right.\nonumber\\
&+\left.:e_j\left(:e_k:\left(\frac12\left\langle\mathcal{D}\left\langle\left[e^j,e_i\right],e^k\right\rangle,e_m\right\rangle +\left\langle\mathcal{D}\left\langle\left[e_i,e^j\right],e_m\right\rangle,e^k\right\rangle\right)e^m::\right):\right.\nonumber\\
\label{eq:parte1}
&+\frac12:e_j:\left(\mathcal{D}\left\langle\left[e^k,e_i\right],e^j\right\rangle\right)_-e_k::+:e_k:\left(\left\langle\mathcal{D}\left[e_i,e^k\right],e_j\right\rangle\right)_-e^j::,\\
B_i &= a_i+2:e_j:e_k\left[e^j,\left[e_i,e^k\right]_-\right]_+::+4:\left(T\left\langle\left[e^j,e_i\right],e_k\right\rangle\right)\left[e^k,e_j\right]_+:\nonumber\\
&+2:\left(T\left\langle\left[e^j,e_i\right],e^k\right\rangle\right)\left[e_k,e_j\right]_{\overline{\ell}}:-2:e^j:e^k\left(\left\langle\mathcal{D}\left\langle e^m,\left[e_i,e_j\right] \right\rangle,e_k\right\rangle e_m\right)::\nonumber\\
&+\left.:e_j:e_k\left(\left(\left\langle\mathcal{D}\left\langle e^k,\left[e_i,e^j\right]\right\rangle,e_m\right\rangle - 2\left\langle\mathcal{D}\left\langle e_m,\left[e_i,e^j\right]\right\rangle,e^k\right\rangle\right)e^m\right)::\right.\label{eq:parte2},\\
C_i &= 2\bigg(\!\!:e_j:e_k\left[e^j,\left[e_i,e^k\right]_-\right]_-\!\!::+:e_j:e_k\left[e^j,\left[e_i,e^k\right]_-\right]_{\ell}\!\!::+:e_j:e^k\left[e_k,\left[e^j,e_i\right]_-\right]_-\!\!::\nonumber\\
&\qquad+\left.:e_j:\left[e_i,e^k\right]_-\left[e^j,e_k\right]_-::\right.%\nonumber\\&+\left.
+:e_j:e^k\left[e_i,\left[e^j,e_k\right]_-\right]_{\overline{\ell}}::\!\!\bigg)\nonumber\\
&+:e_j:\left(\mathcal{D}\left\langle\left[e_i,e^k\right],e^j\right\rangle\right)_-e_k::-2:e_k:\left(\mathcal{D}\left\langle\left[e_i,e^k\right],e_j\right\rangle\right)_-e^j::\nonumber\\
&+4:\left(T\left\langle\left[e^j,e_i\right],e^k \right\rangle\right)\left[e^k,e_j\right]_-:.\label{eq:parte3}
\end{align}
\end{lemma}

\begin{proof}
To prove \eqref{eq:parte0}, note that by the Courant algebroid axioms, \eqref{eq:affa} and \eqref{eq:16},
\[
a_i= :e_j:e_k\left[e_i,\left[e^j,e^k\right]\right]_{\overline{\ell}}::, 
\]
so by the Jacobi identity for the Dorfman bracket and \eqref{eq:prop1},
\[
a_i = :e_j:e_k\left[\left[e_i,e^j\right]_{\overline{\ell}},e^k\right]_{\overline{\ell}}::+:e_j:e_k\left[e^j,\left[e_i,e^k\right]_{\overline{\ell}}\right]_{\overline{\ell}}::.
\]
Using the Courant algebroid axioms, \eqref{eq:affa}, \eqref{eq:cuasicom}, \eqref{eq:cuasicomaso}, \eqref{eq:cuasiasocom}, \eqref{eq:fabcabfc} and \eqref{eq:16},
\[
:e_j:e_k\left[\left[e_i,e^j\right]_{\overline{\ell}},e^k\right]_{\overline{\ell}}::= -a_i-2b_i.
\]
Using the Courant algebroid axioms, \eqref{eq:affa}, \eqref{eq:cuasicom}, \eqref{eq:fabcabfc} and \eqref{eq:16},
\[
:e_j:e_k\left[e^j,\left[e_i,e^k\right]_{\overline{\ell}}\right]_{\overline{\ell}}::= -a_i-b_i.
\]
Hence we have $a_i+b_i=-2(a_i+b_i)$, which gives \eqref{eq:parte0}. 

To prove \eqref{eq:parte1}, we use the Courant algebroid axioms, \eqref{eq:claveAi} and \eqref{eq:13}, so
\begin{align*}
:e^j\left(S\left[e_i,e_j\right]_{\overline{\ell}}\right): &= :\left(Te^j\right)\left\langle\mathcal{D}\left\langle\left[e^k,e_i\right],e_j\right\rangle,e_k\right\rangle:-:\left(T\left\langle\left[e^k,e_i\right],e_j\right\rangle\right)\left[e^j,e_k\right]:\\
& +\frac12:e^j:\left(\mathcal{D}\left\langle e^k,\left[e_i,e_j\right]\right\rangle\right)e_k:: -:\left[e_i,e^j\right]_{\ell}\left(Se_j\right):,\\
:\left[e_i,e_j\right]_{\overline{\ell}}\left(Se^j\right): & = :\left(T\left\langle e^k,\left[e_i,e_j\right]\right\rangle\right)\left[e_k,e^j\right]:-:\left(Te_j\right)\left\langle\mathcal{D}\left\langle e^j,\left[e_i,e_k\right]\right\rangle,e^k\right\rangle:\\
&-\frac12:e_k:\left(\mathcal{D}\left\langle\left[e^k,e_i\right],e_j\right\rangle\right)e^j::-:e_j\left(S\left[e_i,e^j\right]_{\ell}\right):,\\
:e_j\left(S\left[e_i,e^j\right]_{\overline{\ell}}\right): &= :\left(Te_j\right)\left\langle\mathcal{D}\left\langle\left[e^k,e_i\right],e^j\right\rangle,e_k\right\rangle:-:\left(T\left\langle\left[e^k,e_i\right],e^j\right\rangle\right)\left[e_j,e_k\right]:\\
&+\frac12:e_j:\left(\mathcal{D}\left\langle e^k,\left[e_i,e^j\right]\right\rangle\right)e_k::-:\left[e_i,e^j\right]_{\overline{\ell}}\left(Se_j\right):.
\end{align*}
Using again the Courant algebroid axioms, \eqref{eq:affa}, \eqref{eq:cuasicom}, \eqref{eq:cuasicomaso} and \eqref{eq:fabcabfc}, we obtain \eqref{eq:parte1}. 

As for \eqref{eq:parte2}, by the Jacobi identity for the Dorfman bracket and \eqref{eq:prop1},
\begin{align*}
\left[e_i,\left[e^j,e^k\right]\right]_+ &= \left[\left[e_i,e^j\right],e^k\right]_{\overline{\ell}}+\left[\left[e_i,e^j\right]_{\overline{\ell}}\right]_{\ell}+\left[\left[e_i,e^j\right]_{\ell},e^k\right]_{\ell}+\left[\left[e_i,e^j\right]_-,e^k\right]_+\\
&+\left[e^k,\left[e_i,e^j\right]\right]_{\overline{\ell}}+\left[e^k,\left[e_i,e^j\right]_{\overline{\ell}}\right]_{\ell}+\left[e^k,\left[e_i,e^j\right]_{\ell}\right]_{\ell}+\left[e^k,\left[e_i,e^j\right]_-\right]_+,\\
\left[e_i,\left[e_j,e_k\right]_{\overline{\ell}}\right]_{\overline{\ell}} &= \left[\left[e_i,e_j\right]_{\overline{\ell}},e_k\right]_{\overline{\ell}}+\left[e_j,\left[e_i,e_k\right]_{\overline{\ell}}\right]_{\overline{\ell}}.
\end{align*}
By the Courant algebroid axioms, \eqref{eq:14}, \eqref{eq:affa}, \eqref{eq:fabcabfc} and \eqref{eq:15},
\begin{align*}
:e^j:e^k\left[\left[e_i,e_j\right]_{\overline{\ell}},e_k\right]_{\overline{\ell}}:: &= -:\left[e_i,e^j\right]_{\ell}:e^k\left[e_j,e_k\right]_{\overline{\ell}}::+2:\left(T\left\langle\left[e^k,e_i\right],e_j\right\rangle\right)\left[e^j,e_k\right]_{\ell}:\\
&+:e^k:e^j\left(\left\langle\mathcal{D}\left\langle e^m,\left[e_i,e_j\right]\right\rangle,e_k\right\rangle e_m\right)::.
\end{align*}
By the Courant algebroid axioms, \eqref{eq:affa} and \eqref{eq:16},
\begin{align*}
:e^j:e^k\left[e_j,\left[e_i,e_k\right]_{\overline{\ell}}\right]_{\overline{\ell}}:: &= -:e^j:\left[e_i,e^k\right]_{\ell}\left[e_j,e_k\right]_{\overline{\ell}}::-2:\left(T\left\langle\left[e^j,e_i\right],e_k\right\rangle\right)\left[e_j,e^k\right]_{\ell}:\\
&+:e^j:e^k\left(\left\langle\mathcal{D}\left\langle e^m,\left[e_i,e_k\right]\right\rangle,e_j\right\rangle e_m\right)::.
\end{align*}
By the Courant algebroid axioms, \eqref{eq:14}, \eqref{eq:affa}, \eqref{eq:cuasicom}, \eqref{eq:afbafb} and \eqref{eq:fabcabfc},
\begin{align*}
:e_j:e_k\left[\left[e_i,e^j\right]_{\overline{\ell}},e^k\right]_{\ell}:: &= -:\left[e_i,e^j\right]_{\overline{\ell}}:e^k\left[e_j,e_k\right]_{\overline{\ell}}::+2:\left(T\left\langle e^k,\left[e_i,e^j\right]\right\rangle\right)\left[e_j,e_k\right]_{\overline{\ell}}:\\
&+:e_j:e_k\left(\left\langle\mathcal{D}\left\langle e^k,\left[e_i,e^j\right]\right\rangle,e_m\right\rangle e^m\right)::.
\end{align*}
By the Courant algebroid axioms, \eqref{eq:14}, \eqref{eq:affa} and \eqref{eq:fabcabfc},
\begin{align*}
:e_j:e_k\left[\left[e_i,e^j\right]_{\ell},e^k\right]_{\ell}:: &= -:\left[e_i,e_j\right]_{\overline{\ell}}:e_k\left[e^j,e^k\right]_{\ell}::+2:\left(T\left\langle e_k,\left[e_i,e^j\right]\right\rangle\right)\left[e_j,e^k\right]_{\overline{\ell}}:\\
&-:e_j:e_k\left(\left\langle\mathcal{D}\left\langle e_m,\left[e_i,e^j\right]\right\rangle,e^k\right\rangle e^m\right)::.
\end{align*}
By the Courant algebroid axioms, \eqref{eq:affa}, \eqref{eq:cuasiasocom}, \eqref{eq:fabcabfc} and \eqref{eq:16},
\begin{align*}
:e_j:e_k\left[e^j,\left[e_i,e^k\right]_{\overline{\ell}}\right]_{\ell}:: &= -:e^j:\left[e_i,e^k\right]_{\overline{\ell}}\left[e_j,e_k\right]_{\overline{\ell}}::.
\end{align*}
By the Courant algebroid axioms, \eqref{eq:affa} and \eqref{eq:16},
\begin{align*}
:e_j:e_k\left[e^j,\left[e_i,e^k\right]_{\ell}\right]_{\ell}:: & = -:e_j:\left[e_i,e_k\right]_{\overline{\ell}}\left[e^j,e^k\right]_{\ell}::-2:\left(T\left\langle e_j,\left[e_i,e^k\right]\right\rangle\right)\left[e^j,e_k\right]_{\overline{\ell}}:\\
& +:e_j:e_k\left(\left\langle\mathcal{D}\left\langle e_m,\left[e_i,e^k\right]\right\rangle,e^j \right\rangle e^m\right)::.
\end{align*}
By the Courant algebroid axioms and \eqref{eq:cuasicomaso},
\begin{align*}
:e_j:e_k\left[\left[e_i,e^j\right]_-,e^k\right]_+::+:e_j:e_k\left[e^j,\left[e_i,e^k\right]_-\right]_+:: & = 2:e_j:e_k\left[e^j,\left[e_i,e^k\right]_-\right]_+::.
\end{align*}
In conclusion, by the Courant algebroid axioms, \eqref{eq:affa} and \eqref{eq:15}, we obtain \eqref{eq:parte2}. 

For the last identity \eqref{eq:parte3}, we apply the Jacobi identity for the Dorfman bracket, so 
\begin{align*}
\left[e_i,\left[e^j,e^k\right]\right]_- &= \left[\left[e_i,e^j\right]_{\overline{\ell}},e^k\right]_-+\left[\left[e_i,e^j\right]_-,e^k\right]_-+\left[e^j,\left[e_i,e^k\right]_{\overline{\ell}}\right]_-+\left[e^j,\left[e_i,e^k\right]_-\right]_-.
\end{align*}
By the Courant algebroid axioms and \eqref{eq:cuasicomaso},
\begin{align*}
:\left[e_i,e^j\right]_-:e^k\left[e_j,e_k\right]_{\overline{\ell}}:: &= :e^j:\left[e_i,e^k\right]_-\left[e_j,e_k\right]_{\overline{\ell}}::,\\
:e_j:e_k\left[\left[e_i,e^j\right]_{\overline{\ell}},e^k\right]_-:: &= :e_j:e_k\left[e^j,\left[e_i,e^k\right]_{\overline{\ell}}\right]_-::+:e_j:e_k\left(\mathcal{D}\left\langle\left[e_i,e^j\right],e^k\right\rangle\right)_-::,\\
:e_j:e_k\left[\left[e_i,e^j\right]_-,e^k\right]_-:: &= :e_j:e_k\left[e^j,\left[e_i,e^k\right]_-\right]_-::,\\
:e_j:e_k\left[e^j,\left[e_i,e^k\right]_-\right]_{\ell}:: &= :e_j:e_k\left[\left[e_i,e^j\right]_-,e^k\right]_{\ell}::.
\end{align*}
By the Courant algebroid axioms, \eqref{eq:affa} and \eqref{eq:16},
\begin{equation*}
\begin{split}
:e_j:\left[e_i,e^k\right]_{\overline{\ell}}\left[e^j,e_k\right]_-:: &= - :e_j:e_k\left[e^j,\left[e_i,e^k\right]_{\overline{\ell}}\right]_-::.
\end{split}
\end{equation*}
Using the Jacobi identity for the Dorfman bracket and the involutivity of $\overline{\ell}$,
\begin{align*}
\left[e^j,\left[e_k,e_i\right]_{\overline{\ell}}\right]_- &= \left[\left[e^j,e_k\right],e_i\right]_-+\left[e_k,\left[e^j,e_i\right]\right]_-,
\end{align*}
and by the Courant algebroid axioms, \eqref{eq:affa} and \eqref{eq:16},
\begin{align*}
:e_j:\left[e_i,e^k\right]_{\ell}\left[e^j,e_k\right]_-:: =& :e_j:e^k\left[\left[e^j,e_k\right]_{\ell},e_i\right]_-::-:e_j:e^k\left[e_i,\left[e^j,e_k\right]_-\right]_-::\\
+& :e_j:e^k\left[e_k,\left[e^j,e_i\right]_{\ell}\right]_-::+:e_j:e^k\left[e_k,\left[e^j,e_i\right]_-\right]_-::.
\end{align*}
Hence we have obtained the identity
\begin{align*}
C_i &= 2\left(:\left[e_i,e^j\right]_-:e^k\left[e_j,e_k\right]_{\overline{\ell}}::+:e_j:e_k\left[e^j,\left[e_i,e^k\right]_-\right]_-::\right.\\
&\left.+:e_j:e_k\left[\left[e_i,e^j\right]_-,e^k\right]_{\ell}::+:\left[e_i,e_j\right]_{\overline{\ell}}:e^k\left[e^j,e_k\right]_-::+:e_j:e^k\left[\left[e^j,e_k\right]_{\ell},e_i\right]_-::\right.\\
&+\left.:e_j:e^k\left[e_k,\left[e^j,e_i\right]_{\ell}\right]_-::+:e_j:e^k\left[e_k,\left[e^j,e_i\right]_-\right]_-::+:e_j:\left[e_i,e^k\right]_-\left[e^j,e_k\right]_-::\right.\\
&+\left.:e_j:e^k\left[e_i,\left[e^j,e_k\right]_-\right]_{\overline{\ell}}::+:e_j:e_k\left(\mathcal{D}\left\langle\left[e_i,e^j\right],e^k\right\rangle\right)_-::\right).
\end{align*}
Applying again the Courant algebroid axioms and \eqref{eq:cuasiasocom}, we have
\[
:\left[e_i,e^j\right]_-:e^k\left[e_j,e_k\right]_{\overline{\ell}}:: = -:\left[e_j,e_k\right]_{\overline{\ell}}:e^k\left[e_i,e^j\right]_-::+2:\left[e_i,e^j\right]_-\left(T\left\langle\left[e_j,e_k\right],e^k\right\rangle\right):.
\]
Now, by the Courant algebroid axioms \eqref{eq:14}, \eqref{eq:affa} and \eqref{eq:fabcabfc},
\[
:\left[e_i,e_j\right]_{\overline{\ell}}:e^k\left[e^j,e_k\right]_-:: = -:e_j:e^k\left[e_k,\left[e^j,e_i\right]_{\ell}\right]_-::+2:\left(T\left\langle e^k,\left[e_i,e_j\right]\right\rangle\right)\left[e^j,e_k\right]_-:.
\]
Finally, by the Courant algebroid axioms, \eqref{eq:14}, \eqref{eq:affa}, \eqref{eq:cuasiasocom} and \eqref{eq:fabcabfc},
\begin{align*}
:e_j:e^k\left[\left[e^j,e_k\right]_{\ell},e_i\right]_-:: &= :\left[e_i,e_k\right]_{\overline{\ell}}:e^k\left[e_i,e^j\right]_-::-:e^k:e_j\left(\mathcal{D}\left\langle e_i,\left[e^j,e_k\right]\right\rangle\right)_-::\\
& - 2:\left[e_i,e^k\right]_-\left(T\left\langle e^j,\left[e_k,e_j\right]\right\rangle\right):.
\end{align*}
In conclusion, by the Courant algebroid axioms, \eqref{eq:cuasicom} and \eqref{eq:cuasicomaso}, we obtain \eqref{eq:parte3}. 
\end{proof}

\begin{remark}\label{rem:dual-lem:importpro}
Since interchanging the roles of $\ell$ and $\overline{\ell}$ preserves the F-term condition \eqref{eq:Ftermabs}, Lemma~\ref{lem:importprop} provides new identities by simply replacing the dual isotropic frames $\{e^j\}_{j=1}^n$ and $\{e_j\}_{j=1}^n$ with each other.
One can therefore consider the following sections of the chiral de Rham complex, for each $i \in \{1,\ldots,n\}$:
\begin{align}
a^i =& :e^j:\left[e^k,e_j\right]_{\ell}\left[e^i,e_k\right]_{\ell}::,\nonumber\\
b^i =& :e^j:\left(\mathcal{D}\left\langle\left[e_k,e^i\right],e_j\right\rangle\right)_{\ell}e^k::,\nonumber\\
A^i =& :\left[e^i,e_j\right]_+\left(Se^j\right):+:e_j\left(S\left[e^i,e^j\right]_{\ell}\right):+:\left[e^i,e^j\right]_{\ell}\left(Se_j\right):+:e^j\left(S\left[e^i,e_j\right]_+\right):,\nonumber\\
B^i =& :\left[e^i,e_j\right]_+:e_k\left[e^j,e^k\right]_{\ell}::+:e_j:\left[e^i,e_k\right]_+\left[e^j,e^k\right]_{\ell}::\nonumber\\
& +:e_j:e_k\left[e^i,\left[e^j,e^k\right]\right]::+:\left[e^i,e^j\right]_{\ell}:e^k\left[e_j,e_k\right]_{\overline{\ell}}::\nonumber\\
& + :e^j:\left[e^i,e^k\right]_{\ell}\left[e_j,e_k\right]_{\overline{\ell}}::+:e^j:e^k\left[e^i,\left[e_j,e_k\right]\right]_+::,\nonumber\\
C^i =& 2:e^j:e^k\left[e_j,\left[e^i,e_k\right]_-\right]_+::+:e^j:e^k\left[e^i,\left[e_j,e_k\right]\right]_-::\nonumber\\
&+2:\left[e^i,e^j\right]_{\ell}:e_k\left[e_j,e^k\right]_-::+:\left[e^i,e_j\right]_-:e_k\left[e^j,e^k\right]_{\ell}::\nonumber\\
&+2:e^j:\left[e^i,e_k\right]\left[e_j,e^k\right]_-::+:e_j:\left[e^i,e_k\right]_-\left[e^j,e^k\right]_{\ell}::\nonumber\\
&+2:e^j:e_k\left[e^i,\left[e_j,e^k\right]_-\right]::\nonumber.
\end{align}
Then Lemma~\ref{lem:importprop} gives the following identities:
\begin{align}
0 &= a^i+b^i\nonumber,\\
A^i &= 2:\left(T\left\langle e_k,\left[e^i,e^j\right]\right\rangle\right)\left[e^k,e_j\right]:+:\left(T\left\langle\left[e_k,e^i\right],e_j \right\rangle\right)\left[e^k,e^j\right]_{\ell}:\nonumber\\
& + T\Big(:e^j\left(\left\langle\mathcal{D}\left\langle\left[e_k,e^i\right],e_j\right\rangle,e^k\right\rangle-\left\langle\mathcal{D}\left\langle e_j,\left[e^i,e^k\right]\right\rangle,e_k\right\rangle\right):\nonumber\\
& +:e_j\left(\left\langle\mathcal{D}\left\langle\left[e_k,e^i\right],e^j\right\rangle,e^k\right\rangle\right):\Big)+\frac{1}{2}b^i\nonumber\\
&+\left.:e_j:e_k\left(\left\langle\mathcal{D}\left\langle\left[e^i,e^j\right],e_m\right\rangle,e^k\right\rangle e^m\right)::\right.\nonumber\\
&+\left.:e^j\left(:e^k:\left(\frac12\left\langle\mathcal{D}\left\langle\left[e_j,e^i\right],e_k\right\rangle,e^m\right\rangle +\left\langle\mathcal{D}\left\langle\left[e^i,e_j\right],e^m\right\rangle,e_k\right\rangle\right)e_m::\right):\right.\nonumber\\
&+\frac12:e^j:\left(\mathcal{D}\left\langle\left[e_k,e^i\right],e_j\right\rangle\right)_-e^k::+:e^k:\left(\left\langle\mathcal{D}\left[e^i,e_k\right],e^j\right\rangle\right)_-e_j::\nonumber,\\
B^i &= a^i+2:e^j:e^k\left[e_j,\left[e^i,e_k\right]_-\right]_+::+4:\left(T\left\langle\left[e_j,e^i\right],e^k\right\rangle\right)\left[e_k,e^j\right]_+:\nonumber\\
&+2:\left(T\left\langle\left[e_j,e^i\right],e_k\right\rangle\right)\left[e^k,e^j\right]_{\ell}:-2:e_j:e_k\left(\left\langle\mathcal{D}\left\langle e_m,\left[e^i,e^j\right] \right\rangle,e^k\right\rangle e^m\right)::\nonumber\\
&+\left.:e^j:e^k\left(\left(\left\langle\mathcal{D}\left\langle e_k,\left[e^i,e_j\right]\right\rangle,e^m\right\rangle - \left\langle\mathcal{D}\left\langle e^m,\left[e^i,e_j\right]\right\rangle,e_k\right\rangle\right)e_m\right)::\right.\nonumber,\\
C^i &= 2\bigg(\!:e^j:e^k\left[e_j,\left[e^i,e_k\right]_-\right]_-::+:e^j:e^k\left[\left[e^i,e_j\right]_-,e_k\right]_{\overline{\ell}}::+:e^j:e_k\left[e^k,\left[e_j,e^i\right]_-\right]_-::\nonumber\\
&\qquad+:e^j:\left[e^i,e_k\right]_-\left[e_j,e^k\right]_-::%\right.\nonumber\\
%&
+:e^j:e_k\left[e^i,\left[e_j,e^k\right]_-\right]_{\ell}::\!\bigg)\nonumber\\
&+:e^j:\left(\mathcal{D}\left\langle\left[e^i,e_k\right],e_j\right\rangle\right)_-e^k::
-2:e_j:\left(\mathcal{D}\left\langle\left[e^i,e_k\right],e^j\right\rangle\right)_-e^k::\nonumber\\
&+4:\left(T\left\langle\left[e_j,e^i\right],e_k \right\rangle\right)\left[e_k,e^j\right]_-:.\nonumber
\end{align}
\end{remark}

\begin{lemma}\label{lem:rewriting-Fij}
Assume that $\ell\oplus\overline{\ell}$ satisfies the F-term condition \eqref{eq:Ftermabs}. Then the functions $F^{ij}$ and $F_{ij}$ defined in~\eqref{eq:Fijsuper.bis} and~\eqref{eq:Fijsub.bis} are given by
\begin{equation}\label{eq:rewriting-Fij}
\begin{split}
F^{ij} &= \left\langle\left[\left[e_k,e^k\right],e^i\right],e^j\right\rangle+\left\langle\mathcal{D}\left\langle\left[e^i,e^j\right],e_k\right\rangle,e^k\right\rangle,\\
F_{ij} &= \left\langle\left[\left[e^k,e_k\right],e_i\right],e_j\right\rangle+\left\langle\mathcal{D}\left\langle\left[e_i,e_j\right],e^k\right\rangle,e_k\right\rangle.\\
\end{split}
\end{equation}
\end{lemma}

\begin{proof}
By the F-term condition, $\langle[e^i,e^j],e^k\rangle=0$, so the Courant algebroid axioms imply
\begin{align*}
\langle\mathcal{D}\langle[e^i,e^j],&e_k\rangle,e^k\rangle
=\langle[[e^i,e^j],e_k],e^k\rangle+\langle[e_k,[e^i,e^j]],e^k\rangle
\\&
=\langle[[e^i,e^j],e_k],e^k\rangle+\langle\mathcal{D}\langle[e^i,e^j],e^k\rangle,e_k\rangle-\langle[e^i,e^j],[e_k,e^k]\rangle
\\&
=\langle[[e^i,e^j],e_k],e^k\rangle-\langle\mathcal{D}\langle e^j,[e_k,e^k]\rangle,e^i\rangle+\langle e^j,[e^i,[e_k,e^k]]\rangle
\\&
=\langle[[e^i,e^j],e_k],e^k\rangle-\langle\mathcal{D}\langle e^j,[e_k,e^k]\rangle,e^i\rangle-\langle[[e_k,e^k],e^i],e^j\rangle+\langle\mathcal{D}\langle[e_k,e^k],e^i\rangle,e^j\rangle.
\end{align*}
Hence
\begin{align*}
F^{ij}:=&\langle[[e^i,e^j],e_k],e^k\rangle+\left\langle\mathcal{D}\left\langle e^i,\left[e_k,e^k\right]\right\rangle,e^j\right\rangle-\left\langle\mathcal{D}\left\langle e^j,\left[e_k,e^k\right]\right\rangle,e^i\right\rangle,
\\
=& \left\langle\left[\left[e_k,e^k\right],e^i\right],e^j\right\rangle+\left\langle\mathcal{D}\left\langle\left[e^i,e^j\right],e_k\right\rangle,e^k\right\rangle,
\end{align*}
as required. The proof of the formula for $F_{ij}$ is similar. 
\end{proof}

\section{Proof of Proposition \ref{teo:mainresult2} and Proposition \ref{teo:mainresult}}\label{app:MRP}
 
\subsection{Current generator}

Let $\Omega^{\ch}_E$ be the chiral de Rham complex of a complex Courant algebroid $E$ over a smooth manifold $M$. We keep the notation introduced in Section \ref{sec:localgen}, and use the Einstein summation convention for repeated indices.

Define the following auxiliary local section of the chiral de Rham complex of $E$:
\begin{equation}
\label{eq:JHglocal}
J_0 = \frac{i}{2}:e^je_j:.
\end{equation}

\begin{lemma}
\label{lem:lempp}
Given $a \in \Omega^0(\Pi E)$, the following identities hold:
\begin{subequations}
\begin{align}
\label{eq:aJ}
\left[a_\Lambda{J_0}\right] & = \frac{i}{2}\left(:\left[a,e^j\right]e_j:+:e^j\left[a,e_j\right]:\right)+i\left(\chi I_+a_++\lambda\left\langle\left.\left[a,e^j\right]\right|e_j\right\rangle\right),\\
\label{eq:Ja}
\begin{split}
\left[{J_0}_\Lambda a\right] & = \frac{i}{2}\left(:\left[a,e^j\right]e_j:+:e^j\left[a,e_j\right]:\right)\\
& - \left(\left(\chi+S\right)I_+a_++\left(\lambda+T\right)\left\langle\left[a,e^j\right],e_j\right\rangle\right).
\end{split}
\end{align}
\end{subequations}
\end{lemma}

\begin{proof}
The identity~\eqref{eq:aJ} follows directly from the non-commutative Wick formula, since
\begin{align*}
\left[a_\Lambda:e^je_j:\right] & = :\left[a_\Lambda e^j\right]e_j:+:e^j\left[a_\Lambda e_j\right]:+\int_0^\Lambda d\Gamma\left[\left[a_\Lambda e^j\right]_\Gamma e_j\right]\\
& = :\left[a,e^j\right]e_j:+:e^j\left[a,e_j\right]:+2\left(\chi\left(\left\langle a,e^j\right\rangle e_j-\left\langle a,e_j\right\rangle e^j\right)+\lambda\left\langle\left[a,e^j\right],e_j\right\rangle\right)\\
& = :\left[a,e^j\right]e_j:+:e^j\left[a,e_j\right]:+2\chi I_+a_++2\lambda\left\langle\left[a,e^j\right],e_j\right\rangle.
%\intertext{so}
%\left[a_\Lambda{J_0}\right] & = \frac{i}{2}\left(:\left[a,e^j\right]e_j:+:e^j\left[a,e_j\right]:\right)+i\left(\chi I_+a_++\lambda\left\langle\left.\left[a,e^j\right]\right|e_j\right\rangle\right).
\end{align*}
The identity~\eqref{eq:Ja} is obtained applying skew-symmetry of the $\Lambda$-bracket to~\eqref{eq:aJ}.
\end{proof}

%In the next result we prove the first identity in \eqref{eq:firstident}.
Next, we define an auxiliary local section $H'$ of $\Omega^{\ch}_E$ (using~\eqref{eq:def.w}) and a scalar $c_0$ by
\begin{align}
\begin{split}\label{eq:def.H-prime}
H'&= \frac{1}{2}\left(:e_j\left(Se^j\right):+:e^j\left(Se_j\right):\right)\\
& + \frac{1}{4}\left(:e_j:e^k\left[e^j,e_k\right]::+:e^j:e_k\left[e_j,e^k\right]::\right.\\
&-\left.:e_j:e_k\left[e^j,e^k\right]::-:e^j:e^k\left[e_j,e_k\right]::\right)-\frac{T}{2}w,
\end{split}
\\
\label{eq:cc1}
c_0 &= 3\dim \ell = 3n \in \C.
\end{align}

\begin{proposition}
\label{lem:primerpaso}
One has
\begin{equation*}
\begin{split}
\left[{J_0}_\Lambda{J_0}\right] & = -\left(H'+\frac{\lambda\chi}{3}c_0\right).
\end{split}
\end{equation*}
\end{proposition}

\begin{proof}
Applying the non-commutative Wick formula and~\eqref{eq:Ja}, we obtain
\begin{align*}
\big[{J_0}_\Lambda&:e^je_j:\big] = :\left[{J_0}_\Lambda e^j\right]e_j:-:e^j\left[{J_0}_\Lambda e_j\right]:+\int_0^\Lambda d\Gamma\left[\left[{J_0}_\Lambda e^j\right]_\Gamma e_j\right]\\
= & \frac{i}{2}\Big(::\left[e^j,e^k\right]e_k:e_j:+::e^k\left[e^j,e_k\right]:e_j:-:e^j:\left[e_j,e^k\right]e_k::-:e^j:e^k\left[e_j,e_k\right]::\Big)\\
& +i\lambda\Big(\left\langle e_j,\left[e^k,e_k\right]\right\rangle e^j-\left\langle e^j,\left[e^k,e_k\right] \right\rangle e_j\Big)\\
& +i\Big(:e^j\left(T\left\langle e_j,\left[e^k,e_k\right]\right\rangle\right):-:\left(T\left\langle e^j,\left[e^k,e_k\right]\right\rangle\right)e_j:+:e^j\left(Se_j\right):+:\left(Se^j\right)e_j:\Big)\\
& +\frac{i}{2}\int_0^\Lambda d\Gamma I_1+i\int_0^\Lambda d\Gamma I_2,
\end{align*}
where the two integrals are defined and calculated as follows, using sesquilinearity, skew-symmetry and the non-commutative Wick formula:
\begin{align*}
I_1^1 :&= \left[{e_j}_\Gamma:\left[e^j,e^k\right]e_k:\right] = :\left[{e_j}_\Gamma\left[e^j,e^k\right]\right]e_k:+:\left[e^j,e^k\right]\left[{e_j}_\Gamma e_k\right]:\\
& +\int_0^\Gamma d\Omega\left[\left[{e_j}_\Gamma\left[e^j,e^k\right]\right]_\Omega e_k\right] = :\left[e_j,\left[e^j,e^k\right]\right]e_k:+:\left[e^j,e^k\right]\left[e_j,e_k\right]:\\
&+2\eta\left\langle e_j,\left[e^j,e^k\right]\right\rangle e_k+2\gamma\left\langle\left[{e_j},\left[e^j,e^k\right]\right],e_k\right\rangle,\\
I_1^2 :&= \left[{e_j}_\Gamma:e^k\left[e^j,e_k\right]:\right] = :\left[{e_j}_\Gamma e^k\right]\left[e^j,e_k\right]:+:e^k\left[{e_j}_\Gamma\left[e^j,e_k\right]\right]:\\
&+\int_0^\Gamma d\Omega\left[\left[{e_j}_\Gamma e^k\right]_\Omega\left[e^j,e_k\right]\right] = :\left[e_j,e^k\right]\left[e^j,e_k\right]:+:e^k\left[e_j,\left[e^j,e_k\right]\right]:\\
&+2\eta\left(\left[e^j,e_j\right]-\left\langle e_j,\left[e^j,e_k\right]\right\rangle e^k\right)+2\gamma\left\langle\left[{e_j}, e^k\right],\left[e^j,e_k\right]\right\rangle,\\
\int_0^\Lambda d\Gamma I_1 :&= \int_0^\Lambda d\Gamma\left[\({:\left[e^j,e^k\right]e_k:+:e^k\left[e^j,e_k\right]:}\)_\Gamma e_j\right] = -\int_0^\Lambda d\Gamma I_1^1-\int_0^\Lambda d\Gamma I_1^2\\
& = 2\lambda\left(\left[e^j,e_j\right]_{\overline{\ell}}-\left[e^j,e_j\right]-\left[e^j,e_j\right]_\ell\right) = -4\lambda\left[e^j,e_j\right]_\ell,\\ 
\end{align*}
and
\begin{align*}
I_2^1 :&= \left[{\left(\chi+S\right)e^j}_\Gamma e_j\right] = \left(\eta-\chi\right)\left(\left[e^j,e_j\right]+2\eta\left\langle e^j,e_j\right\rangle\right)\\
& = \eta\left(\left[e^j,e_j\right]+2\chi\dim\ell\right)-2\gamma\dim\ell-\chi\left[e^j,e_j\right],\\
I_2^2 :&= \left[{T\left\langle e^j,\left[e^k,e_k\right]\right\rangle}_\Gamma e_j\right] = -\gamma\left[{\left\langle e^j,\left[e^k,e_k\right]\right\rangle}_\Gamma e_j\right]=-\gamma\left\langle\mathcal{D}{\left\langle e^j,\left[e^k,e_k\right]\right\rangle},e_j\right\rangle,\\
\int_0^\Lambda d\Gamma I_2 :&= \int_0^\Lambda d\Gamma\left(I_2^1+I_2^2\right) = \lambda\left[e^j,e_j\right]+2\lambda\chi\dim\ell.
\end{align*}
Hence 
\begin{align*}
\left[{J_0}_\Lambda{J_0}\right] & = \frac{i}{2}\left[{J_0}_\Lambda:e^je_j:\right]\\
& = -\frac{1}{4}\left(::\left[e^j,e^k\right]e_k:e_j:+::e^k\left[e^j,e_k\right]:e_j:-:e^j:\left[e_j,e^k\right]e_k::\right.\\
& -\left.:e^j:e^k\left[e_j,e_k\right]::\right)-\frac{1}{2}\lambda\left(\left[e^j,e_j\right]_\ell-\left[e^j,e_j\right]_{\overline{\ell}}\right)\\
& -\frac{1}{2}\left(:e^j\left(T\left\langle e_j,\left[e^k,e_k\right]\right\rangle\right):-:\left(T\left\langle e^j,\left[e^k,e_k\right]\right\rangle\right)e_j:\right.\\
& +\left.:e^j\left(Se_j\right):+:\left(Se^j\right)e_j:\right)+\lambda\left[e^j,e_j\right]_\ell-\frac{1}{2}\lambda\left[e^j,e_j\right]-\lambda\chi\dim\ell\\
& = -\frac{1}{4}\left(::\left[e^j,e^k\right]e_k:e_j:+::e^k\left[e^j,e_k\right]:e_j:-:e^j:\left[e_j,e^k\right]e_k::\right.\\
& -\left.:e^j:e^k\left[e_j,e_k\right]::\right)-\frac{1}{2}\left(:e^j\left(Se_j\right):+:\left(Se^j\right)e_j:\right)\\
& -\frac{1}{2}\left(:e^j\left(T\left\langle e_j,\left[e^k,e_k\right]\right\rangle\right):-:\left(T\left\langle e^j,\left[e^k,e_k\right]\right\rangle\right)e_j:\right)-\lambda\chi\dim\ell.
\end{align*}
Applying \eqref{eq:cuasicom}, \eqref{eq:cuasicom2}, \eqref{eq:cuasicom3}, we rewrite the following terms, that we have just obtained:
\begin{align*}
:\left(Se^j\right)e_j: & = :e_j\left(Se^j\right):+T\left[e^j,e_j\right],\\
-:e^j:\left[e_j,e^k\right]e_k:: & = :e^j:e_k\left[e_j,e^k\right]::-2:e^j\left(T\left\langle e_j,\left[e^k,e_k\right]\right\rangle\right):,\\
::e^k\left[e^j,e_k\right]:e_j: &= :e_j:e^k\left[e^j,e_k\right]::-2T\left(\left[e^j,e_j\right]+\left[e^j,e_j\right]_\ell\right),\\
::\left[e^j,e^k\right]e_k:e_j: &=-:e_j:e_k\left[e^j,e^k\right]::+2T\left[e^j,e_j\right]_{\overline{\ell}}+2:\left(T\left\langle e^j,\left[e^k,e_k\right]\right\rangle\right)e_j:.
\end{align*}
The required formula follows now by substituting these expressions in the above calculation:
\begin{align*}
\left[{J_0}_\Lambda{J_0}\right] & = -\frac{1}{4}\left(::\left[e^j,e^k\right]e_k:e_j:+::e^k\left[e^j,e_k\right]:e_j:-:e^j:\left[e_j,e^k\right]e_k::\right.\\
& -\left.:e^j:e^k\left[e_j,e_k\right]::\right)-\frac{1}{2}\left(:e^j\left(Se_j\right):+:\left(Se^j\right)e_j:\right)-\frac{1}{2}Tw-\lambda\chi\dim\ell,\\
& = -\left(H'+\frac{\lambda\chi}{3}c_0\right).
\qedhere
\end{align*}
\end{proof}

\subsection{Neveu--Schwarz generator from the F-term condition}

In this section we give the proofs of Propositions \ref{teo:mainresult2} and \ref{teo:mainresult}.
We need to calculate the $\Lambda$-brackets~\eqref{eq:JJcorregido} and~\eqref{eq:teo:mainresult}, under the assumption that $\ell\oplus\overline{\ell}$ satisfies the F-term condition \eqref{eq:Ftermabs}. 
% In the sequel, we assume that $\ell \oplus \overline{\ell}$ satisfies the F-term condition \eqref{eq:Ftermabs}.

Consider the following auxiliary local section of the chiral de Rham complex of $E$:
\begin{equation}\label{eq:def.H_0}
H_0:= H'+\frac{T}{2}w,
\end{equation}

\begin{lemma}\label{lem:lempi}
Assume that $\ell \oplus \overline{\ell}$ satisfies the F-term condition \eqref{eq:Ftermabs}. Then
\begin{equation*}
\begin{split}
H_0 &= \frac{1}{2}\Big(:e_j\left(Se^j\right):+:e^j\left(Se_j\right):\Big)\\
&\, + \frac{1}{4}\Big(2:e_j:e^k\left[e^j,e_k\right]_-::+:e_j:e_k\left[e^j,e^k\right]_\ell::+:e^j:e^k\left[e_j,e_k\right]_{\overline{\ell}}::\Big).\\
\end{split}
\end{equation*}
\end{lemma}

\begin{proof}
Applying \eqref{eq:cuasicomaso}, we obtain
\begin{equation*}
\begin{split}
:e_j:e^k\left[e^j,e_k\right]:: &= -:e^k:e_j\left[e^j,e_k\right]:: =:e^k:e_j\left[e_k,e^j\right]::.
\end{split}
\end{equation*}
Hence, using the F-term condition, we have
\begin{equation*}
\begin{split}
H_0 &= \frac{1}{2}\left(:e_j\left(Se^j\right):+:e^j\left(Se_j\right):\right)\\
& + \frac{1}{4}\left(2:e_j:e^k\left[e^j,e_k\right]::-:e_j:e_k\left[e^j,e^k\right]_\ell::-:e^j:e^k\left[e_j,e_k\right]_{\overline{\ell}}::\right).\\
\end{split}
\end{equation*}
Moreover, using \eqref{eq:cuasiasocom} and \eqref{eq:fabcabfc}, we obtain
\begin{equation*}
\begin{split}
:e^j:e_k\left[e_j,e^k\right]_{\overline{\ell}}:: & = -:\left[e_j,e^k\right]_{\overline{\ell}}:e_ke^j:: -2:e_k\left(T\left\langle e^j,\left[e_j,e^k\right]\right\rangle\right):\\
& = -:\left(\left\langle e^m,\left[e_j,e^k\right]\right\rangle e_m\right):e_ke^j::-2:e_k\left(T\left\langle \left[e^j,e_j\right],e^k\right\rangle\right):\\
& = :e_m:e_k\left(\left\langle e_j,\left[e^m,e^k\right]\right\rangle e^j\right)::\\
&+2\left(:\left(T\left\langle e^j,\left[e_j,e^k\right] \right\rangle\right)e_k:-:e_k\left(T\left\langle \left[e^j,e_j\right],e^k\right\rangle\right):\right)\\
&=:e_j:e_k\left[e^j,e^k\right]_\ell::.
\end{split}
\end{equation*}
Analogously, using \eqref{eq:cuasicom}, \eqref{eq:cuasiasocom} and \eqref{eq:fabcabfc}, we obtain
\begin{equation*}
\begin{split}
:e^j:e_k\left[e_j,e^k\right]_\ell:: & = -:\left[e_j,e^k\right]_\ell:e_ke^j:: +2:e^j\left(T\left\langle e_k,\left[e_j,e^k\right]\right\rangle\right):\\
& = :\left(\left\langle e_m,\left[e_j,e^k\right]\right\rangle e^m\right):e^je_k::+2:e^j\left(T\left\langle e_j,\left[e^k,e_k\right]\right\rangle\right):\\
& = :e^m:e^j\left(\left\langle \left[e_m,e_j\right],e^k\right\rangle e_k\right)::\\
&-2\left(:\left(T\left\langle e_k,\left[e_j,e^k\right] \right\rangle\right)e^j:-:e^j\left(T\left\langle e_j,\left[e^k,e_k\right]\right\rangle\right):\right)\\
& = :e^j:e^k\left[e_j,e_k\right]_{\overline{\ell}}::,
\end{split}
\end{equation*}
and the proof follows.
\end{proof}

\begin{remark}\label{obs:app.MSobs}
% Notice that we can interchange $\ell$ and $\overline{\ell}$ in the expression for $H_0$ and $H'$, obtaining the same sections.
Interchanging $\ell$ and $\overline{\ell}$, and replacing the corresponding dual isotropic frames $\{\epsilon_j\}_{j=1}^n$ and $\{\overline{\epsilon}_j\}_{j=1}^n$ with each other, the sections $w, J_0, H', H_0$ and the scalar $c_0$ change into $\overline{w}=w$, $\overline{J}=-J$, $\overline{H'}=H', \overline{H}_0=H_0$ and $\overline{c}_0=c_0$, respectively. This follows directly from their definitions~\eqref{eq:def.w},~\eqref{eq:JHglocal},~\eqref{eq:def.H-prime},~\eqref{eq:cc1},~\eqref{eq:def.H_0}. 
In particular, since $\overline{H_0}=H_0$ and $\overline{H'}=H'$, a formula for the $\Lambda$-bracket $\left[{H_0}_\Lambda e_i \right]$ can be deduced from one for $\left[{H_0}_\Lambda e^i \right]$ for $i \in \{1,\ldots,n\}$, and vice versa, and the expressions $\left[{H'}_\Lambda e_i \right]$ and $\left[{H'}_\Lambda e^i \right]$ have a similar relation. The aim of the next result is precisely to calculate these $\Lambda$-brackets. 
\end{remark}

\begin{proposition}\label{lem:pasointermedio}
Assume that $\ell\oplus\overline{\ell}$ satisfies the F-term condition \eqref{eq:Ftermabs}. For each $i \in \{1,\ldots,n\}$, the following identities hold:
\begin{subequations}
\begin{align}
\left[{H_0}_\Lambda e_i \right]& = \frac{\lambda}{2}\bigg(\Big(\left\langle\mathcal{D}\left\langle\left[e^k,e_i\right],e^j\right\rangle,e_k\right\rangle-\left\langle\mathcal{D}\left\langle e^j,\left[e_i,e_k\right]\right\rangle,e^k\right\rangle\Big)e_j
-\left\langle\mathcal{D}\left\langle\left[e^k,e_i\right],e_j\right\rangle,e_k\right\rangle e^j\bigg)
\nonumber
\\
&- \frac{1}{2}\left(:e_j:e_k\left[\left[e_i,e^j\right]_-,e^k\right]_\ell::%\right.\nonumber\\
%&+ \left.
+:e_j:e_k\left[\left[e_i,e^j\right]_-,e^k\right]_-::\right.\nonumber \\
&+:\left[e_i,e^j\right]_-:e_k\left[e_j,e^k\right]_-::+ :e_j:e^k\left[e_i,\left[e^j,e_k\right]_-\right]_{\overline{\ell}}::+:e^j:e_k\left[\left[e^k,e_i\right]_-,e_j\right]_-::\!\!\bigg)\nonumber\\
& + \frac{1}{2}\bigg(\chi:e_j\left[e_i,e^j\right]_-:-2:e_j\left(S\left[e_i,e^j\right]_-\right): + 2T\left[e_j,\left[e_i,e^j\right]_-\right]_{\overline{\ell}}\nonumber\\
& + \lambda\Big(\left[e_j,\left[e_i,e^j\right]_-\right]_{\overline{\ell}}+\left[\left[e_i,e^j\right]_-,e_j\right]_-\Big)\bigg)+\left(\lambda+2T+\chi S\right)e_i,
\label{eq:NSbasis10}\\
\left[{H_0}_\Lambda e^i\right] & = \frac{\lambda}{2}\bigg(\Big(\left\langle\mathcal{D}\left\langle\left[e_k,e^i\right],e_j\right\rangle,e^k\right\rangle-\left\langle\mathcal{D}\left\langle e_j,\left[e^i,e^k\right]\right\rangle,e_k\right\rangle\Big)e^j-\left\langle\mathcal{D}\left\langle\left[e_k,e^i\right],e^j\right\rangle,e^k\right\rangle e_j\bigg)\nonumber\\
&-\frac{1}{2}\left(:e^j:e^k\left[\left[e^i,e_j\right]_-,e_k\right]_{\overline{\ell}}::+:e^j:e^k\left[\left[e^i,e_j\right]_-,e_k\right]_-::\right.\nonumber\\
&\left.+:\left[e^i,e_j\right]_-:e^k\left[e^j,e_k\right]_-::+:e^j:e_k\left[e^i,\left[e_j,e^k\right]_-\right]_\ell::+:e_j:e^k\left[\left[e_k,e^i\right]_-,e^j\right]_-::\right)\nonumber\\
&+\frac{1}{2}\left(\chi:e^j\left[e^i,e_j\right]_-:-2:e^j\left(S\left[e^i,e_j\right]_-\right):+2T\left[e^j,\left[e^i,e_j\right]_-\right]_\ell \right.\nonumber\\
&+\left.\lambda\left(\left[e^j,\left[e^i,e_j\right]_-\right]_\ell+\left[\left[e^i,e_j\right]_-,e^j\right]_-\right)\right)+\left(\lambda+2T+\chi S\right)e^i,\nonumber
%\label{eq:NSbasis20}
\\%\end{align}\begin{align}  
\left[{H'}_\Lambda e_i\right] & = \left[{H_0}_\Lambda e_i\right] - \frac{\lambda}{2}\Big(\left\langle\left[e^j,e_j\right],e^k\right\rangle\left[e_k,e_i\right]_{\overline{\ell}}-\left\langle\left[e^j,e_j\right],e_k\right\rangle\left[e^k,e_i\right]\Big)\nonumber\\
& +\frac{\lambda}{2}\Big(\left\langle\mathcal{D}\left\langle e^j,\left[e^k,e_k\right]\right\rangle,e_i\right\rangle e_j - \left\langle\mathcal{D}\left\langle e_j,\left[e^k,e_k\right]\right\rangle,e_i\right\rangle e^j
+\mathcal{D}\left\langle e_i,\left[e^j,e_j\right]\right\rangle\Big)
\nonumber\\
&+\lambda\chi\left\langle\left[e^j,e_j\right],e_i\right\rangle,
\label{eq:NSbasis1}\\
\left[{H'}_\Lambda e^i\right] & = \left[{H_0}_\Lambda e^i\right] + \frac{\lambda}{2}\Big(\left\langle\left[e^j,e_j\right],e_k\right\rangle\left[e^k,e^i\right]_\ell-\left\langle\left[e^j,e_j\right],e^k\right\rangle\left[e_k,e^i\right]\Big)\nonumber\\
& -\frac{\lambda}{2}\Big( \left\langle\mathcal{D}\left\langle e_j,\left[e^k,e_k\right]\right\rangle,e^i\right\rangle e^j- \left\langle\mathcal{D}\left\langle e^j,\left[e^k,e_k\right]\right\rangle,e^i\right\rangle e_j+\mathcal{D}\left\langle e^i,\left[e^j,e_j\right]\right\rangle\Big)\nonumber\\
&-\lambda\chi\left\langle\left[e^j,e_j\right],e^i\right\rangle.\nonumber
\end{align}
\end{subequations}
\end{proposition}

\begin{proof}
We will need some identities collected in Lemma \ref{lem:importprop}.
To prove~\eqref{eq:NSbasis10}, note that 
\[
\left[{H_0}_\Lambda e_i\right] = -\left[{e_i}_{-\Lambda-\nabla}{H_0}\right]
\]
by skew-symmetry of the $\Lambda$-bracket, so we need to compute
\[
\left[{e_i}_\Lambda{H_0}\right] = \frac{1}{2}\Upsilon_i^1+\frac{1}{4}\Upsilon_i^2,
\]
where $\Upsilon_i^1$ and $\Upsilon_i^2$ are defined and calculated as follows, using the formula for $H_0$ in Lemma \ref{lem:lempi}. First, we compute
\[
\Upsilon_i^1 := \left[{e_i}_\Lambda\left(:e_j\left(Se^j\right):+:e^j\left(Se_j\right):\right)\right] = \Upsilon_i^{1,1}+\Upsilon_i^{1,2}.
\]
Applying the non-commutative Wick formula once on each summand, by sesquilinearity,
\begin{align*}
\Upsilon_i^{1,1} :&=\left[{e_i}_\Lambda:e_j\left(Se^j\right):\right]\\
&=:\left[e_i,e_j\right]_{\overline{\ell}}\left(Se^j\right):+:e_j\left(S\left[e_i,e^j\right]\right):-\chi:e_j\left[e_i,e^j\right]:\\
&+\lambda\left(2e_i+\left[\left[e_i,e_j\right]_{\overline{\ell}},e^j\right]+\mathcal{D}\left\langle e_i,\left[e^j,e_j\right]\right\rangle\right),\\
\Upsilon_i^{1,2} :&= \left[{e_i}_\Lambda:e^j\left(Se_j\right):\right]\\
& = :\left[e_i,e^j\right]\left(Se_j\right):+:e^j\left(S\left[e_i,e_j\right]_{\overline{\ell}}\right):-\chi:e^j\left[e_i,e_j\right]_{\overline{\ell}}:+2\chi Se_i\\
&+\lambda\left(\left[\left[e_i,e^j\right],e_j\right]+\mathcal{D}\left\langle e_i,\left[e_j,e^j\right]\right\rangle\right),
\end{align*}
where we have used the involutivity of $\overline{\ell}$. Therefore,
\begin{align*}
\Upsilon_i^1
& = :\left[e_i,e^j\right]_-\left(Se_j\right):+:e_j\left(S\left[e_i,e^j\right]_-\right):+\left.:e^j:e^k\left(\left\langle\mathcal{D}\left\langle\left[e_i,e_j\right],e^m\right\rangle,e_k\right\rangle e_m\right)::\right.\\
& + 2:\left(T\left\langle e^k,\left[e_i,e_j\right]\right\rangle\right)\left[e_k,e^j\right]:+:\left(T\left\langle\left[e^k,e_i\right],e^j \right\rangle\right)\left[e_k,e_j\right]_{\overline{\ell}}:\nonumber\\
& + T\Big(:e_j\left(\left\langle\mathcal{D}\left\langle\left[e^k,e_i\right],e^j\right\rangle,e_k\right\rangle-\left\langle\mathcal{D}\left\langle e^j,\left[e_i,e_k\right]\right\rangle,e^k\right\rangle\right):\nonumber\\
& +:e^j\left(\left\langle\mathcal{D}\left\langle\left[e^k,e_i\right],e_j\right\rangle,e_k\right\rangle\right):\Big)+\frac{1}{2}:e_j:\left(\mathcal{D}\left\langle\left[e^k,e_i\right],e^j\right\rangle\right)_{\overline{\ell}}e_k::\nonumber\\
&+\left.:e_j\left(:e_k:\left(\frac12\left\langle\mathcal{D}\left\langle\left[e^j,e_i\right],e^k\right\rangle,e_m\right\rangle +\left\langle\mathcal{D}\left\langle\left[e_i,e^j\right],e_m\right\rangle,e^k\right\rangle\right)e^m::\right):\right.\nonumber\\
&+\frac12:e_j:\left(\mathcal{D}\left\langle\left[e^k,e_i\right],e^j\right\rangle\right)_-e_k::+:e_k:\left(\left\langle\mathcal{D}\left[e_i,e^k\right],e_j\right\rangle\right)_-e^j::\\
& +\chi\left(2Se_i-:e_j\left[e_i,e^j\right]:-:e^j\left[e_i,e_j\right]:\right)+\lambda\left(2e_i+\left[\left[e_i,e_j\right],e^j\right]+\left[\left[e_i,e^j\right],e_j\right]\right).
\end{align*}
Here, we have used the identity \eqref{eq:parte1}. Secondly, we compute
\begin{equation*}
\begin{split}
\Upsilon_i^2 :&= \left[{e_i}_\Lambda\left(:e_j:e_k\left[e^j,e^k\right]_\ell::+:e^j:e^k\left[e_j,e_k\right]_{\overline{\ell}}::+2:e_j:e^k\left[e^j,e_k\right]_-::\right)\right]\\
& = \Upsilon_i^{2,1}+\Upsilon_i^{2,2}+\Upsilon_i^{2,3}.
\end{split}
\end{equation*}
Applying the non-commutative Wick formula twice on each summand, we obtain
\begin{align*}
\Upsilon_i^{2,1} :&= \left[{e_i}_\Lambda:e_j:e_k\left[e^j,e^k\right]_\ell::\right]\\
& = :\left[e_i,e_j\right]_{\overline{\ell}}:e_k\left[e^j,e^k\right]_\ell::+:e_j:\left[e_i,e_k\right]_{\overline{\ell}}\left[e^j,e^k\right]_\ell::+:e_j:e_k\left[e_i,\left[e^j,e^k\right]_\ell\right]::\\
& + 2\chi:e_j\left[e_i,e^j\right]_{\overline{\ell}}:+2\lambda\Big(2\left[e^j,\left[e_i,e_j\right]_{\overline{\ell}}\right]_{\overline{\ell}}+2\left\langle\mathcal{D}\left\langle e^j,\left[e_k,e_i\right]\right\rangle,e^k\right\rangle e_j\Big),\\
\Upsilon_i^{2,2} :&= \left[{e_i}_\Lambda:e^j:e^k\left[e_j,e_k\right]_{\overline{\ell}}::\right]\\
& = :\left[e_i,e^j\right]:e^k\left[e_j,e_k\right]_{\overline{\ell}}::+:e^j:\left[e_i,e^k\right]\left[e_j,e_k\right]_{\overline{\ell}}::+:e^j:e^k\left[e_i,\left[e_j,e_k\right]_{\overline{\ell}}\right]_{\overline{\ell}}::\\
&+4\chi:e^j\left[e_i,e_j\right]_{\overline{\ell}}:+2\lambda\Big(2\left[e_j,\left[e_i,e^j\right]_\ell\right]_\ell+\left[e_k,\left[e_i,e^k\right]_{\overline{\ell}}\right]_{\overline{\ell}}\\
&+2\left\langle\mathcal{D}\left\langle\left[e^k,e_i\right],e_j\right\rangle,e_k\right\rangle e^j + \left\langle\mathcal{D}\left\langle\left[e^k,e_i\right],e^j\right\rangle,e_k\right\rangle e_j\Big),\\
\Upsilon_i^{2,3} :&= \left[{e_i}_\Lambda:e_j:e^k\left[e^j,e_k\right]_-::\right]\\
& = :\left[e_i,e_j\right]_{\overline{\ell}}:e^k\left[e^j,e_k\right]_-::+:e_j:\left[e_i,e^k\right]\left[e^j,e_k\right]_-::+:e_j:e^k\left[e_i,\left[e^j,e_k\right]_-\right]::\\
& + 2\chi:e_j\left[e_i,e^j\right]_-:+2\lambda\left(\left[e_j,\left[e_i,e^j\right]_-\right]_{\overline{\ell}}+\left[e^j,\left[e_i,e_j\right]_{\overline{\ell}}\right]_-\right),
\end{align*}
where we have used the involutivity of $\ell$ and $\overline{\ell}$, and the Courant algebroid axioms. Then
\begin{align*}
\Upsilon_i^2 & = :e_j:\left[e_k,e^j\right]_{\overline{\ell}}\left[e_i,e^k\right]_{\overline{\ell}}::+4:\left(T\left\langle\left[e^j,e_i\right],e_k\right\rangle\right)\left[e^k,e_j\right]:\nonumber\\
&+2:\left(T\left\langle\left[e^j,e_i\right],e^k\right\rangle\right)\left[e_k,e_j\right]_{\overline{\ell}}:-2:e^j:e^k\left(\left\langle\mathcal{D}\left\langle e^m,\left[e_i,e_j\right] \right\rangle,e_k\right\rangle e_m\right)::\nonumber\\
&+\left.:e_j:e_k\left(\left(\left\langle\mathcal{D}\left\langle e^k,\left[e_i,e^j\right]\right\rangle,e_m\right\rangle - 2\left\langle\mathcal{D}\left\langle e_m,\left[e_i,e^j\right]\right\rangle,e^k\right\rangle\right)e^m\right)::\right.\\
&+\left.:e_j:\left(\mathcal{D}\left\langle\left[e_i,e^k\right],e^j\right\rangle\right)_-e_k::-2:e_k:\left(\mathcal{D}\left\langle\left[e_i,e^k\right],e_j\right\rangle\right)_-e^j::\right.\\
&+2\Big(:e_j:e_k\left[e^j,\left[e_i,e^k\right]_-\right]::+:e_j:e^k\left[e_k,\left[e^j,e_i\right]_-\right]_-::\nonumber\\
&+:e_j:\left[e_i,e^k\right]_-\left[e^j,e_k\right]_-::+:e_j:e^k\left[e_i,\left[e^j,e_k\right]_-\right]_{\overline{\ell}}::\Big)\\
& + 2\chi\,\Big(:e_j\left[e_i,e^j\right]_{\overline{\ell}}:+2:e^j\left[e_i,e_j\right]_{\overline{\ell}}:+:e_j\left[e_i,e^j\right]_-:\Big)\\
&+2\lambda\,\Big(2\left[e^j,\left[e_i,e_j\right]_{\overline{\ell}}\right]_{\overline{\ell}}+2\left\langle\mathcal{D}\left\langle e^j,\left[e_k,e_i\right]\right\rangle,e^k\right\rangle e_j+2\left[e_j,\left[e_i,e^j\right]_\ell\right]_\ell\\
&+\left.\left[e_k,\left[e_i,e^k\right]_{\overline{\ell}}\right]_{\overline{\ell}}+2\left\langle\mathcal{D}\left\langle\left[e^k,e_i\right],e_j\right\rangle,e_k\right\rangle e^j + \left\langle\mathcal{D}\left\langle\left[e^k,e_i\right],e^j\right\rangle,e_k\right\rangle e_j\right.\\
&+\left[e_j,\left[e_i,e^j\right]_-\right]_{\overline{\ell}}+\left[e^j,\left[e_i,e_j\right]_{\overline{\ell}}\right]_-\Big).
\end{align*}
Here, we have used the identities~\eqref{eq:parte2}, \eqref{eq:parte3}. Now, using \eqref{eq:afbafb}, \eqref{eq:affa}, we obtain 
\[
2\partial_\chi\left(\frac{1}{2}\Upsilon_i^1+\frac{1}{4}\Upsilon_i^2\right) = 2Se_i+:e_j\left[e_i,e^j\right]_-:.
\]
Analogously, using the Courant algebroid axioms, we have
\begin{align*}
2\partial_\lambda\left(\frac{1}{2}\Upsilon_i^1+\frac{1}{4}\Upsilon_i^2\right) &= 2e_i+\left[e_j,\left[e_i,e^j\right]_-\right]_{\overline{\ell}}-\left[e_j,\left[e_i,e^j\right]_-\right]_-\\
&+\Big(\left\langle\mathcal{D}\left\langle e^j,\left[e_k,e_i\right]\right\rangle,e^k\right\rangle + \left\langle\mathcal{D}\left\langle\left[e^k,e_i\right],e^j\right\rangle,e_k\right\rangle\Big)e_j\\
&+\left\langle\mathcal{D}\left\langle\left[e^k,e_i\right],e_j\right\rangle,e_k\right\rangle e^j.
\end{align*}
In conclusion, using \eqref{eq:15}, \eqref{eq:parte0} and the Courant algebroid axioms, we obtain
\begin{align*}
\left[{e_i}_\Lambda{H_0}\right] &= \frac{1}{2}\Upsilon_i^1+\frac{1}{4}\Upsilon_i^2\\
& = \left(\chi S+\lambda\right)e_i+\frac{\chi}{2}:e_j\left[e_i,e^j\right]_-:+\frac{\lambda}{2}\left(\left[e_j,\left[e_i,e^j\right]_-\right]_{\overline{\ell}}-\left[e_j\left[e_i,e^j\right]_-\right]_-\right)\\
&+\frac{\left(\lambda+T\right)}{2}\bigg(\Big(\left\langle\mathcal{D}\left\langle\left[e^k,e_i\right],e^j\right\rangle,e_k\right\rangle -\left\langle\mathcal{D}\left\langle e^j,\left[e_i,e_k\right]\right\rangle,e^k\right\rangle\Big)e_j\\
& +\left\langle\mathcal{D}\left\langle\left[e^k,e_i\right],e_j\right\rangle,e_k\right\rangle e^j\bigg)+\frac{1}{2}\bigg(:\left[e_i,e^j\right]_-\left(Se_j\right):+:e_j\left(S\left[e_i,e^j\right]_-\right):\bigg)\\
&+\frac{1}{2}\bigg(:e_j:e_k\left[e^j,\left[e_i,e^k\right]_-\right]_-::+:e_j:e_k\left[\left[e_i,e^j\right]_-,e^k\right]_\ell::\\
&+:e_j:e^k\left[e_k,\left[e^j,e_i\right]_-\right]_-::+:e_j:\left[e_i,e^k\right]_-\left[e^j,e_k\right]::
+:e_j:e^k\left[e_i,\left[e^j,e_k\right]_-\right]_{\overline{\ell}}::\bigg).
\end{align*}
Before continuing, note that \eqref{eq:cuasicom} and \eqref{eq:cuasicomaso} imply
\begin{align*}
:e_j:e_k\left[e^j,\left[e_i,e^k\right]_-\right]:: &= :e_j:e_k\left[\left[e_i,e^j\right]_-,e^k\right]::,\\
:e_j:e^k\left[e_k,\left[e^j,e_i\right]_-\right]_-:: &= :e^k:e_j\left[\left[e^k,e_i\right]_-,e_j\right]_-::,\\
:e_j:\left[e_i,e^k\right]_-\left[e^j,e_k\right]_-:: &= :\left[e_i,e^j\right]_-:e_k\left[e_j,e^k\right]_-::.
\end{align*}
Formula~\eqref{eq:NSbasis10} follows from the above identities, using the Courant algebroid axioms, that $S$ is an odd derivation for the normally ordered product, and \eqref{eq:cuasicom2}.

Formula~\eqref{eq:NSbasis1} is now a consequence of~\eqref{eq:NSbasis10} and the following calculation, where we use sesquilinearity and \eqref{eq:14}:
\begin{align*}
\left[Tw_\Lambda e_i\right] &= -\lambda\left[w_\Lambda e_i\right]\\
&=-\lambda\left(\left\langle\left[e^j,e_j\right],e^k\right\rangle\left[e_k,e_i\right]_{\overline{\ell}}-\left\langle\left[e^j,e_j\right],e_k\right\rangle\left[e^k,e_i\right]\right)\nonumber\\
& +\lambda\Big(\left\langle\mathcal{D}\left\langle e^j,\left[e^k,e_k\right]\right\rangle,e_i\right\rangle e_j - \left\langle\mathcal{D}\left\langle e_j,\left[e^k,e_k\right]\right\rangle,e_i\right\rangle e^j+\mathcal{D}\left\langle e_i,\left[e^j,e_j\right]\right\rangle\Big)\nonumber\\
&+\lambda\chi\left\langle\left[e^j,e_j\right],e_i\right\rangle.
\end{align*}

Finally, the formulae for $\left[{H_0}_\Lambda e^i \right]$ and $\left[{H'}_\Lambda e^i \right]$ are obtainted from~\eqref{eq:NSbasis10} and~\eqref{eq:NSbasis1} by interchanging the roles of $\ell$ and $\overline{\ell}$, as explained in Remark~\ref{obs:app.MSobs}.
\end{proof}

%\begin{remark}

%\textbf{MGF: esto es necesario???? Donde se usa?}

%We collect here some formulae \textbf{which follow from the previous proof?}, which will be useful later:
%Before ending, notice that
%\begin{equation*}
%\begin{split}
%\left[e_n,e_i\right] &= \left\langle e^j,\left[e_n,e_i\right]\right\rangle e_j,\\
%\left[e^n,e_i\right] & = \left\langle e^j,\left[e^n,e_i\right]\right\rangle e_j+\left\langle e_j,\left[e^n,e_i\right]\right\rangle e^j+\left[e^n,e_i\right]_-,\\
%\left[e^n,e^i\right] &= \left\langle e_j,\left[e^n,e^i\right]\right\rangle e^j,\\
%\left[e_n,e^i\right] & = \left\langle e^j,\left[e_n,e^i\right]\right\rangle e_j+\left\langle e_j,\left[e_n,e^i\right]\right\rangle e^j+\left[e_n,e^i\right]_-.
%\end{split}
%\end{equation*}
%\end{remark}

Before proving Proposition \ref{teo:mainresult2}, we need the following technical result.

\begin{proposition}
\label{lem:lempasofinal}
Assume that $\ell \oplus \overline{\ell}$ satisfies the F-term condition \eqref{eq:Ftermabs}. Then
\begin{equation*}
\label{eq:desiredfor}
\begin{split}
\left[{H'}_\Lambda{J_0}\right] & = \left(2\lambda+2T+\chi S\right)\left(J_0-\frac{i}{2}S\left[e^j,e_j\right]_-\right)\\
& + \frac{i}{4}TS\mathcal{D}R+\frac{i}{4}\lambda\bigg(:F^{ij}:e_je_i::-:F_{ij}:e^je^i::\bigg)\\
& -\frac{i}{2}\left(T+\frac{3}{2}\lambda\right)\left(:e^i\left[\left[e^j,e_j\right]_-,e_i\right]_-:+:e_i\left[\left[e^j,e_j\right]_-,e^i\right]_-:\right).
\end{split}
\end{equation*}
\end{proposition}

\begin{proof}
Applying the non-commutative Wick formula, we have
\begin{equation*}
\begin{split}
\left[{H'}_\Lambda:e^ie_i:\right] & = :\left[{H'}_\Lambda{e^i}\right]e_i:+:e^i\left[{H'}_\Lambda{e_i}\right]:+\int_0^\Lambda d\Gamma\left[\left[{H'}_\Lambda{e^i}\right]_\Gamma{e_i}\right]\\
& = P + I,
\end{split}
\end{equation*}
where $I$ is the third summand of the right-hand side in the first line, that splits into a sum
\begin{equation*}
I:=\int_0^\lambda d\lambda\left(I_1+I_2+I_3+I_4\right),
\end{equation*}
where
\begin{alignat*}{2}
I_1 &:= \partial_\eta\left[{A^i}_\Gamma{e_i}\right],&\qquad
I_2 &:= \partial_\eta\left[{B^i}_\Gamma{e_i}\right],\\
I_3 &:= \partial_\eta\left[{C^i}_\Gamma{e_i}\right],&
I_4 &:= \partial_\eta\left[{D^i}_\Gamma{e_i}\right],
\end{alignat*}
and
\begin{align*}
A^i :&= \frac{\chi}{2}:e^j\left[e^i,e_j\right]_-:-:e^j\left(S\left[e^i,e_j\right]_-\right):\\
& -\frac{1}{2}\,\bigg(:e^j:e^k\left[\left[e^i,e_j\right]_-,e_k\right]::+:\left[e^i,e_j\right]_-:e^k\left[e^j,e_k\right]_-::\nonumber\\
&+:e^j:e_k\left[e^i,\left[e_j,e^k\right]_-\right]_\ell::+:e_j:e^k\left[\left[e_k,e^i\right]_-,e^j\right]_-::\bigg)\nonumber\\
B^i :&= \frac{\lambda}{2}\left(\left[e^j,\left[e^i,e_j\right]_-\right]_\ell+\left[\left[e^i,e_j\right]_-,e^j\right]_-\right)+T\left[e^j,\left[e^i,e_j\right]_-\right]_\ell,\\
C^i :&= \left(\lambda+2T+\chi S\right)e^i,\\
D^i :&= \lambda\chi\left\langle\left[e_j,e^j\right],e^i\right\rangle\\
&+\frac{\lambda}{2}\bigg(\Big(\left\langle e^m,\left[e^k,e_k\right]\right\rangle\left\langle\left[e^i,e_m\right],e^j\right\rangle +\left\langle\mathcal{D}\left\langle e^j,\left[e^k,e_k\right]\right\rangle,e^i\right\rangle\\
&+\left\langle\mathcal{D}\left\langle e_k,\left[e^i,e^j\right]\right\rangle,e^k\right\rangle\Big)e_j+\Big(\left\langle e^m,\left[e^k,e_k\right]\right\rangle\left\langle\left[e^i,e_m\right],e_j\right\rangle\\
&+\left.\left\langle\mathcal{D}\left\langle e_k,\left[e^i,e_j\right]\right\rangle,e^k\right\rangle - \left\langle\mathcal{D}\left\langle e_j,\left[e^i,e^k\right]\right\rangle,e_k\right\rangle\right.\\
&-\left\langle e_m,\left[e^k,e_k\right]\right\rangle\left\langle\left[e^i,e^m\right],e_j\right\rangle - \left\langle\mathcal{D}\left\langle e_j,\left[e^k,e_k\right]\right\rangle,e^i\right\rangle\Big)e^j\\
&+\left\langle e^k,\left[e^j,e_j\right]\right\rangle\left[e^i,e_k\right]_--\mathcal{D}\left\langle e^i,\left[e^j,e_j\right]\right\rangle\bigg).
\end{align*}
To calculate $I_1$, we proceed as follows. By skew-symmetry of the $\Lambda$-bracket, we can write
\begin{align*}
I_1&=-\partial_\eta\left[{e_i}_{-\Gamma-\nabla}A^i\right] = \partial_\eta\left[{e_i}_\Gamma A^i\right]\\
& = \frac{1}{2}\left(I_1^1-2I_1^2\right)-\frac{1}{2}\left(I_1^3+I_1^4+I_1^5+I_1^6\right),
\end{align*}
where $I_1^1, I_1^2, I_1^3, I_1^4, I_1^5$ and $I_1^6$ are defined and calculated as follows, using the non-commutative Wick formula, the Courant algebroid axioms, \eqref{eq:cuasicom} and \eqref{eq:clave}:
\begin{align*}
I_1^1 :&= \partial_\eta\left[{e_i}_\Gamma\left(\chi:e^j\left[e^i,e_j\right]_-:\right)\right]\\
& = 2\chi\left[e_j,e^j\right]_-,\\
I_1^2 :&= \partial_\eta\left[{e_i}_\Gamma:e^j\left(S\left[e^i,e_j\right]_-\right):\right]\\
& = 2\left[e^j,e_j\right]_--:e^j\left[e_i,\left[e^i,e_j\right]_-\right]:,\\
I_1^3 :&= \partial_\eta\left[{e_i}_\Gamma:e^j:e^k\left[\left[e^i,e_j\right]_-,e_k\right]::\right]\\
& = 2\left(:e^k\left[\left[e^j,e_j\right]_-,e_k\right]:-:e^j\left[\left[e^k,e_j\right]_-,e_k\right]:\right),\\
I_1^4 :&= \partial_\eta\left[{e_i}_\Gamma:\left[e^i,e_j\right]_-:e^k\left[e^j,e_k\right]_-::\right]\\
& = -2T\left\langle\left[e^k,e_j\right]_-,\left[e^j,e_k\right]_-\right\rangle,\\
I_1^5 :&= \partial_\eta\left[{e_i}_\Gamma:e^j:e_k\left[e^i,\left[e_j,e^k\right]_-\right]_\ell::\right]\\
& = 2\left(:e_k\left[e^j,\left[e_j,e^k\right]_-\right]_{\ell}:+:e^k\left[e_k,\left[e^j,e_j\right]_-\right]_{\overline{\ell}}:\right),\\
I_1^6 :&= \partial_\eta\left[{e_i}_\Gamma:e_j:e^k\left[\left[e_k,e^i\right]_-,e^j\right]_-::\right]\\
& = 2:e_j\left[\left[e^k,e_k\right]_-,e^j\right]_-:.
\end{align*}
Combining the previous expressions and using the F-term condition \eqref{eq:Ftermabs}, we have
\begin{align*}
I_1 & = T\left\langle\left[e^k,e_j\right]_-,\left[e^j,e_k\right]_-\right\rangle + \left(\chi+2S\right)\left[e_i,e^j\right]_-\\
& +:e_k\left[\left[e_j,e^j\right]_-,e^k\right]_-:+:e^k\left[\left[e_j,e^j\right]_-,e_k\right]_-:-:e_k\left[e^j,\left[e_j,e^k\right]_-\right]_\ell:.
\end{align*}
Furthermore, by sesquilinearity, the non-commutative Wick formula and the Courant algebroid axioms,
\begin{align*}
I_2 & = \left(\lambda-2\gamma\right)\left\langle\left[e^j,\left[e^i,e_j\right]_-\right],e_i\right\rangle,\\
I_3 & = 2\left(\lambda-2\gamma\right)\left\langle e^j,e_j\right\rangle + \chi\left[e^j,e_j\right],\\
I_4 & = 0.
\end{align*}
In conclusion, we obtain
\begin{align*}
I & = \lambda T \left\langle\left[e^k,e_j\right]_-,\left[e^j,e_k\right]\right\rangle + \lambda\chi\left[e^j,e_j\right]_+-\lambda:e_k\left[e^j,\left[e_j,e^k\right]_-\right]_\ell:\\
& +2\lambda S\left[e_j,e^j\right]_-+\lambda:e_k\left[\left[e_j,e^j\right]_-,e^k\right]_-:+\lambda:e^k\left[\left[e_j,e^j\right]_-,e_k\right]_-:.
\end{align*}
Next, we split 
\[
P:= P_1+P_2+P_3
\]
into three summands
\begin{align*}
P_1 :&= :a^je_j:+:e^ja_j:,\\
P_2 :&= :b^je_j:+:e^jb_j:,\\
P_3 :&= :c^je_j:+:e^jc_j:,
\end{align*}
where
\begin{alignat*}{2}
c^j &:=\lambda\partial_\lambda\left[{H'}_\Lambda e^j\right]-\lambda e^j,&\qquad c_j &:=\lambda\partial_\lambda\left[{H'}_\Lambda e_j\right]-\lambda e_j,\\
b^j &:=\left(\lambda+2T+\chi S\right)e^j,&b_j &:=\left(\lambda+2T+\chi S\right)e_j,\\
a^j &:=\left[{H'}_\Lambda{e^j}\right]-b^j-c^j,&a_j &:=\left[{H'}_\Lambda{e_j}\right]-b_j-c_j.
\end{alignat*}
To calculate $P_1$, we proceed as follows. First, by a direct calculation, we have
\begin{equation*}
\begin{split}
P_1 & = -\frac{1}{2}Q+\frac{\chi}{2}W+\frac{\lambda}{2}X+Y-Z,
\end{split}
\end{equation*}
where
\begin{align*}
Q :&= Q_1'+Q_2'+Q_3'+Q_4'+Q_5',\\
W :&= ::e^j\left[e^i,e_j\right]_-:e_i:-:e^i:e_j\left[e_i,e^j\right]_-::,\\
X :&= :e^i\left[e_j,\left[e_i,e^j\right]_-\right]_{\overline{\ell}}:+:e^i\left[\left[e_i,e^j\right]_-,e_j\right]_-:\\
&+:\left[e^j,\left[e^i,e_j\right]_-\right]_\ell e_i:+:\left[\left[e^i,e_j\right]_-,e^j\right]_-e_i:,\\
Y :&= :e^i\left(T\left[e_j,\left[e_i,e^j\right]_-\right]_{\overline{\ell}}\right):+:\left(T\left[e^j,\left[e^i,e_j\right]_-\right]_\ell\right)e_i:,\\
Z :&= :e^i:e_j\left(S\left[e_i,e^j\right]_-\right)::+::e^j\left(S\left[e^i,e_j\right]_-\right):e_i:,
\end{align*}
and
\begin{align*}
Q_1' :&= ::e^j:e^k\left[\left[e^i,e_j\right]_-,e_k\right]_{\overline{\ell}}::e_i:+::e^j:e_k\left[e^i,\left[e_j,e^k\right]_-\right]_\ell::e_i:,\\
Q_2' :&= :e^i:e_j:e_k\left[\left[e_i,e^j\right]_-,e^k\right]_\ell:::+:e^i:e_j:e^k\left[e_i,\left[e^j,e_k\right]_-\right]_{\overline{\ell}}:::,\\
Q_3' :&= :e^i:\left[e_i,e^j\right]_-:e_k\left[e_j,e^k\right]_-:::+::\left[e^i,e_j\right]_-:e^k\left[e^j,e_k\right]_-::e_i:,\\
Q_4' :&= :e^i:e^j:e_k\left[\left[e^k,e_i\right]_-,e_j\right]_-:::+::e^j:e^k\left[\left[e^i,e_j\right]_-,e_k\right]_-::e_i:,\\
Q_5' :&= :e^i:e_j:e_k\left[\left[e_i,e^j\right]_-,e^k\right]_-:::+::e_j:e^k\left[\left[e_k,e^i\right]_-,e^j\right]_-::e_i:.
\end{align*}
Secondly, using the Courant algebroid axioms, \eqref{eq:affa}, \eqref{eq:cuasicom}, \eqref{eq:afbafb} and \eqref{eq:cuasiaso3}, we define $Q_1, Q_2, Q_3, Q_4$ and $Q_5$ in terms of  $Q_1', Q_2', Q_3', Q_4'$ and $Q_5'$ by the following formulae:
\begin{align*}
Q_1' :&= ::e_j:e^k\left[e^i,\left[e^j,e_k\right]_-\right]_\ell::e_i:+::e^k:\left[e^i,\left[e^j,e_k\right]_-\right]_\ell e_i::e_j:\\
&+2::e^j\left(T\left\langle\left[e_j,e^i\right]_-,\left[e_k,e^k\right]_-\right\rangle\right):e_i:\\
&=:Q_1+2::e^j\left(T\left\langle\left[e_j,e^i\right]_-,\left[e_k,e^k\right]_-\right\rangle\right):e_i:,\\
Q_2' :&= :e^i:e_j:e^k\left[e_i,\left[e^j,e_k\right]_-\right]_{\overline{\ell}}:::+:e^k:\left[e_i,\left[e^j,e_k\right]_-\right]_{\overline{\ell}}:e_je^i:::\\
&+2:e^j:\left(T\left\langle\left[e_k,e^k\right]_-,\left[e_j,e^i\right]_-\right\rangle\right)e_i::\\
&=:Q_2+2:e^j:\left(T\left\langle\left[e_k,e^k\right]_-,\left[e_j,e^i\right]_-\right\rangle\right)e_i::,\\
Q_3' :&= ::\left[e^i,e_j\right]_-:e^k\left[e^j,e_k\right]_-::e_i:+:e^k:\left[e^j,e_k\right]_-:e_i\left[e^i,e_j\right]_-:::\\
& =: Q_3,\\
Q_4' :&= ::e^j:e^k\left[\left[e^i,e_j\right]_-,e_k\right]_-::e_i:+:e^j:e^k:e_i\left[\left[e^i,e_j\right]_-,e_k\right]_-:::\\
&=:Q_4,\\
Q_5' :&= ::e^j:e_k\left[\left[e^i,e_j\right]_-,e^k\right]_-::e_i:+:e^j:e_i:e_k\left[\left[e_j,e^i\right]_-,e^k\right]_-:::\\
& =: Q_5.
\end{align*}
Applying the Courant algebroid axioms and \eqref{eq:17}, we obtain 
\begin{equation*}
\begin{split}
Q &= Q_1+Q_2+Q_3+Q_4+Q_5\\
&+4\left(::e^j\left(T\left\langle\left[e_j,e^i\right]_-,\left[e_k,e^k\right]_-\right\rangle\right):e_i:-T^2\left\langle\left[e_j,e^j\right]_-,\left[e_k,e^k\right]_-\right\rangle\right).
\end{split}
\end{equation*}
We will compute $Q_1,Q_2,Q_3,Q_4,Q_5$ separately. Regarding $Q_1$, we use the Courant algebroid axioms, \eqref{eq:20}, \eqref{eq:cuasiaso2}, \eqref{eq:18}, \eqref{eq:19} and the fact that $T$ is an even derivation for the normally ordered product, obtaining
\begin{equation*}
\begin{split}
Q_1 &=2\left(::\left(T\left\langle e_k,\left[e^j,\left[e^i,e_i\right]_-\right]\right\rangle\right)e^k:e_j:\right.\\
&+\left.T^2\left(\left\langle\left[e^k,e_j\right]_-,\left[e^j,e_k\right]_-\right\rangle - \left\langle\left[e^j,e_j\right]_-,\left[e^k,e_k\right]_-\right\rangle\right)\right).
\end{split}
\end{equation*}
By the Courant algebroid axioms, \eqref{eq:cuasicom}, \eqref{eq:cuasicom2}, \eqref{eq:20}, \eqref{eq:afbafb}, \eqref{eq:cuasiaso3}, \eqref{eq:18}, \eqref{eq:19}, and the fact that $T$ is an even derivation for the normally ordered product, we obtain 
\begin{equation*}
\begin{split}
Q_2&= 2::\left(T\left\langle e_k,\left[e^j,\left[e^i,e_i\right]_-\right]\right\rangle\right)e^k:e_j:.
\end{split}
\end{equation*}
Applying the Courant algebroid axioms, \eqref{eq:cuasicom}, \eqref{eq:cuasicom3}, \eqref{eq:cuasiaso2} \eqref{eq:18} and the fact that $T$ is an even derivation for the normally ordered product, we obtain
\begin{equation*}
\begin{split}
Q_3 &= 2\left(T^2\left\langle\left[e^j,e_k\right]_-,\left[e^k,e_j\right]_-\right\rangle + ::\left(T\left\langle\left[e^j,\left[e^i,e_j\right]_-\right],e_k \right\rangle\right)e^k:e_i:\right).
\end{split}
\end{equation*}
By the Courant algebroid axioms, \eqref{eq:cuasicom}, \eqref{eq:cuasiaso2} and \eqref{eq:18}, we obtain
\begin{equation*}
\begin{split}
Q_4 &= 2T\left(:e^k\left[\left[e^i,e_i\right]_-,e_k\right]_-:-:e^j\left[\left[e^k,e_j\right]_-,e_k\right]_-:\right).
\end{split}
\end{equation*}
Finally, by the Courant algebroid axioms, \eqref{eq:cuasiaso2}, \eqref{eq:cuasiasocom} and \eqref{eq:18}, we obtain 
\begin{equation*}
\begin{split}
Q_5 & = 2T\left(:e_k\left[\left[e^i,e_i\right]_-,e^k\right]_-:\right).
\end{split}
\end{equation*}
Putting together these computations, and using the Courant algebroid axioms and \eqref{eq:affa}, we obtain 
\begin{equation*}
\begin{split}
Q &= 4T^2\left\langle\left[e^j,e_k\right]_-,\left[e^k,e_j\right]_-\right\rangle-6T^2\left\langle\left[e^j,e_j\right]_-,\left[e^k,e_k\right]_-\right\rangle\\
&+2::\left(T\left\langle\left[e^j,\left[e^i,e_j\right]_-\right],e_k\right\rangle\right)e^k:e_i:\\
&+2T\left(:e^k\left[\left[e^i,e_i\right]_-,e_k\right]_-:+:e_k\left[\left[e^i,e_i\right]_-,e^k\right]_-:-:e^j\left[\left[e^k,e_j\right]_-,e_k\right]_-:\right).
\end{split}
\end{equation*}
To compute $P_1$, it remains to calculate $W,X,Y$ and $Z$. By the Courant algebroid axioms, \eqref{eq:cuasicom} and \eqref{eq:cuasiaso2},
\begin{equation*}
\begin{split}
W &= -2T\left[e^j,e_j\right]_-.
\end{split}
\end{equation*}
By the Courant algebroid axioms, \eqref{eq:cuasicom} and \eqref{eq:afbafb},
\begin{equation*}
\begin{split}
X &= 2:e_j\left[e^k,\left[e_k,e^j\right]_-\right]_\ell:+:e^i\left[\left[e_i,e^j\right]_-,e_j\right]_-:+:\left[\left[e^i,e_j\right]_-,e^j\right]_-e_i:\\
&+2T\left\langle e_j,\left[e^k,\left[e^j,e_k\right]_-\right]\right\rangle.
\end{split}
\end{equation*}
Applying Courant algebroid axioms, \eqref{eq:21} and using the fact that $T$ is an even derivation for the normally ordered product,
\begin{equation*}
\begin{split}
Y &= T\left(:e^i\left[e_j,\left[e_i,e^j\right]_-\right]_{\overline{\ell}}:\right)+::\left(T\left\langle e_k,\left[e^j,\left[e^i,e_j\right]_-\right]\right\rangle\right)e^k:e_i:\\
&+T^2\left\langle\left[e^j,e_k\right]_-,\left[e^k,e_j\right]_-\right\rangle.
\end{split}
\end{equation*}
Applying the Courant algebroid axioms, \eqref{eq:cuasicom2} and using that $T$ an even derivation for the normally ordered product,
\begin{equation*}
\begin{split}
Z &= T\left(:e^i\left[\left[e^j,e_i\right]_-,e_j\right]:\right) +2TS\left[e^j,e_j\right]_-.
\end{split}
\end{equation*}
Putting together the above calculations for $Q,W,X,Y$ and $Z$, and using the Courant algebroid axioms, we obtain
\begin{equation*}
\begin{split}
P_1 & = T\left(2S+\chi\right)\left[e_j,e^j\right]_-+T^2R-\lambda T\left\langle\left[e^k,e_j\right]_-,\left[e^j,e_k\right]_-\right\rangle\\
&-T\left(:e^k\left[\left[e^j,e_j\right]_-,e_k\right]_-:+:e_k\left[\left[e^j,e_j\right]_-,e^k\right]_-:\right)\\
&+\frac{\lambda}{2}\left(2:e_k\left[e^j,\left[e_j,e^k\right]_-\right]_\ell:+:e^k\left[\left[e_k,e^j\right]_-,e_j\right]_-:+:\left[\left[e^k,e_j\right]_-,e^j\right]_-e_k:\right).
\end{split}
\end{equation*}
The calculation of $P_2$ is straighforward, because $S$ and $T=S^2$ are an odd and an even derivation for the normally ordered product, respectively, so 
\[
P_2=\left(2\lambda+2T+\chi S\right):e^je_j:.
\]
To compute $P_3$, we expand this term as a polynomial in the variables $\Lambda=(\lambda,\chi)$, that is, 
\[
P_3= \lambda\chi A+\frac{\lambda}{2}B
\]
so the coefficients are
\begin{equation*}
\begin{split}
A :&= -:e^i\left(\left\langle\left[e^j,e_j\right],e_i\right\rangle\right):-:\left(\left\langle\left[e^j,e_j\right],e^i\right\rangle\right)e_i:,\\
B :&= B_1+B_1'+B_2+B_2'+B_3+B_3'+B_4+B_4'-B_5-B_5',
\end{split}
\end{equation*}
where
\begin{align*}
B_1 :&= :\left(\left\langle\mathcal{D}\left\langle e_k,\left[e^i,e_j\right]\right\rangle,e^k \right\rangle e^j\right)e_i:-:e^i\left(\left\langle\mathcal{D}\left\langle e^j,\left[e_i,e_k\right]\right\rangle,e^k\right\rangle e_j\right):,\\
B_1' :&= :e^i\left(\left\langle\mathcal{D}\left\langle\left[e^k,e_i\right],e^j \right\rangle,e_k\right\rangle e_j\right):-:\left(\left\langle\mathcal{D}\left\langle e_j,\left[e^i,e^k\right]\right\rangle,e_k\right\rangle e^j\right)e_i:,\\
B_2 :&= :e^i\left(\mathcal{D}\left\langle e_i,\left[e^j,e_j\right]\right\rangle\right):\\
& -:\left(\left\langle\mathcal{D}\left\langle e_j,\left[e^k,e_k\right]\right\rangle,e^i \right\rangle e^j\right)e_i:- :e^i\left(\left\langle\mathcal{D}\left\langle e_j,\left[e^k,e_k\right]\right\rangle,e_i\right\rangle e^j\right):,\\
B_2' :&= :e^i\left(\left\langle\mathcal{D}\left\langle e^j,\left[e^k,e_k\right]\right\rangle,e_i\right\rangle e_j\right):+:\left(\left\langle\mathcal{D}\left\langle e^j,\left[e^k,e_k\right]\right\rangle,e^i\right\rangle e_j\right)e_i:\\
&-:\left(\mathcal{D}\left\langle e^i,\left[e^j,e_j\right]\right\rangle\right)e_i:,\\
B_3 :&= :e^i\left(\left\langle\mathcal{D}\left\langle\left[e^k,e_i\right],e_j\right\rangle,e_k\right\rangle e^j\right):,\\
B_3' :&= :\left(\left\langle\mathcal{D}\left\langle\left[e_k,e^i\right],e^j\right\rangle,e^k \right\rangle e_j\right)e_i:,\\
B_4 :&= :e^i\left(\left\langle e_m,\left[e^k,e_k\right]\right\rangle \left\langle e^j,\left[e^m,e_i\right]\right\rangle e_j\right):+:e^i\left(\left\langle e_m,\left[e^k,e_k\right]\right\rangle \left\langle e_j,\left[e^m,e_i\right]\right\rangle e^j\right):\\
&+:e^i\left(\left\langle e_j,\left[e^k,e_k\right]\right\rangle\left[e^j,e_i\right]_-\right):,\\
B_4' :&= :\left(\left\langle e^m,\left[e^k,e_k\right]\right\rangle \left\langle\left[e^i,e_m\right],e^j\right\rangle e_j\right)e_i:+:\left(\left\langle e^m,\left[e^k,e_k\right]\right\rangle \left\langle\left[e^i,e_m\right],e_j\right\rangle e^j\right)e_i:\\
&+:\left(\left\langle e^j,\left[e^k,e_k\right]\right\rangle \left[e^i,e_j\right]_-\right)e_i:,\\
B_5 :&= :e^i\left(\left\langle e^m,\left[e^k,e_k\right]\right\rangle \left\langle e^j,\left[e_m,e_i\right]\right\rangle e_j\right):,\\
B_5' :&= :\left(\left\langle e_m,\left[e^k,e_k\right]\right\rangle \left\langle\left[e^i,e^m\right],e_j\right\rangle e^j\right)e_i:.
\end{align*}
Clearly,
\begin{equation*}
\begin{split}
A&=-\left[e^j,e_j\right]_+.
\end{split}
\end{equation*}
By Courant algebroid axioms, \eqref{eq:affa} and \eqref{eq:afbafb},
\begin{equation*}
\begin{split}
B_1 &= 2T\left\langle\mathcal{D}\left\langle e_k,\left[e^j,e_j\right]\right\rangle,e^k \right\rangle.
\\
B_1' &= 2T\left\langle\mathcal{D}\left\langle e^k,\left[e_j,e^j\right]\right\rangle,e_k\right\rangle.
\end{split}
\end{equation*}
By the Courant algebroid axioms, \eqref{eq:affa}, \eqref{eq:cuasicom} and \eqref{eq:afbafb},
\begin{equation*}
\begin{split}
B_2&= 2:e^i\left(\mathcal{D}\left\langle e_i,\left[e^j,e_j\right]\right\rangle\right)_\ell:+:e^i\left(\mathcal{D}\left\langle e_i,\left[e^j,e_j\right]\right\rangle\right)_-:-2T\left\langle\mathcal{D}\left\langle e_k,\left[e^j,e_j\right]\right\rangle,e^k\right\rangle,
\\
B_2' &= -2:\left(\mathcal{D}\left\langle e^i,\left[e^j,e_j\right]\right\rangle\right)_{\overline{\ell}}e_i:-:\left(\mathcal{D}\left\langle e^i,\left[e^j,e_j\right]\right\rangle\right)_-e_i:+2T\left\langle\mathcal{D}\left\langle e^k,\left[e^j,e_j\right]\right\rangle,e_k\right\rangle.
\end{split}
\end{equation*}
By the Courant algebroid axioms (including~\eqref{eq:14}), as well as \eqref{eq:affa}, \eqref{eq:cuasicom} and \eqref{eq:afbafb},
\begin{equation*}
\begin{split}
B_4 &= :\left(\mathcal{D}\left\langle e_i,\left[e^j,e_j\right]\right\rangle\right)_\ell e^i:-:e^i\left[e_i,\left[e^k,e_k\right]_\ell\right]:,
\\
B_4' &= :\left(\mathcal{D}\left\langle e^i,\left[e^j,e_j\right]\right\rangle\right)_{\overline{\ell}}e_i:+:\left[e^i,\left[e^k,e_k\right]_{\overline{\ell}}\right]e_i:.
\end{split}
\end{equation*}
By the Courant algebroid axioms, \eqref{eq:affa} and \eqref{eq:afbafb},
\begin{equation*}
\begin{split}
B_5 &= 2T\left\langle\mathcal{D}\left\langle e^i,\left[e^k,e_k\right]\right\rangle,e_i\right\rangle+:e^i\left[e_i,\left[e^k,e_k\right]_{\overline{\ell}}\right]:-:\left(\mathcal{D}\left\langle e^i,\left[e^k,e_k\right]\right\rangle\right)_\ell e_i:,
\\
B_5' &= 2T\left\langle\mathcal{D}\left\langle e_i,\left[e^k,e_k\right]\right\rangle,e^i\right\rangle-:\left[e^i,\left[e^k,e_k\right]_\ell\right]e_i:+:e^i\left(\mathcal{D}\left\langle e_i,\left[e^k,e_k\right]\right\rangle\right)_{\overline{\ell}}:.
\end{split}
\end{equation*}
Consequently, applying \eqref{eq:extraprop}, we obtain
\begin{align*}
P_3 & = -\lambda\chi\left[e^j,e_j\right]_+-\lambda T\left\langle\left[w,e^j\right],e_j\right\rangle\\
&+\frac{\lambda}{2}\Big(:e^k\left[e_k,w\right]_\ell:+:\left[e^k,w\right]_{\overline{\ell}}e_k:+:e^k\left[e_k,w\right]_-:+:\left[e^k,w\right]_-e_k:\\
&\left.+2T\left(\left\langle\mathcal{D}\left\langle e^i,\left[e^k,e_k\right]\right\rangle,e_i\right\rangle + \left\langle\mathcal{D}\left\langle e_i,\left[e^k,e_k\right]\right\rangle,e^i\right\rangle\right)\right.\\
&+\left.:e^i\left(\left(\mathcal{D}\left\langle e_i,\left[e^k,e_k\right]\right\rangle\right)_\ell+\left(\mathcal{D}\left\langle e_i,\left[e^k,e_k\right]\right\rangle\right)_-\right):\right.\\
&-:\left(\left(\mathcal{D}\left\langle e^i,\left[e^k,e_k\right]\right\rangle\right)_{\overline{\ell}}+\left(\mathcal{D}\left\langle e^i,\left[e^k,e_k\right]\right\rangle\right)_-\right)e_i:\\
&+:e^i\left(\left\langle\mathcal{D}\left\langle\left[e^k,e_i\right],e_j\right\rangle,e_k\right\rangle e^j\right):+:\left(\left\langle\mathcal{D}\left\langle\left[e_k,e^i\right],e^j\right\rangle,e^k \right\rangle e_j\right)e_i:\Big).
\end{align*}
Adding together and using the identity $T=S^2$, we are now ready to compute
\begin{align*}
\left[{H'}_\Lambda:e^ie_i:\right] &= I+P_1+P_2+P_3\\
& = \left(2\lambda+2T+\chi S\right)\left(:e^je_j:-S\left[e^j,e_j\right]_-\right)\\
&-T\left(:e^k\left[\left[e^j,e_j\right]_-,e_k\right]_-:+:e_k\left[\left[e^j,e_j\right]_-,e^k\right]_-:\right)\\
&+T^2R+\frac{\lambda}{2}V,
\end{align*}
where
\begin{align*}
V :&= :\left[e^k,\left[e^j,e_j\right]_{\overline{\ell}}\right]_{\overline{\ell}}e_k:-:e^k\left[e_k,\left[e^j,e_j\right]_\ell\right]_\ell:\\
&+3\left(:e^k\left[e_k,\left[e^j,e_j\right]_-\right]_-:-:\left[e^k,\left[e^j,e_j\right]_-\right]_-e_k:\right)\\
&\left.+:\left[e^k,\left[e^j,e_j\right]\right]_-e_k:-:e^k\left[e_k,\left[e^j,e_j\right]\right]_-:\right.\\
&\left.+:e^k\left[\left[e_k,e^j\right]_-,e_j\right]_-:+:\left[\left[e^k,e_j\right]_-,e^j\right]_-e_k:\right.\\
&+2T\left(\left\langle\mathcal{D}\left\langle e^i,\left[e^k,e_k\right]\right\rangle,e_i\right\rangle +\left\langle\mathcal{D}\left\langle e_i,\left[e^k,e_k\right] \right\rangle,e^i\right\rangle -\left\langle\left[w,e^j\right],e_j\right\rangle\right)\\
&\left.+:e^i\left(\left(\mathcal{D}\left\langle e_i,\left[e^k,e_k\right]\right\rangle\right)_\ell+\left(\mathcal{D}\left\langle e_i,\left[e^k,e_k\right]\right\rangle\right)_-\right):\right.\\
&+\left.:\left(\left(\mathcal{D}\left\langle e^i,\left[e^k,e_k\right]\right\rangle\right)_{\overline{\ell}}+\left(\mathcal{D}\left\langle e^i,\left[e^k,e_k\right]\right\rangle\right)_-\right)e_i:\right.\\
&+\left.:e^i\left(\left\langle\mathcal{D}\left\langle \left[e^k,e_i\right],e_j\right\rangle,e_k\right\rangle e^j\right):+:\left(\left\langle\mathcal{D}\left\langle\left[e_k,e^i\right],e^j\right\rangle,e^k \right\rangle\right)e_i:\right..
\end{align*}
Using the involutivity of $\ell$ and $\overline{\ell}$, and the Courant algebroid axioms (in particular, the Jacobi identity in Definition~\ref{defn:CA}), we see that the Dorfman brackets in the third line of the definition of $V$ are 
\begin{align*}
\left[e^k,\left[e^j,e_j\right]\right]_- &= -\left[\left[e^k,e_j\right]_-,e^j\right]_-+\left(\mathcal{D}\left\langle e^k,\left[e^j,e_j\right]\right\rangle\right)_-,\\
\left[e_k,\left[e^j,e_j\right]\right]_- &= \left[\left[e_k,e^j\right]_-,e_j\right]_-+\left(\mathcal{D}\left\langle e_k,\left[e^j,e_j\right]\right\rangle\right)_-.
\end{align*}
Hence, using the Courant algebroid axioms, \eqref{eq:affa}, \eqref{eq:cuasicom}, \eqref{eq:abfabf}, \eqref{eq:afbafb} and Lemma \ref{lem:applem},
\begin{equation*}
\begin{split}
V & = -3\left(:e^k\left[\left[e^j,e_j\right]_-,e_k\right]_-:+:e_k\left[\left[e^j,e_j\right]_-,e^k\right]_-:\right)\\
&+:\left(\left\langle\left[\left[e_k,e^k\right],e^i\right],e^j\right\rangle +\left\langle\left[e_k,\left[e^i,e^j\right]_\ell\right],e^k\right\rangle +\left\langle\left[\left[e^i,e^j\right]_\ell,e_k\right],e^k\right\rangle\right):e_je_i::\\
&-:\left(\left\langle\left[\left[e^k,e_k\right],e_i\right],e_j\right\rangle + \left\langle\left[e^k,\left[e_i,e_j\right]_{\overline{\ell}}\right],e_k\right\rangle +\left\langle\left[\left[e_i,e_j\right]_{\overline{\ell}},e^k\right],e_k\right\rangle\right):e^je^i::.
\\&
= -3\left(:e^k\left[\left[e^j,e_j\right]_-,e_k\right]_-:+:e_k\left[\left[e^j,e_j\right]_-,e^k\right]_-:\right)+:F^{ij}:e_je_i::-:F_{ij}:e^je^i::.
\end{split}
\end{equation*}
where we have used Lemma~\ref{lem:rewriting-Fij} to obtain the second equality. 
The required formula for $\left[{H'}_\Lambda{J_0}\right]$ follows from the above expression for $V$.
\end{proof}

% We are now ready to prove Proposition \ref{teo:mainresult2}.

\begin{proof}[Proof of Proposition \ref{teo:mainresult2}]
This follows from Proposition \ref{lem:lempasofinal}. In more detail, we start applying sesquilinearity, so
\begin{align*}
\left[J_\Lambda J\right] & = \left[{J_0}_\Lambda{J_0}\right]- i\left[{J_0}_\Lambda Su\right] - i\left[Su_\Lambda{J_0}\right] - \left[Su_\Lambda Su\right]\\
& = \left[{J_0}_\Lambda{J_0}\right] + i\left(\chi+S\right)\left[{J_0}_\Lambda u\right] - i\chi\left[u_\Lambda{J_0}\right] + \left(\lambda-\chi S\right)\left[u_\Lambda u\right].
\end{align*}
The first summand of the right-hand side is given by Proposition~\ref{lem:primerpaso}, while Lemma \ref{lem:lempp} for $a=u$ calculates the second and third summands. As for the last one, we have
\begin{equation*}
\begin{split}
\left[u_\Lambda u\right] & = \left[u,u\right]+2\chi\left\langle u,u\right\rangle  = (S+2\chi)\left\langle u,u\right\rangle,
\end{split}
\end{equation*}
since $\mathcal{D}\langle u,u \rangle = 2[u,u]$. Therefore we obtain the required identity, because
\begin{align*}
\left[J_\Lambda J\right] %& = \left[{J_0}_\Lambda{J_0}\right]+i\frac{2}{k}\left(\chi+S\right)\left[{J_0}_\Lambda u\right]-i\frac{2}{k}\chi\left[u_\Lambda{J_0}\right]+\frac{4}{k^2}\left(\lambda-\chi S\right)\left[u_\Lambda u\right]\\
& =-\left(H'- TI_+u_++\frac{S}{2}\left(:\left[u,e^j\right]e_j:+:e^j\left[u,e_j\right]:-2T\left\langle\left.\left[u,e^j\right]\right|e_j\right\rangle\right)\right.\\
& +\left.\frac{\lambda\chi}{3}c+\left(\chi S+\lambda \right)S\left(\left\langle\left.\left[u,e^j\right]\right|e_j\right\rangle-\left\langle u,u\right\rangle\right)\right).
\qedhere
\end{align*}
\end{proof}

We need the following technical result to prove Proposition \ref{teo:mainresult}.

\begin{lemma}\label{lem:uHu}
Assume that $\ell\oplus\overline{\ell}$ satisfies the F-term condition \eqref{eq:Ftermabs} and that the section $u \in \Omega^0(\Pi E)$ satisfies that $\left[u,\cdot\right]=0$. Then
\begin{subequations}\label{eq:lem:uHu}
\begin{align}
\begin{split}\label{eq:lem:uHu.a}
\left[u_\Lambda H'\right] & = -\chi \left(:\left[u_+,e^j\right]_-e_j:-:e^j\left[u_+,e_j\right]_-:+\left(\lambda+T\right)\left\langle\left[I_+u_+,e^j\right],e_j\right\rangle\right)\\
& +\chi\left(:\left\langle\left[u_+,e^j\right],e^k\right\rangle:e_je_k::+:\left\langle\left[u_+,e_j\right],e_k\right\rangle:e^je^k::\right)+\left(\lambda+\chi S\right)u_+,
\end{split}
\\\label{eq:lem:uHu.b}
\begin{split}
\left[{H'}_\Lambda u\right] & = -\left(\chi+S\right)\left(:\left[u_+,e^j\right]_-e_j:-:e^j\left[u_+,e_j\right]_-:-\lambda\left\langle\left[I_+u_+,e^j\right],e_j\right\rangle\right)\\
& +\left(\chi+S\right)\left(:\left\langle\left[u_+,e^j\right],e^k\right\rangle:e_je_k::+:\left\langle\left[u_+,e_j\right],e_k\right\rangle:e^je^k::\right)\\
&+\left(\lambda+\chi S+2T\right)u_+.
\end{split}
\end{align}
\end{subequations}
\end{lemma}

\begin{proof}
Applying the Jacobi identity for the $\Lambda$-bracket and Lemma \ref{lem:primerpaso}, we obtain
\begin{align*}
\left[u_\Lambda{H'}\right] & = \left[u_\Lambda\left(H'+\frac{\gamma\eta}{3}c_0\right)\right]=-\left[u_\Lambda\left[{J_0}_\Gamma{J_0}\right]\right]\\
& = -\left[\left[u_\Lambda{J_0}\right]_{\Lambda+\Gamma}{J_0}\right]-\left[{J_0}_\Gamma\left[u_\Lambda{J_0}\right]\right]\\
& = -\left(\left[\left[u_\Lambda{J_0}\right]_{\Lambda+\Gamma}{J_0}\right]+\left[\left[u_\Lambda{J_0}\right]_{-\nabla-\Gamma}{J_0}\right]\right),
\end{align*}
where for the last identity we have applied skew-symmetry of the $\Lambda$-bracket.
Let $\Omega=(\omega,\xi)$ be another copy of the formal variables $\Lambda=(\lambda,\chi)$ and $\Gamma=(\gamma,\eta)$ (see Section~\ref{ssec:background}). 
Using \eqref{eq:aJ}, since $\left[u,\cdot\right]=0$, we obtain
\begin{align*}
\left[\left[u_\Lambda{J_0}\right]_\Omega{J_0}\right] & = -i\chi\left[{I_+u_+}_\Omega{J_0}\right]\\
& = \frac{\chi}{2}\left(:\left[I_+u_+,e^j\right]e_j:+:e^j\left[I_+u_+,e_j\right]:\right)+\chi\left(\xi I_+u_++\omega\left\langle\left[I_+u_+,e^j\right],e_j\right\rangle\right).
\end{align*}
Hence
\begin{align*}
\left[u_\Lambda{H'}\right] & = -\chi\left(:\left[u_\ell-u_{\overline{\ell}},e^j\right]e_j:-:e^j\left[u_\ell-u_{\overline{\ell}},e_j\right]:+\left(\lambda-T\right)\left\langle\left[I_+u_+,e^j\right],e_j\right\rangle\right)\\
& +\left(\lambda+\chi S\right)u_+\\
& = -\chi\left(:\left[u_+,e^j\right]_-e_j:-:e^j\left[u_+,e_j\right]_-:+\left(\lambda+T\right)\left\langle\left[I_+u_+,e^j\right],e_j\right\rangle\right)\\
& +\chi\left(:\left\langle\left[u_+,e^j\right],e^k\right\rangle:e_je_k::+:\left\langle\left[u_+,e_j\right],e_k\right\rangle:e^je^k::\right)+\left(\lambda+\chi S\right)u_+,
\end{align*}
where last identity follows from \eqref{eq:identech3} for $a = u$. This concludes the proof of the identity~\eqref{eq:lem:uHu.a}. The identity~\eqref{eq:lem:uHu.b} follows now from~\eqref{eq:lem:uHu.a} and skew-symmetry of the $\Lambda$-bracket.
\end{proof}

% We are now ready to prove Proposition \ref{teo:mainresult}.

\begin{proof}[Proof of Proposition \ref{teo:mainresult}] 
The identities \eqref{eq:newiden} and the formula for $c$ follow trivially from $\left[u,\cdot\right]=0$. The fact that $c\in\C$ follows by Remark~\ref{rem:constant-central-charge.1}. The formula for the $\Lambda$-bracket $[J_\Lambda J]$ follows from Proposition~\ref{teo:mainresult2}, since the derivation $S$ vanishes when applied to $c-3n\in\C$. To calculate the $\Lambda$-bracket $[H_\Lambda J]$, we proceed as follows. Applying sesquilinearity, we have
\begin{align*}
\left[H_\Lambda J\right] & = \left[{H'}_\Lambda{J_0}\right]-i\left[{H'}_\Lambda Su\right]-\left[{TI_+u_+}_\Lambda{J_0}\right]+i\left[{TI_+u_+}_\Lambda Su\right]\\
& = \left[{H'}_\Lambda{J_0}\right]-i\left(\chi+S\right)\left[{H'}_\Lambda u\right]+\lambda\left[{I_+u_+}_\Lambda{J_0}\right]-i\lambda\left(\chi+S\right)\left[{I_+u_+}_\Lambda u\right].
\end{align*}
The first summand of the right-hand side is given by Proposition~\ref{lem:lempasofinal}. The second summand is calculated applying~\eqref{eq:lem:uHu.b}. The third summand is calculated applying first Lemma \ref{lem:lempp} with $a=I_+u_+$, and then \eqref{eq:identech3} with $a=u$:
\begin{align*}
\left[{I_+u_+}_\Lambda{J_0}\right] & = \frac{i}{2}\left(:\left[I_+u_+,e^j\right]e_j:+:e^j\left[I_+u_+,e_j\right]:\right)\\
& +i\left(\chi u_++\lambda\left\langle\left[e_+,e^j\right],e_j\right\rangle\right)\\
& = \frac{i}{2}\left(:\left[u_+,e^j\right]_-e_j:-:e^j\left[u_+,e_j\right]_-:\right)\\
& + i(\lambda+T)\left(\left\langle\left[Iu_+,e^j\right],e_j\right\rangle-\chi u_+\right)\\
& +\frac{i}{2}\left(:\left\langle\left[u_+,e^j\right],e^k\right\rangle:e_ke_j::+:\left\langle\left[u_+,e_j\right],e_k\right\rangle:e^ke^j::\right).
\end{align*}
To calculate the last summand, we apply the definition of the $\Lambda$-bracket (part (3) of Proposition~\ref{prop:CDRE}) and then the fact that $\left[u,\cdot\right]=0$:
\begin{align*}
\left[{I_+u_+}_\Lambda u\right] & = \left[I_+u_+,u\right]+2\chi\left\langle I_+u_+,u\right\rangle = -\left[u,u_{\overline{\ell}}\right]+\left[u,u_{\ell}\right]+\mathcal{D}\left(\left\langle u,u_{\overline{\ell}}\right\rangle-\left\langle u,u_{\ell}\right\rangle\right) = 0.
\end{align*}
Using the above identities, \eqref{eq:prop1} and \eqref{eq:prop2} for $a=u$, and applying Lemma \ref{lem:pasointermedio}, we can now obtain the required identity:
\begin{align*}
\left[H_\Lambda J\right] & = \left[{H'}_\Lambda{J_0}\right]-i\left(\chi+S\right)\left[{H'}_\Lambda u\right]+\lambda\left[{I_+u_+}_\Lambda{J_0}\right]\\
& = \left[{H'}_\Lambda{J_0}\right]-i\left(T+\frac{3}{2}\lambda\right)\left(:e_j\left[u_+,e^j\right]_-:+:e^j\left[u_+,e_j\right]_-:\right)\\
& -\left(2\lambda+2T+\chi S\right)Su_+\\
& = \left(2\lambda+2T+\chi S\right)\left(J_0-Si\left(u_++\frac12\left[e^j,e_j\right]_-\right)\right)\\
&+\frac{i}{4}TS\mathcal{D}R+\frac{i}{4}\lambda\left(:F^{ij}:e_je_i::-:F_{ij}:e^je^i::\right)\\
&-i\left(T+\frac32\lambda\right)\left(:e_k\left[u_++\frac12\left[e^j,e_j\right]_-,e^k\right]:+:e^k\left[u_++\frac12\left[e^j,e_j\right]_-,e_k\right]_-:\right).
\mbox{\qedhere}
\end{align*}
\end{proof}

\end{document}